\documentclass[a4paper,10 pt,reqno]{amsart}
\usepackage[latin1]{inputenc}
\usepackage[english]{babel}
\usepackage{amsmath, amsthm, amssymb, amsopn, amsfonts, amstext, stmaryrd, enumerate, color, mathtools, hyperref, mathrsfs, enumitem, url}
\usepackage[left=2.3cm,right=2.3cm,top=2.5cm,bottom=2.4cm]{geometry}
\numberwithin{equation}{section}
\usepackage{dsfont}

\usepackage{pifont}
\hypersetup{
    colorlinks=true,
    linkcolor=blue!75!black,
    citecolor=red!50!black,
    urlcolor=green!50!black,
    linktoc=all
}

\usepackage{colortbl}
\usepackage{tikz}

\theoremstyle{plain}

\numberwithin{equation}{section}
\long\def\salta#1{\relax}

\newcommand{\dint}{\dyle\int}

\newcommand{\re}{\mathbb{R}}
\newcommand{\ren}{\re^N}

\newcommand{\dyle}{\displaystyle}

\newcommand{\io}{\int\limits_\O}

\newcommand{\W}{\mathbb{W}}

\renewcommand{\a }{\alpha }
\renewcommand{\b }{\beta }
\renewcommand{\d }{\delta }
\newcommand{\D }{\Delta }
\newcommand{\e }{\varepsilon }

\newcommand{\g }{\gamma}

\newcommand{\G }{\Gamma }
\renewcommand{\l }{\lambda }

\newcommand{\n }{\nabla }

\newcommand{\s }{\sigma }

\renewcommand{\O }{\Omega }

\newcommand{\Int}{\displaystyle\int}

\newtheorem{Theorem}{Theorem}[section]
\newtheorem{Corollary}[Theorem]{Corollary}
\newtheorem{Definition}[Theorem]{Definition}
\newtheorem{Lemma}[Theorem]{Lemma}
\newtheorem{Proposition}[Theorem]{Proposition}
\newtheorem{remarks}[Theorem]{Remark}
\newcommand{\cqd}{{\unskip\nobreak\hfil\penalty50
        \hskip2em\hbox{}\nobreak\hfil\mbox{\rule{1ex}{1ex} \qquad}
        \parfillskip=0pt \finalhyphendemerits=0\par\medskip}}

\begin{document}
    \title[]{ Global regularity results for the fractional heat equation and application to a class of non-linear KPZ problem}

    \author[B. Abdellaoui, S. Atmani, K. Biroud, E.-H. Laamri ]{Boumediene Abdellaoui, Somia Atmani, Kheiredine  Biroud, El-Haj Laamri}

    \address{\hbox{\parbox{5.7in}{\medskip\noindent{B. Abdellaoui, S. Atmani: Laboratoire d'Analyse Nonlin\'eaire et Math\'ematiques
                    Appliqu\'ees. \hfill \break\indent Universit\'e AbouBekr Belkaid,
                    Tlemcen, \hfill\break\indent Tlemcen 13000, Algeria. }}}}\email{{
            \tt boumediene.abdellaoui@inv.uam.es, somiaatmani@gmail.com}}

    \address{\hbox{\parbox{5.7in}{\medskip\noindent{K. Biroud: Laboratoire d'Analyse Nonlin\'eaire et Math\'ematiques
                    Appliqu\'ees. \hfill \break\indent Ecole Sup\'rieure de Management
                    de Tlemcen . \hfill \break\indent No. 01, Rue Barka Ahmed
                    Bouhannak Imama, \hfill\break\indent Tlemcen 13000, Algeria.
    }}}}\email{{ \tt kh$_{-}$biroud@yahoo.fr}}

    \address{\hbox{\parbox{5.7in}{\medskip\noindent{E.-H. Laamri, Institut Elie Cartan,\\ Universit\'{e} Lorraine,\\ B. P. 239, 54506 Vand{\oe}uvre l\'es
                    Nancy, France. \\[3pt]
                    \em{E-mail addresses: }\\{\tt el-haj.laamri@univ-lorraine.fr}.}}}}


   \date{06/06/2025}

    \begin{abstract}

      In the first part of this   paper, we  prove the global
  regularity,  in an adequate parabolic Bessel-Potential space and then in the
corresponding parabolic fractional Sobolev space,  of  the unique solution  to  following  fractional heat equation $ w_t+(-\Delta)^sw=  h\;;\; w(x,t)=0 \text{ in } \;  (\mathbb{R}^N\setminus\Omega)\times(0,T)\;;\;   w(x,0)=w_0(x)  \; \text{in}\;   \Omega$,  where $\Omega$  is  an open bounded  subset of $\mathbb{R}^N$. The proof is based on a new pointwise estimate on the fractional gradient of the corresponding kernel.   Moreover, we establish the compactness of  $(w_0,h)\mapsto w$. As a majeur application,  in the second part , we establish  existence and regularity  of solutions to a class of   Kardar--Parisi--Zhang equations with fractional diffusion and a nonlocal gradient term. Additionally, several
auxiliary results of independent interest are obtained.
\bigbreak

    \end{abstract}
    \maketitle

    \tableofcontents

   \section{Introduction}

    This work is  structured in two main parts.

  --- The \textbf {\em first part} is devoted to a fine study of the regularity of solutions to the following fractional heat equation
    \begin{equation*}
        (FHE)\qquad  \left\{\
        \begin{array}{lllll}
            w_t+(-\Delta)^sw&=&h(x,t) & \text{in}&\Omega_T=\Omega\times (0,T),\vspace{0.2cm}\\
            w(x,t)&=&0&\text{in} & (\mathbb{R}^N\setminus\Omega)\times(0,T),\vspace{0.2cm}\\
            w(x,0)&=&w_0(x)&  \text{in} & \Omega,
        \end{array}
        \right.
    \end{equation*}
    where $h$ and $w_0$ are measurable functions satisfying some integrability assumptions.  We assume here and in the {\em rest of this paper} that   $0<s<1$  and $\Omega$ is a bounded open subset of $\mathbb{R}^N$ with regular boundary. The operator $(-\Delta)^s$ denotes the fractional Laplacian introduced par M. Riesz in \cite{Riesz}  and  defined for any $\phi\in \mathscr{S}(\mathbb{R}^N)$ by:

    \begin{equation}\label{eq:fraccionarioel2}
        (-\Delta)^{s}\phi(x):=a_{N,s}\mbox{ P.V. }\int_{\mathbb{R}^{N}}{\frac{\phi(x)-\phi(y)}{|x-y|^{N+2s}}\, dy},\; s\in(0,1),
    \end{equation}
    where P.V. stands for the Cauchy principal value and
    \begin{equation}\label{aNS}
    a_{N,s}:=\frac{s 2^{2 s} \Gamma\left(\frac{N+2 s}{2}\right)}{\pi^{\frac{N}{2}} \Gamma(1-s)}
    \end{equation}
    is a normalization constant chosen
    so that
    the following pair of identities :
    $$\lim\limits_{s\to 0^+} (-\Delta)^s\phi=\phi\quad\text{and }\quad \lim\limits_{s\to 1^-} (-\Delta)^s\phi=-\Delta \phi$$
    holds.
    For further details, see, {\it e.g.}, \cite[Proposition 4.4]{dine}
    or \cite[Proposition 2.1]{DaouLaam}.
\medbreak
 Prior to stating our results, we briefly review
some previous works that are closely related to the present study.
\smallbreak
\noindent Existence and uniqueness of a weak solution to $(FHE)$, in the
sense of Definition~\ref{veryweak}, were established
in~\cite{LPPS} for data in $ L^1 $. Moreover, by employing
approximation techniques and suitable test functions, the authors
obtained partial regularity of the solution in certain fractional
parabolic Sobolev spaces $ L^\s(0,T; W^{s,\s}_0(\Omega)) $
(defined in Subsection~\ref{sub11}), with $\s \leq 2$. It is worth
noting, however, that such approximation arguments are not
well-suited for proving higher-order fractional regularity in more
general parabolic fractional spaces. Therefore, a novel approach
is needed to address this limitation.
\smallbreak
\noindent
In \cite{BWZ}, the authors considered the transformation $ u = w \phi $, where $ \phi \in \mathcal{C}^\infty_0(\Omega \times (0,T)) $, and derived local (interior) regularity results in a fractional framework as well as in Besov spaces, depending on the smoothness of the source term $ h $. The key idea in \cite{BWZ} was to analyze the equation satisfied by $ u $ rather than $ w $ itself. However, this approach has a purely local nature and does not capture the influence of the initial condition on the global regularity of the solution $ w $. 
\smallbreak
 Further refined regularity results were established in
\cite{Gubb0},\cite{Gubb1}, and \cite{AGubb} for an extended class
of pseudodifferential operators. The approach adopted in these works
combines the calculus of pseudodifferential operators with
interpolation techniques in Triebel-Lizorkin spaces. Under
suitable hypotheses on the data $(h,w_0)$ -- in particular, by imposing $w_0=0$ or prescribing additional regularity
for  $u_0$--, the authors derive regularity
estimates for solutions of the fractional heat equation.

\medbreak

It is worth noting that in the case $ h = 0 $, the authors of
\cite{FerRos} investigated boundary regularity for solutions in
bounded $ \mathcal{C}^{1,1} $ domains. They showed that if the
initial datum belongs to $ L^2(\Omega) $,  the solution
is H\"older continuous for $ t > 0 $, and satisfies $
\dfrac{u(t,\cdot)}{\delta^s} \in \mathcal{C}^{s -
\varepsilon}(\overline{\Omega}) $ for any $\varepsilon\in (0, s)$,
where $ \delta(x) $ denotes the distance from $ x $ to $
\partial\Omega $.
\smallbreak
 Regarding the elliptic case, the first complete global
(up to the boundary) regularity result was recently established in
\cite{AFTY}. The authors used a representation formula along with
precise estimates on the Green's function to derive optimal
regularity in Bessel-Potential spaces $
\mathbb{L}^{s,p}(\mathbb{R}^N) $ (see Subsection~\ref{sub11} for
the definition). For a comprehensive treatment of regularity
theory in the elliptic setting, we refer the reader to
\cite{AFTY}.
\medbreak
 The {\em first goal } of the present work is to extend these results to the parabolic setting by establishing \textbf{global regularity} for the solution $ w $. Depending on the regularity of the data $ (h, w_0) $, we show that the solution belongs to appropriate parabolic Bessel-Potential spaces $ L^{p}(0,T;\mathbb{L}^{s,p}(\mathbb{R}^N)) $, as well as to the corresponding parabolic fractional Sobolev spaces $ L^{p}(0,T;W^{s,p}(\mathbb{R}^N)) $, defined in Subsection~\ref{sub11}.
\smallbreak

Our methodology extends the recent framework introduced in \cite{AFTY}. Let \( P_\Omega \) denote the heat kernel associated with the homogeneous exterior Dirichlet problem for the fractional heat operator $\partial_t + (-\Delta)^s $. A central component of our analysis is the derivation of pointwise estimates for the fractional gradient of $P_\Omega$. These kernel estimates are then used to establish pointwise bounds for the fractional gradient of the solution to the Dirichlet problem \((\mathrm{FHE})\), even when the forcing term and initial datum are merely integrable with respect to the distance to the boundary. \\
To the best of our knowledge, these gradient estimates are entirely novel. They lead to refined regularity results formulated within the appropriate functional framework.

\smallbreak
\noindent Compared to the elliptic case, addressing time
regularity introduces additional challenges, particularly near $ t
= 0 $ and when $ x $ approaches the boundary of $ \Omega $. As we
will see, another critical difficulty arises as $ t \to 0 $,
especially when the fractional parameter $ s $ is small. In
particular, obtaining uniform integral estimates over $ \Omega_T $
requires imposing an additional constraint on $ s $.

\smallbreak
$\bullet$ We begin this first part with the following main result :

 \begin{Theorem}\label{Thm1.1IN}
           Assume that $\rho \in [s, \max\{1,2s\})$. Then, there exists a positive constant $C$ such that for any $(x,y)\in\Omega\times \Omega $
        with $x\neq y$ and $t>0$,
        we have:
        \begin{enumerate}
            \item[{\em 1.}] If $s\le \frac 12$, then
            \begin{equation*}
                \begin{array}{lll}
                    & &|(-\Delta) ^{\frac{\rho}{2}}P_\O(x,y,t)|\leq\dfrac{C
                        (\frac{\d^s(y)}{\sqrt{t}}\wedge 1)}{
                        (t^{\frac{1}{2s}}+|x-y|)^{N+2s+\rho-1}} \\
                    &\times &
\bigg(\frac{\d^{s-\rho}(x)}{(t^{\frac{1}{2s}}+|x-y|)^{1-s-\rho}}+t^{\frac{2s-1}{2s}}\log\frac{D}{|x-y|}+t^{\frac{s+\rho-1}{2s}}\d^{s-\rho}(x)+
t^{\frac{2s-1}{2s}}|\log(\d(x))|+|x-y|^{2s-1}\bigg).
                \end{array}
            \end{equation*}

             \item[{\em 2.}] If $s>\frac 12$, then
            \begin{equation*}
                \begin{array}{lll}
                    & & |(-\Delta) ^{\frac{\rho}{2}}P_\O(x,y,t)|\leq \dfrac{C
                        (\frac{\d^s(y)}{\sqrt{t}}\wedge 1)}{
                        (t^{\frac{1}{2s}}+|x-y|)^{N+\rho}}\\
                    &\times & \bigg( \frac{\d^{s-\rho}(x)}{
(t^{\frac{1}{2s}}+|x-y|)^{s-\rho}}+\frac{t^{\frac{2s-1}{2s}}}{|x-y|^{2s-1}}+t^{\frac{\rho-s}{2s}}\d^{s-\rho}(x)+
                    \log\frac{D}{|x-y|}+|\log(\d(x))|\bigg),
                \end{array}
            \end{equation*}
        \end{enumerate}
        where $D>>\text{diam}(\O):=\max\{|x-y| \mbox{ ;
        }x,y\in \overline{\O}\}$ and $\delta (x):=\text{dist}(x,\partial\Omega)$.
    \end{Theorem}
    \begin{remarks}
    It should be noted that  the constant $C$  in the previous theorem is  independent of $x,y$ and $t$.
      \end{remarks}

\noindent The proof of Theorem \ref{Thm1.1IN} is presented
in Section \ref{Estimate_heat_Kernel}.
\smallbreak

As a byproduct, integrating the preceding estimate in time yields the following result, recently established in \cite{AFTY}, concerning the fractional regularity of the Green function, at least when $s>\frac{1}{4}$.

    \begin{Corollary}\label{Corol1.3}
        Assume that $s>\frac 14$. Then,  for any $ \rho \in [s, \max\{1,2s\})$,  
        we have $$
        \begin{array}{lll}
|(-\Delta)_x^{\frac{\rho}{2}}\mathcal{G}_s(x,y)|&\leq \dfrac{C}{
|x-y|)^{N-(2s-\rho)}}\left(\dfrac{\d^{s-\rho}(x)}{|x-y|^{s-\rho}}+\log(\frac{D}{|x-y|})
            +|\log(\delta(x))|\right).
        \end{array}
        $$
    \end{Corollary}

$\bullet$  As a significative  application of the preceding pointwise
estimate on $|(-\D)^{\frac{\rho}{2}}P_\Omega|$, and  using the
    representation formula of the solution $w$ to Problem $(FHE)$, we obtain a complete and \textbf{\em global regularity}
    result for $(-\Delta)^{\frac{\rho}{2}} w(.,t)$ for any $\rho \in  [s, \max\{1,2s\})$.
    A key tool in achieving our objective is  an {$L^p$-estimate} for a class of hyper-singular
    integrals, whose full proof is provided in the Appendix.

    \smallbreak
    To keep the analysis as structured as possible, we treat separately the following two cases:
    \begin{itemize}
    \item[\ding{226}] case : $w_0\neq 0$ and $h=0$ ;
    \item[\ding{226}] case : $w_0=0$ and $h\neq0$.
    \end{itemize}
    \smallbreak
   \noindent\ding{226} First case: $w_0\neq 0$ and $h=0$.
  Our first result is  the following.
    \begin{Theorem}\label{Thm1.4}
        Let $ \rho \in [s, \max\{1,2s\})$. Let $w$ be the unique weak
        solution to { Problem  $(FHE)$} with $w_0\in L^\s(\O)$ where
        $\s\ge 1$ and define $\widehat{\s}\le
        \min\{\s,\frac{1}{\rho-s}\}$. Then
        \begin{enumerate}
            \rm  \item[ {1.}]  If $2s+\rho\ge 1$, for all $\eta>0$ small enough and for all
            $p>\widehat{\s}$, we have
            $$
            \begin{array}{lll}
                & &||(-\Delta)^{\frac{\rho}{2}} w(.,t)||_{L^p(\O)}\leq C(\O)
t^{-\frac{N}{2s}(\frac{1}{\widehat{\s}}-\frac{1}{p})-\frac{1}{2}}\\
                &\times & \bigg( t^{-\frac{N}{2s}(1+\eta)(\rho-s)}+
                t^{-\frac{\rho-s}{2s}}+t^{-\frac{\rho+\eta-s}{2s}}
                +t^{-\frac{N\eta}{2sp(1+\eta)}-\frac{\rho-s}{2s}}
                \bigg)||w_0||_{L^{\widehat{\s}}(\O)}.
            \end{array}
            $$
            \rm   \item[ {2.}]  If $2s+\rho<1$, then setting $\s_0<\min\{\widehat{\s},
            \frac{N}{1-2s-\rho}\}$, for all
            $p>\frac{\s_0N}{N-\s_0(1-2s-\rho)}$, we have
            $$
            \begin{array}{lll}
                & &||(-\Delta)^{\frac{\rho}{2}} w(.,t)||_{L^p(\O)}\leq C(\O)
t^{-\frac{N}{2s}(\frac{1}{\s_0}-\frac{1}{p})-\frac{1}{2}}\\
                &\times & \bigg( t^{-\frac{N}{2s}(1+\eta)(\rho-s)}+
                t^{-\frac{\rho+\eta-s}{2s}}+
                t^{-\frac{N\eta}{2sp(1+\eta)}-\frac{\rho-s}{2s}}
+t^{-\frac{\rho-s}{2s}}\bigg)||w_0||_{L^{\widehat{\s}}(\O)}.
            \end{array}
            $$
        \end{enumerate}
    \end{Theorem}
  Further regularity results in Bessel potential spaces are obtained for fixed  $t>0$.\\
    \medbreak
    \noindent \ding{226} Case: $w_0=0$ and $h\neq 0$.
     \smallbreak
       In order to  establish  the
    regularity of the solution in $\O_T$, we {need} additional
    assumptions  on $s$ and $\rho$. This is a consequence of the
    presence of the singular term $t^{\frac{2s-1}{2s}}$,  which naturally appears
   in the estimates of heat kernel for small times. More
    precisely, we have:

    \begin{Theorem} 
        Assume that $s>\frac 14$ and  fixed $ \rho \in [s, \max\{1,2s\})$
        such that $\rho-s<\min\{\frac{s}{(N+2s)},
        \frac{4s-1}{(N+2s-1)}\}$.

        Let $h\in L^{m}(\Omega_T)$ where $m\ge 1$ and  $w$ 
         be the
        unique solution to Problem {(FHE)}. Then :
        \begin{enumerate}
            \rm   \item[1. ] Case where $2s+\rho\ge 1$.\\
            --  If $1\leq m\le \frac{N+2s}{2s-\rho}$, then  $(-\Delta)^{\frac{\rho}{2}} w\in {L^{r}(\Omega_T)}$  for all $m<r<\overline{\overline{m}}_{s,\rho}:=
            \frac{m(N+2s)}{(N+2s)(m(\rho-s)+1)-ms}$, moreover, we have
            \begin{equation*}
                ||(-\Delta)^{\frac{\rho}{2}} w||_{L^{r}(\Omega_T)}\leq CT^{\frac
12-\frac{N+2s}{2s}(\frac{1}{m}-\frac{1}{r})}(1+
T^{-\frac{\rho+\eta-s}{2s}}+T^{-\frac{\rho-s}{2s}})||h||_{L^m(\Omega_T)}.
            \end{equation*}
            --   If $m>\frac{N+2s}{2s-\rho}$, then
            $(-\Delta)^{\frac{\rho}{2}} w\in {L^{r}(\Omega_T)}$ for all
            $r<\frac{1}{\rho-s}$ and
            $$
            ||(-\Delta)   ^{\frac{\rho}{2}} w||_ {L^{r}(\Omega_T)}\le
            CT^{\frac
                12-\frac{N+2s}{2sm}}(1+
T^{-\frac{\rho+\eta-s}{2s}}+T^{-\frac{\rho-s}{2s}})||h||_{L^m(\Omega_T)}.
            $$
            \rm   \item[2.]  Case $2s+\rho<1$.\\
         -- If $m\le \frac{N+2s}{4s-1}$, then $(-\Delta)^{\frac{\rho}{2}} w\in {L^{r}(\Omega_T)}$ for all
$m<r<\frac{m(N+2s)}{(N+2s)(m(\rho-s)+1)-m(3s+\rho-1)}$ and
            \begin{equation*}
                \begin{array}{lll}
                    ||(-\Delta)^{\frac{\rho}{2}} w||_{L^{r}(\Omega_T)}&\leq &
C(T^{\frac{1-2s-\rho}{2s}}+1)T^{-\frac{N+2s}{2s}(\frac{1}{m}-\frac{1}{r})}\\
                    &\times & \bigg(T^{\frac 12} + T^{\frac{4s-1-\eta}{2s}}+
                    T^{\frac{3s+\rho-1}{2s}} + T^{\frac{4s-1}{2s}}
                    +T^{\frac{2s-\rho}{2s}}\bigg) ||h||_{L^m(\Omega_T)}.
                \end{array}
            \end{equation*}
            -- If $m>\frac{N+2s}{4s-1}$, then for $(-\Delta)^{\frac{\rho}{2}} w\in {L^{r}(\Omega_T)}$  for all
            $r<\frac{1}{\rho-s}$, and for $\eta>0$ small enough, we have
            \begin{equation*}
                \begin{array}{lll}
                    ||(-\Delta)^{\frac{\rho}{2}} w||_{L^{r}(\Omega_T)}&\leq &
C(T^{\frac{1-2s-\rho}{2s}}+1)T^{\frac{1}{r}-\frac{N+2s}{2sm}}\\
                    &\times & \bigg(T^{\frac 12} + T^{\frac{4s-1-\eta}{2s}}+
                    T^{\frac{3s+\rho-1}{2s}} + T^{\frac{4s-1}{2s}}
                    +T^{\frac{2s-\rho}{2s}}\bigg) ||h||_{L^m(\Omega_T)}.
                \end{array}
            \end{equation*}
        \end{enumerate}
        In any case, we have
        \begin{equation*}
            ||(-\Delta)^{\frac{\rho}{2}} w||_{L^{r}(\Omega_T)}\leq C(\O,
            T)||h||_{L^m(\Omega_T)}.
        \end{equation*}
    \end{Theorem}

We refer to Section
\ref{Regularity_heat_Equation_Section} for a complete proof and
other related results.$\square$

    \smallbreak
    \smallbreak
\medbreak
  ---  In the \textbf{\em second part}, we apply the results obtained in the first part to analyze
  a class of Kardar-Parisi-Zhang equations with fractional gradients. More specifically, we consider the following problem :
   \begin{equation*}
    (KPZ_f)\qquad
    \left\{
    \begin{array}{llll}
        u_t + (-\Delta)^s u &=& \big|(-\Delta)^{\frac{s}{2}} u\big|^q + f & \text{in } \Omega_T = \Omega \times (0,T), \vspace{0.2cm} \\
        u(x,t) &=& 0 & \text{in } (\mathbb{R}^N \setminus \Omega) \times (0,T), \vspace{0.2cm} \\
        u(x,0) &=& u_0(x) & \text{in } \Omega.
    \end{array}
    \right.
\end{equation*}
where $q\geq 1$, $f$ and $u_0$ are measurable functions satisfying
some integrability assumptions.
\vskip2mm
\noindent  It is worth recalling that the fractional gradient, or the so-called ``\textbf{the
half-s
        Laplacian}'',  is given by the expression:
    \begin{equation*}
(-\Delta)^{\frac{s}{2}}\phi(x)=\int_{\mathbb{R}^N}\dfrac{\phi(x)-\phi(y)}{\rvert
            x-y\rvert^{N+ s}}dy.
    \end{equation*}
\vskip2mm

{ The Kardar--Parisi--Zhang $(KPZ_f)$  problem  with fractional
diffusion is a prominent and actively studied model at the
intersection of stochastic partial differential equations (SPDEs),
statistical physics, and nonlocal analysis. It generalizes the
classical $(KPZ_f) $ equation by replacing the standard Laplacian
(which models ordinary diffusion) with a fractional Laplacian,
thereby allowing the description of anomalous transport and
long-range interactions in interface growth phenomena. This type
of equations arises in various fields such as turbulent flows,
porous media dynamics, mathematical finance, and biological
transport processes. For an in-depth discussion and further
references, we refer the reader to the seminal
works~\cite{KPZ,Ka,KT,H} and the interesting monograph~\cite{W1}.
}

\smallbreak

\noindent

{
  As far as we know, Problem $(KPZ_f)$, with
fractional gradient, has not yet been investigated in the
literature, and addressing it represents the {\em second goal} of
the present work. Under suitable integrability assumptions on $ f
$ and $ u_0 $, we establish the existence of a  solution to this
Problem,  in the sense of Definition~
\ref{weak_kpz_solution}.

 \smallbreak

\noindent

{
 It is worth recalling that the case where the nonlocal term $|(-\Delta)^{\frac{s}{2}} u|^q$ is replaced by a local gradient term has been studied in~\cite{CV} and~\cite{AB01}, where the existence of solution was obtained under the condition $ q < \dfrac{s}{1 - s} $, provided the data satisfy appropriate assumptions. We also note that the elliptic counterpart of this problem was recently addressed in~\cite{AFTY2}, and further analyzed in~\cite{YBFA} under additional smallness conditions on the data.
}

\smallbreak

\noindent
 {\em Our aim is to extend this result to the largest
possible class of data $(f,u_0)$ and exponents $q$.}
\smallbreak
 \noindent
 The existence results are proved by combining the regularity estimates obtained in the first part of this paper with a
 suitable application of the Schauder fixed-point Theorem.\\
 In order to do this, an important tool will be a compactness
result, in a suitable fractional space, for the operator $\Phi$
defined on $L^1(\Omega_T) \times L^1(\Omega)$ by $\Phi(h, w_0) =
w$, where $w$ denotes the unique solution to $(FHE)$.
%
%
Taking into account the non-local character of the fractional gradient, proving the compactness of the operator $\Phi$
(in a parabolic Bessel potential space), especially for $s \leq \frac{1}{2}$, is a delicate matter. It requires a fine analysis that combines
various a priori estimates with a new representation of elements in the Bessel potential space, as stated in Theorem \ref{GUT}.\\
This compactness result appears to be new and, to the best of our
knowledge, may prove useful in other related problems.\\
A detailed proof is given in Section
\ref{Compactness_Section}.

\medbreak

Finally, related to Problem $(KPZ_f)$, for the
sake of clarity, we shall present our main results by
distinguishing between the following two prototypical cases:
\begin{itemize}
    \item[\ding{226}]  $f\neq 0$ and $u_0=0$ ;
    \item[\ding{226}]  $f=0$ and $u_0\neq0$.\\
    \end{itemize}

Our results can be stated as follows.

\smallbreak
   \noindent\ding{227} First case: $f\neq 0$ and $u_0=0$.
  \begin{Theorem}\label{into6}
        Suppose $f\in L^m(\O_T)$ with $1\leq m$ and $s\in (\frac 14, 1)$. Assume that one of the following conditions hold

        \begin{enumerate}
            \item[ {\em (i)}] $s\in (\frac 13, 1)$ and
            \begin{equation*}
                \left\{\begin{array}{lll}
                    1\leq m\leq \frac{N+2s}{s} \vspace{0.2cm}\mbox{  and } 1<q \leq
                    \frac{N+2s}{N+2s-ms},\\
                    \mbox{  or  }\\
                    m>\frac{N+2s}{s}\vspace{0.4cm}\mbox{  and }q \in (1, +\infty).
                \end{array}
                \right.
            \end{equation*}

            \item[ {\em (ii)}]  $s\in (\frac 14,\frac 13]$ and
            \begin{equation*}
                \left\{\begin{array}{lll}
                    1\leq m\leq \frac{N+2s}{4s-1}\vspace{0.4cm} \mbox{  and } 1<q \leq
                    \frac{N+2s}{N+2s-m(4s-1)},\\
                    \mbox{  or  }\\
                    m>\frac{N+2s}{4s-1}\vspace{0.6cm} \mbox{ and  } q\in(1, +\infty).
                \end{array}
                \right.
            \end{equation*}
        \end{enumerate}
        Then, there exists $0<T^*<T$ such that Problem $(KPZ_f)$
        has a solution $u\in L^{\gamma}(0,T^*; \mathbb{L
        }^{s,\gamma}_0(\Omega))$ with
        \begin{enumerate}
           \item[ {\em 1.}]  $1<\gamma< \overline{m}:=\frac{m(N+2s)}{N+2s-ms}$ if $1\leq
            m<\frac{N+2s}{s}$, and $1<\g<+\infty$ if $
            m\geq\frac{N+2s}{s}$, if $s\in (\frac 13,1)$, \\
            or
            \item[ {\em 2.}] $1<\gamma< \overline{m}:=\frac{m(N+2s)}{N+2s-m(4s-1)}$ if $1\leq
            m<\frac{N+2s}{4s-1}$, and $1<\g<+\infty$ if $
            m\geq\frac{N+2s}{4s-1}$, if $s\in (\frac 14,\frac 13]$.
        \end{enumerate}
    \end{Theorem}
 \begin{remarks}
  As for the unexpected restriction $s>\frac{1}{4}$, it is not clear whether it is only technical,
or due to deeper phenomena. Nevertheless, we shall see that it naturally appears in the proof, and we suggest some possible reasons. 
\end{remarks}

   \noindent\ding{227} Second case: $u_0\neq 0$ and $f=0$.

    \begin{Theorem}\label{into7}
        Suppose $u_0\in L^{\sigma}(\O)$ with $1\le \s<\frac{N}{(1-3s)_+}$. Assume that $1\leq q <\frac{N+ 2s\s}{N+\s s}$. Then, there exists $T^*>0$
        such that Problem $(KPZ_f)$ has a weak solution $u$ such that :
        \begin{enumerate}
           \item[ {\em 1.}] if $s>\frac 13$, then  $u \in L^{\theta}(0,T^*;
            \mathbb{L}^{s,\theta_1}_0(\Omega))$ for all
            $\theta<\min\left\{\frac{q(N+\s s)}{q(N+\s s)-\s
                s},\frac{\s(N+2s)}{N+\s
                s}\right\}$.
           \item[ {\em 2.}] if $s\in (\frac 14,\frac 13]$, then $u\in L^{\theta}(0,T^*; \mathbb{L}^{s,\theta}_0(\Omega))$
            with $\theta<\min\left\{\frac{q(N+\s s)}{q(N+\s s)-\s(4
                s-1)},\frac{\s(N+2s)}{N+\s
                s}\right\}$. $\square$
        \end{enumerate}
    \end{Theorem}

    \smallbreak
    \begin{remarks}

For Problem $(KPZ_f)$ with $u_0\neq 0$ and $f=0$, it is possible
to bypass the usual restriction $s>\frac 14$  by working in the
functional space $\mathbb{X}^{s,p,\g}_0(\O\times (0,\infty))$
defined by $$
\mathbb{X}^{s,p,\g}_0(\O\times
(0,\infty))=\{\phi(t,.)\in
        \mathbb{L}^{s,p}_0(\O)\:\:  \mbox{ for {\textit{a.e.}} }t\in (0,T) \mbox{ and
        }\sup_{t\in
            (0,T)}t^\g||\phi(.,t)||_{\mathbb{L}^{s,p}_0(\O)}<\infty\},
            $$
where $1\le p<\infty, \g>0$ choosing adequately. A detailed
analysis of this case, along with related questions, will be
presented in a forthcoming work.
\end{remarks}
The proof of Theorems \ref{into6} and \ref{into7}
will be given in Section \ref{Application_Section}.

        \medbreak

  The remainder of this paper is organized as follows. In Section~\ref{Preliminaries_Section}, we gather several preliminary tools and inequalities that will be used throughout the paper. This section is divided into three subsections. In Subsection~\ref{sub11}, we introduce the main functional spaces relevant to our study, including fractional Sobolev spaces and Bessel potential spaces, and discuss some of their key properties. Subsection~\ref{sub12} presents a number of technical lemmas and useful inequalities. Finally, in Subsection~\ref{sub13}, we review known regularity results for the (unique) weak solution to Problem~$(\mathrm{FHE})$.
Section~\ref{Estimate_heat_Kernel} is devoted to establishing the main
pointwise estimate on the ``$\rho$-Laplacian'' of $P_\Omega$. The
proof relies on integral functional estimates and fine properties
of the heat kernel.
Once these estimates have been established, we derive the
main regularity results for  the (unique) solution to the fractional heat
equation $(FHE)$, using hyper-singular integrals. For clarity of
the presentation, we consider two cases separately:  (i) $h = 0$
and $w_0 \neq 0$ and (ii)  $h\neq 0$ and $w_0=0$.
It is worth noting that the resulting regularity is global in
nature. These results are presented in
Section~\ref{Regularity_heat_Equation_Section}.
As a significative  application of the regularity theory developed
in Section~\ref{Regularity_heat_Equation_Section}, we establish
the compactness of the operator $ \Phi: L^1(\Omega_T) \times
L^1(\Omega) \to L^q(0,T; \mathbb{L}^{s,q}_0(\Omega)), $ for $q <
\widehat{\kappa}_{s,\rho}$, where $\widehat{\kappa}_{s,\rho} > 1$
is defined in~\eqref{kapp}. The proof is given in
Section~\ref{Compactness_Section}, where we refine the previous
regularity estimates and apply a new nonlocal characterization of
the Bessel-potential space established in~\cite{GU}, using
suitable Marcinkiewicz space.
   Building on the compactness result established in the previous section, we address a class of Kardar--Parisi--Zhang problems involving a nonlocal gradient term.
   Specifically, in Section~\ref{Application_Section}, we prove the existence of a solution to Problem $(KPZ_f)$, treating separately the cases:
   (i)  $f\neq 0$ and $u_0=0$ and  (ii) $f=0$ and $u_0\neq 0$.
    Some extensions and perspectives are presented in Section~\ref{open}. Finally, the paper concludes with an appendix containing the proofs of some technical results used throughout the paper


\section{Preliminaries and Functional Settings. }\label{Preliminaries_Section}

    In this preliminary section, we present the functional framework and some essential tools that will be used throughout the paper.
    The first subsection  is devoted to a review of relevant function spaces, including fractional Sobolev spaces and Bessel potential spaces,
    which provide the appropriate setting for our analysis. In the second subsection, we recall  useful functional inequalities and  several technical
    lemmas that will play a crucial role in establishing our main regularity result. In the third subsection, we recall known regularity results
    for the (unique) weak solution to Problem~$(\mathrm{FHE})$.

    \subsection{Some {useful} functional spaces and their properties.}{}\label{sub11}
    \
    \

     Let $s\in (0,1)$,  $p\ge 1$ and  $U$ be an open subset of $\ren$.

  \subsubsection* {$\bullet$ \underline{Fractional Sobolev
            Spaces}}
            \
            \
            \

      \noindent ---  The fractional Sobolev space $W^{s,p}(U)$ is defined
    as follows
    {\[ W^{s,p}(U):= \left\{ \phi \in L^p(U)\,\, ;\;\;
        \iint\limits_{U\times U}
        \frac{|\phi(x)-\phi(y)|^p}{|x-y|^{N+sp}}dx dy < \infty \right\}.
        \]}
        and it is a
     Banach space  endowed   with the norm
{    {   \[ \|\phi\|_{W^{s,p}(U)} := \left(
        \|\phi\|_{L^p(U)}^p + \iint\limits_{U\times U}
        \frac{|\phi(x)-\phi(y)|^p}{|x-y|^{N+sp}} dx dy
        \right)^{\frac{1}{p}},.\]}}
 We refer to \cite{dine,DaouLaam} and their references for more
    properties of the fractional Sobolev spaces $W^{s,p}(U)$.
   \medbreak
   \noindent ---  When $ U \subsetneq \mathbb{R}^N$, the space $\W_0^{s,p}(U)$ is defined as
    \[ \W_0^{s,p}(U):= \bigg\{\phi\in W^{s,p}(\mathbb{R}^N)
    \,\, ;\,\, \phi= 0 \textup{ in } \mathbb{R}^N\setminus
    U\bigg\}. \]
    If in addition $ U $ is bounded, then $\W_0^{s,p}(U)$ is a Banach
    space endowed with the norm
    \[ \|\phi\|_{\W^{s,p}_0(U)} := \left( \iint\limits_{D_{U}} \frac{|\phi(x)-\phi(y)|^p}{|x-y|^{N+sp}} dx dy \right)^{1/p},\]
    with $D_{U} := (\mathbb{R}^N \times \mathbb{R}^N) \setminus
    (\mathcal{C}U \times \mathcal{C}U).$
   \medbreak
  \noindent ---  If $p=2$, then
  \begin{itemize}
  \item[--]
  {$W^{s,2}(\mathbb{R}^N)=H^s(\mathbb{R}^N)$}  is a Hilbert space.  Moreover,     for $\phi,\psi\in H^{s}(\mathbb R^{N}) $, we have
{    \begin{equation*}
<(-\Delta)^s\phi,\psi>=\dfrac{a_{N,s}}{2}\iint\limits_{\mathbb{R}^{2N}}\dfrac{\left(\phi(x)-\phi(y)\right)\left(\psi(x)-\psi(y)\right)}{||
            x-y||^{N+2s}}dxdy,
    \end{equation*}}
    where the constant $ a_{N,s} $  was introduced via the formula \eqref{aNS}.\\
      \item[--] $\mathbb{H}_0^s(U) : = \mathbb{H}_0^{s,2}(U)$. Furthermore, we denote by $\mathbb{H}^{-s}(U)$ the dual space of $\mathbb{H}_0^s(U)$.
      \end{itemize}

   \medbreak
    \noindent --- In the same way, we define the parabolic fractional space
    $L^{p}(0,T;\, \W^{s,p}_0(U))$ as the set of functions $\phi
    \in L^{p}(U\times(0,T)$ such that
    $$|| \phi||_{L^{p}(0,T;\,\W_0 ^{s,p}(U))}<\infty,$$
    with
$$||\phi||_{L^{p}(0,T;\,\W_0^{s,p}(U))}=\left(\int_{0}^{T}\iint\limits_{D_U}\frac{|\phi(x,t)-\phi(y,t)|^p}{|x-y|^{N+sp}} dx dy dt \right)^{\frac{1}{p}}.$$
    \vspace{0.2cm}
    \subsubsection*{$\bullet$ {\underline{ Marcinkiewicz space $\mathcal{M}^p(U)$} }}
    \
    \

    We denote by {$\mathcal{M}^p(U)$}, the
        Marcinkiewicz space, defined as the set of measurable functions $\phi$
        defined on {$U$} such that
        \begin{equation}\label{TRT}
     \bigg|\bigg\{x\in U: |\phi(x)| >k\bigg\}\bigg|\le \frac{C}{k^p}.
        \end{equation}

        If we set $$||\phi||_{\mathcal{M}^p(U)}=\inf\bigg\{C\ge 0 \mbox{ such
            that }\eqref{TRT} \mbox{  holds  }\bigg\},$$ then
        $||.||_{\mathcal{M}^p(U)}$ is a semi-norm.$\square$

    \subsubsection*{ $\bullet$ {\underline{Bessel Potential Spaces $\mathbb{L}^{s,p}(\mathbb{R}^N)$ }}}{}
    \
    \
    \

    Before proceeding with the definition of Bessel potential spaces, we recall two important nonlocal analogues of the classical gradient that will be relevant in our analysis. For any test function $ \phi \in \mathcal{C}_0^\infty(\mathbb{R}^N) $ and for $ x \in \mathbb{R}^N $, we define:

\begin{equation}\label{frac-g}
    \nabla^s \phi(x) := \int_{\mathbb{R}^N}
    \frac{\phi(x)-\phi(y)}{|x-y|^{s}} \cdot \frac{x - y}{|x - y|} \cdot \frac{dy}{|x - y|^N}, \quad s \in (0, 1),
\end{equation}

and the scalar-valued quantity

\begin{equation} \label{nonlocal gradient}
    \mathbb{D}_s (\phi)(x) := \left( \frac{a_{N,s}}{2}
    \int_{\mathbb{R}^N} \frac{|\phi(x)-\phi(y)|^2}{|x-y|^{N+2s}} \, dy
    \right)^{\frac{1}{2}}.
\end{equation}
 These objects provide different ways of capturing
nonlocal variations of smooth functions. For further details, we
refer the reader to \cite{S_S_2015, Ponce_Book_2015}.
\smallbreak

It is known (see e.g. \cite{Schikorra, Stein1}) that $
\mathbb{D}_s $ behaves asymptotically like the classical gradient
as $ s \to 1^- $, in the sense that
\begin{equation} \label{limit gradient term}
    \lim_{s \to 1^{-}} (1-s)\,\mathbb{D}_s^2(\phi)(x) = |\nabla \phi(x)|^2, \quad \text{for all } \phi \in \mathcal{C}_0^\infty(\mathbb{R}^N).
\end{equation}

    \medbreak
    Given the nonlocal character of the fractional Laplacian, it is natural to investigate the regularity of $w$ (the weak solution to Problem $(FHE)$, in a functional space that reflects this structural feature.

\medbreak

   \noindent ---  The Bessel potential space
    $\mathbb{L}^{s,p}(\mathbb{R}^N)$ is defined as the completion of
    $\mathcal{C}_0^{\infty}(\mathbb{R}^N) $ with respect to the norm
    $\|\cdot\|_{\mathbb{L}^{s,p}(\mathbb{R}^N)}$ defined by
    $$
    ||\phi||_{\mathbb{L}^{s,p}(\ren)} =
    \|(1-\Delta)^{\frac{s}{2}}\phi\|_{L^p(\mathbb{R}^N)} \quad \textup{
        and } \quad (1-\Delta)^{\frac{s}{2}}\phi = \mathcal{F}^{-1} (
    (1+|\cdot|^2)^{\frac{s}{2}} \mathcal{F} \phi).
    $$
    According to \cite{adams, S_S_2015,Stein1, Stein}, the space
    {   $\mathbb{L}^{s,p}(\mathbb{R}^N)$} can be endowed with the
    equivalent norm {$$
        \|\phi\|_{\mathbb{L}^{s,p}(\mathbb{R}^N)} := \|\phi\|_{L^p(\mathbb{R}^N)} +
        \|(-\Delta)^{\frac{s}{2}}\phi\|_{L^p(\mathbb{R}^N)},
        $$}
    or
    {   $$
        \|\phi\|_{\mathbb{L}^{s,p}(\mathbb{R}^N)} := \|\phi\|_{L^p(\mathbb{R}^N)} +
        \|\n^s \phi\|_{L^p(\mathbb{R}^N)}.
        $$}
        \smallbreak
     \noindent --- If  $p>\frac{2N}{N+2s}$,
    {$\mathbb{L}^{s,p}(\mathbb{R}^N)$} can  also  be defined as the set of
    functions $\phi\in L^p(\mathbb{R}^N)$ such that $ \mathbb{D}_s
    (\phi)\in L^p(\mathbb{R}^N)$, endowed with the equivalent norm

{$$||\phi||_{\mathbb{L}^{s,p}(\mathbb{R}^N)}=\|\phi\|_{L^p(\mathbb{R}^N)}+
        \|\mathbb{D}_s (\phi)\|_{L^p(\mathbb{R}^N)}.
        $$}
    {   We refer} to \cite{Stein} for more details.
     \medbreak
        We can summarize the essential
    properties of {$\mathbb{L}^{s,p}(\mathbb{R}^N)$} in the following
    proposition.

    \begin{Proposition}\label{proper}
        Assume that $1\le p<+\infty$ and $0<s<1$. Then, we have

\begin{enumerate}
        \item[{\em 1.}] 
         {$W^{s,p}(\mathbb{R}^N)\hookrightarrow
            \mathbb{L}^{s,p}(\mathbb{R}^N)$ if $p\in (1,2]$ and
            $\mathbb{L}^{s,p}(\mathbb{R}^N)\hookrightarrow
            W^{s,p}(\mathbb{R}^{N})$ if $p\in [2,+\infty)$.}

         \item[{\em 2.}] For all $0 < \e < s < 1$ and $1 < p < +\infty$,
        {$$
            \mathbb{L}^{s+\e,p}(\mathbb{R}^N)\subset W^{s,p}(\mathbb{R}^N)
            \subset \mathbb{L}^{s-\e,p}(\mathbb{R}^N).
            $$}
           \end{enumerate}
    \end{Proposition}
The first assertion follows from   \cite[Theorem 5 in Chapter V]{Stein}, 
while the second is a consequence of    \cite[Theorem 7.63]{adams}.
%
\medbreak
    According to \cite{GU}  (see    \cite[Theorem 1.2]{GU}), we have the
    next beautiful partial characterization of the {Bessel potential }
    space {$\mathbb{L}^{s,p}(\mathbb{R}^N)$}.
    \begin{Theorem}\label{GUT}
        Assume that {$\phi\in \mathbb{L}^{s,p}(\mathbb{R}^N)$}, then
        {   \begin{equation}\label{GUTF}
\bigg\|\frac{\phi(x)-\phi(y)}{|x-y|^{\frac{N}{p}+s}}\bigg\|_{\mathcal{M}^p(\mathbb{R}^N\times
                    \mathbb{R}^N)}\le C(p,N)
\|(1-\D)^{\frac{s}{2}}\phi\|_{L^p(\mathbb{R}^N)}=||\phi||_{\mathbb{L}^{s,p}(\mathbb{R}^N)}.\square
        \end{equation}}
    \end{Theorem}

      \medbreak

   \subsubsection*{ $\bullet$ {\underline{Bessel Potential Spaces $\mathbb{L}^{s,p}_0(U)$ }}}
    \
    \

    In this subsection, we assume that $U$ is \textbf{\em bounded}. \\
    Since we are working under homogeneous exterior Dirichlet boundary condition, we
        need to
        use the Bessel potential space $\mathbb{L}^{s,p}_0(U)$ defined by
        $$ \mathbb{L}^{s,p}_0(U):= \left\{\phi\in \mathbb{L}^{s,p}(\mathbb{R}^N) \,\,
        ;\,\, \phi= 0 \textup{ in } \mathbb{R}^N\setminus U\right\}.
        $$
    Moreover, we can endow $\mathbb{L}^{s,p}_0(U)$ by
        the following equivalent norms
        $$\|\phi\|_{\mathbb{L}^{s,p}_0(U)}\backsimeq
\|(-\Delta)^{\frac{s}{2}}\phi\|_{L^p(\mathbb{R}^N)}\backsimeq
\|\n^s
        \phi\|_{L^p(\mathbb{R}^N)}.
        $$


         \noindent According to \eqref{GUTF}, we have for $\phi\in
        \mathbb{L}^{s,p}_0(U)$
    \begin{equation}\label{GUTFB}
\bigg\|\frac{\phi(x)-\phi(y)}{|x-y|^{\frac{N}{p}+s}}\bigg\|_{\mathcal{M}^p(\mathbb{R}^N\times
                \mathbb{R}^N)}\le C(p,N,U) ||(-\D)^{\frac{s}{2}}\phi||_{{L}^{p}(\mathbb{
                    R}^N)}.\square
    \end{equation}

       \medbreak

        \subsubsection*{ $\bullet$ {\underline{Parabolic {  Bessel potential } space $L^{p}(0,T;\,
        \mathbb{L}^{s,p}_0(\Omega))$}}}
    \
    \

        We define  the Parabolic {  Bessel potential } space $L^{p}(0,T;\,
        \mathbb{L}^{s,p}_0(U))$ as the set of functions $\phi \in
        L^{p}(U\times (0,T))$ such that
        $$|| \phi||_{L^{p}(0,T;\,\mathbb{L}^{s,p}_0(U))}<\infty,$$
        with
$$||\phi||_{L^{p}(0,T;\,\mathbb{L}^{s,p}_0(U))}=\left(\int_{0}^{T}||\phi(.,t)||_{\mathbb{L}^{s,p}_{0}(U)}^pdt \right)^{\frac{1}{p}}.\square$$

     \medbreak

      \subsection{Some  useful  inequalities and technical lemmas.}{}\label{sub110}
    \
    \

     To make the paper self-contained and to ease the reading, we recall in this subsection some (more or less)  inequalities and classical technical lemmas.

        \subsubsection*{ $\bullet$ {\underline{A nonlocal versions of the Sobolev's inequality }}}\label{sub12}
    \
    \

We now present a nonlocal version of the Sobolev inequality, which
we use  in establishing our regularity estimates
      \begin{Theorem}
        Assume that  $N>ps$. There exists a positive constant
        $S_i=S_i(N,p,s), i=1,2,3$ such that for $\phi\in
        C_0^{\infty}(\mathbb{R}^N)$, we
        have\\
        $$\dyle
        S_1||\phi||_{L^{p^*_s}(\mathbb{R}^N)}\le \left(\iint_{\mathbb{R}^N\times \mathbb{R}^N} \dfrac{|\phi(x)-\phi(y)|^p}{|x-y|^{N+sp}} dx
        dy \right)^{\frac{1}{p}}.$$
        $$ \dyle
        S_2||\phi||_{L^{p^*_s}(\mathbb{R}^N)}\le \|\nabla^s
        \phi\|_{L^p(\mathbb{R}^N)}, $$ \\and \\
        $$ \noindent  \dyle
        S_3||\phi||_{L^{p^*_s}(\mathbb{R}^N)}\le
        \|(-\Delta)^{\frac{s}{2}}\phi\|_{L^p(\mathbb{R}^N)},$$\\
        \noindent where $p^*_s=\frac{pN}{N-ps}$. 
    \end{Theorem}
   See \cite[Theorem 6.5]{dine} for the first inequality and \cite{adams} for more details regarding the two last inequalities.
 .
      \medbreak

        \subsubsection*{ $\bullet$ {\underline{Hardy Inequality}}}
    \
    \

   The following Hardy-type inequality will be employed repeatedly throughout this paper.
   For a detailed proof and further discussion, we refer the reader to \cite{Edmunds2}.
    \begin{Theorem}\label{hardyd}
        Assume that $N \geq 2s$.  {Then, for all $\phi\in \W^{s,2}_0(\O)$,  we
        have}
        \begin{equation}\label{hhardy}
            \int\limits_{\Omega} \frac{\phi^2(x)}{\d^{2s}(x)}dx \leq C(\O)
            \iint\limits_{\ \mathbb{R}^N\times \mathbb{R}^N}\frac{|\phi(x)-\phi(y)|^2}{|x-y|^2}dxdy.
        \end{equation}
    \end{Theorem}

   \medbreak

        \subsubsection*{ $\bullet$ {\underline{Some technical lemmas}}}
    \
    \

   We recall some technical lemmas to be used in the proof of our main regularity results.

       \begin{Lemma}[{\cite[Lemma 2.2]{AFTY}}]\label{Gr}
        Let $N \geq 1$, $R > 0$ and $\alpha, \beta \in (-\infty,N)$. There
        exists $C:= C(N,R, \alpha, \beta) > 0$ such that:
        \begin{itemize}
            \item[$\bullet$] {If} $ N-\alpha - \beta\neq 0$, then
            $$
            \int_{B_R (0)} \frac{dz}{|x-z|^{\alpha} |y-z|^{\beta} }  \leq  C
            \Big(1+ |x-y|^{N-\alpha - \beta} \Big) \quad \textup{ for all }
            x,y \in B_R(0) \textup{ with } x \neq y;
            $$
            \item[$\bullet$] {If} $ N - \alpha - \beta = 0$, then
            $$
            \int_{B_R (0)} \frac{dz}{|x-z|^{\alpha} |y-z|^{\beta} }  \leq  C
            \Big(1+ \big|\ln |x-y|\big| \Big)\le \bigg(1+\log(\frac{D}{|x-y|})\bigg)  \quad \textup{ for all } x,y
            \in B_R(0)\textup{ with } x \neq y, $$
        \end{itemize}
        with $D>>\text{diam\,\,} B_R(0):=\max\{|x-y|\;;\; x,y\in \bar{B}_R(0)\}.$
    \end{Lemma}
    \begin{Lemma} [{\cite[Lemma 2.5]{J-S-W-2020}, \cite[Lemma 2.4]{AFTY}}]\label{Tobias}
        Let $\Omega \subset \mathbb{R}^N$ with $N \geq 2$, be a bounded domain with
        boundary $\partial \Omega$ of class $C^{2}$, $0 < \lambda < 1$ and
        $0 < a < 1$. There exists $C:= C(N, \lambda, a, \Omega) > 0$ such
        that
        $$
        \int_{\Omega} |x-y|^{\lambda-N} \delta^{-a}(y)\, dy  \leq C
        \big(1+ \omega_{a,\lambda}(x) \big) \quad \textup{ for all } x
        \in \Omega,
        $$
        where
        \begin{equation} \label{omega}
            \omega_{a,\lambda} (x) :=
            \begin{cases}
                1 & \mbox{ if } \lambda > a, \\
                |\ln \delta(x)| \qquad &\mbox{ if }  \lambda = a, \\
                \delta^{-(a-\lambda)}(x) &\mbox{ if } \lambda < a.
            \end{cases}
        \end{equation}
    \end{Lemma}
\smallbreak
    \begin{Proposition}[Algebraic inequalities]
        Assume that $a,b\geq 0$ and $\a\geq 1$, then we have
        \begin{equation}\label{first_algebraic inequality}
            (a-b)(a^\alpha-b^\alpha)\geq C\rvert
a^{\frac{\alpha+1}{2}}-b^{\frac{\alpha+1}{2}}\rvert^{2},
        \end{equation}
        and
        \begin{equation}\label{second_algebraic_inequality}
            (a-b)^2(a+b)^{\alpha-1}\leq C(a^{\frac{\alpha+1}{2}}-b^{\frac{\alpha+1}{2}})^2.
        \end{equation}
    \end{Proposition}
    Finally, let us recall the next well known result for singular
    integral, we refer to \cite{Stein} for a complete proof.
    \begin{Theorem}\label{stein1}
        Let $0<\lambda<N$ and  $1\le p<\ell<\infty$ be such that  $\dfrac{1}{\ell}+1=\dfrac{1}{p}+\dfrac{\lambda}{N}$. For $g\in L^p(\mathbb{R}^N)$, we define $$J_\lambda(g)(x)=\int_{\mathbb{R}^N}
        \dfrac{g(y)}{|x-y|^\lambda}dy.$$
        Then, it follows that:
        \begin{itemize}
            \item[$a)$] $J_\lambda$ is well defined in the sense that the integral converges absolutely for  almost all $x\in \mathbb{R}^N$ ; \vspace{0.15cm}
            \item[$b)$] if $p>1$, then $\|J_{\lambda}(g)\|_{L^\ell(\mathbb{R}^N)} \leq c_{p,q} \|g\|_{L^p(\mathbb{R}^N)};$ 
            \item[$c)$] if $p=1$, then $\big|\{x\in \mathbb{R}^N\,| J_\lambda(g)(x)>\sigma\}\big|\le \left(\dfrac{
                A\|g\|_{L^1(\mathbb{R}^N)}}{\sigma}\right)^\ell$. $\square$
        \end{itemize}
    \end{Theorem}

    \subsection{Regularity results for the Fractional Heat Equation (FHE):  a review }\label{sub13}
    \
    \

 For the reader's convenience, let us recall the fractional heat equation presented in the introduction :
    \begin{equation*}\label{eq_linear_0}
        (FHE)\qquad \left\{
        \begin{array}{llll}
            w_t+(-\Delta)^s w&=& h
            & \text{in}\quad\Omega_T,\vspace{0.2cm}\\
            w(x,t)&=&0&\text{in} \quad (\mathbb{R}^N\setminus\Omega)\times(0,T),\vspace{0.2cm}\\
            w(x,0)&=&w_0(x)&\text{in} \quad  \Omega,
        \end{array}
        \right.
    \end{equation*}
    where $h$ and $w_0$ are functions defined in suitable Lebesgue spaces.\\

     In this {subsection},
    we aim to recall some known results about the  regularity of the (unique)
    solution  to  Problem $(FHE)$ that we need in our study. We refer
    to \cite{LPPS}, \cite{BWZ} and \cite{APPS} for more details.

    \

    Let us begin by the following definition of the {``finite energy"}
    solution.
    \begin{Definition}[{\cite[Section5]{LPPS}}]\label{def:energie}
        Assume {that }$(h,w_0)\in L^2(\Omega_T)\times L^2(\Omega)$. We say that $w$ is an energy solution to Problem $(FHE)$
        if $w\in L^{2}(0,T;\mathbb{H}_0^s(\Omega
        ))\cap {\mathcal{C}([0,T]; L^{2}(\Omega))}  $, $w_t\in L^{2}(0,T; \mathbb{H}^{-s}(\Omega))$, and for all $v\in L^{2}(0,T; \mathbb{H}_0^s(\Omega
        ))$,
        {\begin{equation}\label{sol_energie}
                \Int_{0}^{T}<w_t,v> dt+ \dfrac{1}{2}\Int_{0}^{T}\iint\limits_{D_\Omega}\dfrac{\left(w(x,t)-w(y,t)\right)\left(v(x,t)-v(y,t)\right)}{|x-y|^{N+2s}}dxdydt=\iint\limits_{\O_T}h(x,t)v(x,t)dxdt,
        \end{equation}}
        and $w(.,t)\rightarrow w_0(.)$ strongly in $ L^2(\Omega)$, as $t\rightarrow 0$.
    \end{Definition}

    For the existence of finite energy solution to Problem $(FHE)$, we
    refer to \cite{LPPS}.
    \smallbreak

    In  case where the data
    $(h,w_{0})\in L^1(\O_T)\times L^{1}(\Omega)$, we need to precise
    the sense for which solutions are defined. To this end, let us
    begin by defining the set of test functions $\Upsilon$
    such that       \begin{equation}
        \begin{array}{lll}
            \Upsilon:=\left\{\phi: \mathbb{R}^N\times [0,T]\rightarrow
            \mathbb{R}, \; ; \; \phi \,\hbox{ is a solution of  } (P_\varphi)
            \right\},
        \end{array}
    \end{equation} where
    $\varphi \in  L^{\infty}(\O_T)\times \mathcal{C}^{\a,\b}_0(\O_T) $ such that  $\a,\b\in (0,1)$  and \begin{equation*} 
        (P_\varphi)\quad\quad\left\{
        \begin{array}{rclll}
            -\phi_t+(-\D)^s\phi&=& \dyle \varphi  & \text{ in }& \O_{T}, \\ \phi(x,t)&=&0 & \text{ in }&(\mathbb{R}^N\setminus\O) \times (0,T], \\
            \phi(x,T)&=&0& \mbox{  in}&\O,
        \end{array}%
        \right.
    \end{equation*}
    The authors in \cite{FelsKass}, proved the existence of  a regular solution  $\phi \in L^\infty(\O_T)$ to
    Problem $(P_\varphi)$. Moreover,  $ -\partial_t\phi+(-\D)^s\phi= \dyle \varphi$ a.e. in  $\O_T$.
    Notice that if  $\phi \in \Upsilon$, then $\phi \in
    L^{\infty}(\Omega \times(0,T))$ (see \cite{LPPS} for more
    details).
    \medbreak
    Now we can state the definition of weak solution to Problem
    $(FHE)$.
    \begin{Definition}\label{veryweak}  Assume that $(h,w_{0})\in L^1(\O_T)\times L^{1}(\Omega)$.
        We  say that $w\in \mathcal{C}([0,T); {L}^{1}(\O))$ is a weak solution to Problem $(FHE)$, if for all $\phi\in \Upsilon$, we have
        \begin{equation}\label{eq:subsuper}
            \iint\limits_{\O_T}\,w\big(-\phi_t\, +(-\Delta)^{s}\phi\big)\,dx\,dt=\\ \iint\limits_{\O_T}\,h\phi\,dxdt +\int\limits_\Omega{w_0(x)\phi(x,0)\,dx}.
        \end{equation}
    \end{Definition}
    Before giving the  existence {theorem} in this case, let us recall
    first the {definition} of the truncating function. For $k>0$, we
    define the truncating function as
    \begin{equation}
        T_k(\phi):=\left\{
        \begin{array}{lll}
            \phi&\text{if} & \rvert \phi\rvert\leq k,\\ \vspace{0.3cm}
            k\dfrac{\phi}{\rvert  \phi\rvert}& \text{if}&\rvert \phi\rvert>k.
        \end{array}
        \right.
    \end{equation}
    In what follows, the constant $C$ will be a positive constant
    depends only on {$\O$} and the data $N,m,s$ and it is independent
    of $w$, that can changes from a line to other.

    The first regularity result for $w$, the weak solution of Problem
    $(FHE)$, is given in the next {theorem.}

    \begin{Theorem}[{\cite[Theorem 28]{LPPS}}]\label{main-exis}
        Suppose that $(h,w_0)\in L^{1}(\Omega_T)\times L^{1}(\Omega)$.
        Then, Problem $(FHE)$ has a unique weak solution $w$
        such that $w\in \mathcal{C}([0,T];L^{1}(\Omega))\cap L^{\theta}(\Omega_T)$
        for all $\theta \in [1,\frac{N+2s}{N}), \quad
        |(-\Delta)^{\frac{s}{2}}w|\in L^{r}(\Omega_T),$ for all $r\in
        [1,\frac{N+2s}{N+s})$, and $T_k(w)\in L^{2}(0,T;\mathbb{H}_0^{s}(\Omega))$
        for all $k>0$. {Moreover,} $w\in L^{q}(0,T; \W^{s,q}_0(\Omega))$ for
        all   $q\in [1,\frac{N+2s}{N+s})$, and
        \begin{equation}
            \begin{array}{lll}
                ||w||_{\mathcal{C}([0,T]; L^{1}(\Omega))}+||w||_{L^\theta(\Omega_T)}+ ||(-\Delta)^{\frac{s}{2}}w||_{L^{r}(\Omega_T)}+||w||_{L^{q}(0,T; \W^{s,q}_0(\Omega))}\\ \\
                \leq C \left(||h||_{
L^{1}(\Omega_T)}+||w_0||_{L^{1}(\Omega)}\right).
            \end{array}
        \end{equation}
    \end{Theorem}
    For further details, we refer also to \cite{AABP}  and \cite{BWZ}.

    \vspace{0.2cm}

    In the case where $h$ and $w_0$ are more regular, specifically,  $h\in
    L^m(\O_T), w_0\in L^\s(\O)$ with $m,\s>1$, then we get more
    regularity of the solution $w$ which will be given in the next
    {theorems}.

    \begin{Theorem}[{\cite[Theorem 3.11 ]{APPS}}]\label{u_delta_s}
        Assume that $w_0\equiv 0$ and  $h\in L^{m}(\Omega_T)$ with $m\geq 1$. Let
        $w$ be the unique weak solution to Problem $(FHE)$, then there exists a positive constant $C$ such that
        \begin{itemize}
            \item[$i)$] {If} $1\le m\le \frac{N+2s}{s}$, then $\dfrac{w}{\delta^s}\in L^{\theta}(\Omega_T)$ for all $\theta<\frac{m(N+2s)}{N+2s-ms}$.
            Moreover, for $m<\theta<\frac{m(N+2s)}{N+2s-ms}$, we have
            $$\bigg\|\dfrac{w}{\delta^s}\bigg\|_ {L^{\theta}(\Omega_T)}\leq C T^{\frac{N+2s}{2s}(\frac{1}{\theta}-\frac{1}{m})+\frac{1}{2}}||h||_{L^{m}(\Omega_T)}.$$
            \item[$ii)$] {If} {$m>\frac{N+2s}{s}$}, then $\dfrac{w}{\delta^s}\in L^{\theta}(\Omega_T)$ for all $\theta\le \infty$ and
$$\bigg\|\dfrac{w}{\delta^s}\bigg\|_{L^{\infty}(\Omega_T)}\leq C
            T^{\frac 12-\frac{N+2s}{2sm}}||h||_{L^{m}(\Omega_T)}.$$
        \end{itemize}
    \end{Theorem}

If $s\in (\frac 12, 1)$, the local gradient of the solution $w$ to
Problem $(FHE)$ is well-defined and lies in an appropriate
Lebesgue space. More precisely, we have
    \begin{Theorem}[{\cite[Theorem 3.11]{APPS}}]\label{th_grad_u_deltas}
        Suppose that $s\in (\frac 12,1)$. Assume that the {hypotheses} of Theorem \ref{u_delta_s} hold. Let $w$ be the unique weak solution to
        Problem {$(FHE)$}, then there exists a positive constant $C$ such that
        \begin{itemize}
            \item[$i)$] {If} $m<\frac{N+2s}{2s-1}$, then $|\nabla w|\delta^{1-s}\in L^{\theta}(\Omega_T)$ for all $\theta<\frac{m(N+2s)}{N+2s-m(2s-1)}$. Moreover
            $$\bigg\||\nabla w|\delta^{1-s}\bigg\|_ {L^{\theta}(\Omega_T)}\leq CT^{\gamma_2} ||h||_{L^{m}(\Omega_T)};$$
            \item[$ii)$] {If} $m\geq\frac{N+2s}{2s-1}$, then $|\nabla w|\delta^{1-s}\in L^{\theta}(\Omega_T)$ for all $\theta<\infty$ and
            $$\bigg\||\nabla w|\delta^{1-s}\bigg\|_ {L^{\theta}(\Omega_T)}\leq CT^{\gamma_2}||h||_{L^{m}(\Omega_T)},$$
        \end{itemize}
        where $\gamma_2=\frac{N+2s}{2s}(\frac{1}{\theta}-\frac{1}{m})+\frac{2s-1}{2s}$.
    \end{Theorem}

    \begin{Corollary}[{\cite[Corollary 3.13]{APPS}}]\label{first_corrolary}
        By  combining Theorems \ref{u_delta_s} and \ref{th_grad_u_deltas} and  using the fact that $|\nabla \delta|=1$ a.e. in $\Omega$, we deduce that
        if $w$ is the unique
        weak solution to { Problem} $(FHE)$ with $w_0=0$, then $w\delta^{1-s}\in L^{\theta}(0,T; \mathbb{W}^{1,\theta}_0(\Omega))$ for all $\theta$ such that
        $\frac{1}{\theta}>\frac{1}{m}-\frac{2s-1}{N+2s}$ and
        $$||w\delta^{1-s}||_{L^{\theta}(0,T; \mathbb{W}^{1,\theta}_0(\Omega))}\leq C(T^{\gamma_1}+T^{\gamma_2})||h||_{L^{m}(\Omega_T)},$$
        where $\g_1=\frac{N+2s}{2s}(\frac{1}{\theta}-\frac{1}{m})+\frac{1}{2}$ and $\gamma_2=\frac{N+2s}{2s}(\frac{1}{\theta}-\frac{1}{m})+\frac{2s-1}{2s}$. Moreover, if $m< \frac{N+2s}{2s-1}$, then the above estimate holds for all $\theta<\frac{m(N+2s)}{N+2s-m(2s-1)}.$
    \end{Corollary}

    In the case where $h=0$ and $w_0\in L^\s(\O)$ with $\s\ge 1$, we
    have the next regularity result.
    \begin{Theorem}[{\cite[Proposition 3.20]{APPS} }]\label{pro:lp2}
        Suppose that $h\equiv 0$ and $w_0\in L^\s(\O)$.  Let $w$ be  the unique weak solution to Problem $(FHE)$, then
        there exists a positive constant $C$ such that, for all
        $r, \rho\ge \s$ and for all $t>0$, we have
        \begin{equation}\label{sem1}
            ||w(\cdot,t)||_{L^r(\O)}\le Ct^{-\frac{N}{2s}(\frac{1}{\s}-\frac{1}{r})}||w_0||_{L^\s(\O)},
        \end{equation}
        and
        \begin{equation}\label{sem001}
            \Big\|\dfrac{w(\cdot,t)}{\d^s}\Big\|_{L^\rho(\O)}\le Ct^{-\frac{N}{2s}(\frac{1}{\s}-\frac{1}{\rho})-\frac{1}{2}}||w_0||_{L^\s(\O)}.
        \end{equation}
        Moreover, $w\in L^{r}(\O_T)$ for all $r<\frac{\s(N+2s)}{N}$ and $\dfrac{w}{\d^s}\in L^{\rho}(\O_T)$ for all
        $\rho<\frac{\s(N+2s)}{N+\s s}$.
    \end{Theorem}


    \section{Estimate on the fractional gradient of the heat kernel.}\label{Estimate_heat_Kernel}

    In this section, we prove the main pointwise estimate on the half-fractional gradient of $P_\Omega$, the  heat kernel with Dirichlet
    boundary condition.
    This estimate will be the key tool in establishing the main regularity result for $w$, the unique weak solution of $(FHE)$.
    \smallbreak

    We recall that, $w$ to Problem $(FHE)$ is given by
    \begin{equation}\label{repres_ solution}
        w(x,t)= \Int_\Omega w_0(y) P_{\Omega} (x,y,t) dy+\iint_{\Omega_t}h(y,\tau) P_{\Omega}(x,y,t-\tau)dy d\tau,
    \end{equation}
    where $ \Omega_t= \Omega \times (0,t)$.\\
     In the next {lemma}, we recall the most useful properties of the heat kernel of the fractional
    Laplacian. We refer to \cite{APPS, BJak, kernel2, kernel3} for the proofs.
    \begin{Lemma}\label{lemma_P}
        Assume that $s\in (0,1)$, then for all $x,y\in\Omega$ and for all
        $0<t<T$, we have
        \begin{equation}\label{first_prop}
            P_{\Omega}(x,y,t) \backsimeq \left(1\wedge
            \frac{\delta^s(x)}{\sqrt{t}}\right)\times\left(1\wedge
            \frac{\delta^s(y)}{\sqrt{t}}\right)\times\left(t^{\frac{-N}{2s}}\wedge
            \dfrac{t}{|x-y|^{N+2s}}\right)
        \end{equation}
        and
        \begin{equation}\label{second_prop}
            |\nabla_xP_{\Omega}(x,y,t)|\leq
            C(N,s,\O)\left(\dfrac{1}{\delta(x)\wedge
                t^{\frac{1}{2s}}}\right)P_{\Omega}(x,y,t).
        \end{equation}
        Thus
        \begin{equation}\label{lllll}
            |\nabla_x P_{\Omega} (x,y,t)|\leq \left\{\begin{array}{rcl}
                \Big( 1\wedge \dfrac{\delta^s(y)}{\sqrt{t}}\Big)\cdot \dfrac{\sqrt{t}}{(\delta(x))^{1-s}} \cdot \dfrac{C}{(t^{\frac{1}{2s}} +|x-y|)^{(N+2s)}} \qquad \mbox{ if } \quad
                \delta(x)<t^{\frac{1}{2s} },\\
                \\
                C \Big( 1\wedge \dfrac{\delta^s(y)}{\sqrt{t}}\Big)\cdot \dfrac{t^{\frac{2s-1}{2s}}}{(t^{\frac{1}{2s}}+ |x-y|)^{N+2s}}\qquad    \mbox{ if }\quad  \delta(x)\geq t^{\frac{1}{2s}
                }.
            \end{array}
            \right.
        \end{equation}
        Moreover, if $s>\frac 12$, then $|\nabla_xP_{\Omega}|\in
        L^q(\O\times \O\times (0,T))$ for all $q<\dfrac{N+2s}{N+1}$ and
        $$
        \dint_{0}^{T} \iint_{\Omega\times \Omega}|\nabla_x
        P_{\Omega}|^q\, dx\,dy\,dt\le
        C(\O)(T^{\frac{N+2s-q(N+s)}{2s}}+T^{\frac{N+2s-q(N+1)}{2s}}).
        $$
    \end{Lemma}
    Notice that
    \begin{equation*}
        t^{\frac{-N}{2s}}\wedge \dfrac{t}{|x-y|^{N+2s}}\simeq
        \dfrac{t}{(t^{\frac{1}{2s}}+|x-y|)^{N+2s}}.
    \end{equation*}
    Hence
    \begin{equation}\label{third_prop}
        P_{\Omega}(x,y,t)\backsimeq C\left(1\wedge
        \frac{\delta^s(x)}{\sqrt{t}}\right)\times\left(1\wedge
        \frac{\delta^s(y)}{\sqrt{t}}\right)\dfrac{t}{(t^{\frac{1}{2s}}+|x-y|)^{N+2s}}\backsimeq C\left(1\wedge
        \frac{\delta^s(x)\delta^s(y)}{t}\right)\dfrac{t}{(t^{\frac{1}{2s}}+|x-y|)^{N+2s}}.
    \end{equation}
    In the same way, we have
    \begin{equation*}
        \Int_{0}^{+\infty} P_{\Omega}(x,y,t) dt =\mathcal{G}_s(x,y),
    \end{equation*}
    where $\mathcal{G}_s(x,y)$ is the Green function of the fractional
    Laplacian with Dirichlet condition. It is known  that

    \begin{equation}
        \mathcal{G}_s(x,y) \simeq
        \dfrac{1}{|x-y|^{N-2s}}\bigg(\frac{\delta^s(x)}{|x-y|^{s}}\wedge
        1\bigg) \bigg(\frac{\delta^s(y)}{|x-y|^{s}}\wedge 1\bigg).
    \end{equation}
    We refer to  \cite{CZ} for more details.

    \

    Now, we are in position to give the key estimate on the {nonlocal gradient of the } heat
    kernel.

    \begin{Theorem} \label{regulaP}
   
        For any  $ \rho \in [s, \min\{1,2s\})$, there exists a positive constant $C$  such that for {\em a.e.}  $x,y\in\Omega$ with $x\neq y$ and for all $t>0$,
        we have
        \begin{enumerate}
            \item If $s\le \frac 12$, we have
            \begin{equation}\label{gradp12}
                \begin{array}{lll}
                    & &|(-\Delta) ^{\frac{\rho}{2}}P_\O(x,y,t)|\leq\dfrac{C
                        (\frac{\d^s(y)}{\sqrt{t}}\wedge 1)}{
                        (t^{\frac{1}{2s}}+|x-y|)^{N+2s+\rho-1}} \\
                    &\times &
                    \bigg(\frac{\d^{s-\rho}(x)}{(t^{\frac{1}{2s}}+|x-y|)^{1-s-\rho}}+t^{\frac{2s-1}{2s}}\log\frac{D}{|x-y|}+t^{\frac{s+\rho-1}{2s}}\d^{s-\rho}(x)+
                    t^{\frac{2s-1}{2s}}|\log(\d(x))|+|x-y|^{2s-1}\bigg).
                \end{array}
            \end{equation}

            \item If $s>\frac 12$, we have
            \begin{equation}\label{gradp22}
                \begin{array}{lll}
                    & & |(-\Delta) ^{\frac{\rho}{2}}P_\O(x,y,t)|\leq \dfrac{C
                        (\frac{\d^s(y)}{\sqrt{t}}\wedge 1)}{
                        (t^{\frac{1}{2s}}+|x-y|)^{N+\rho}}\\
                    &\times & \bigg( \frac{\d^{s-\rho}(x)}{
                        (t^{\frac{1}{2s}}+|x-y|)^{s-\rho}}+\frac{t^{\frac{2s-1}{2s}}}{|x-y|^{2s-1}}+t^{\frac{\rho-s}{2s}}\d^{s-\rho}(x)+
                    \log\frac{D}{|x-y|}+|\log(\d(x))|\bigg),
                \end{array}
            \end{equation}
        \end{enumerate}
        where $D>>\text{diam}(\O):=\max\{|x-y| \mbox{ ;
        }x,y\in \overline{\O}\}$.
    \end{Theorem}
    \begin{remarks}
   As mentioned in the introduction, the constant $C$ in the previous theorem is  independent of $x,y$ and $t$.
      \end{remarks}

    \begin{proof}
        The proof will be a consequence of a more general estimate on
        $|(-\Delta) ^{\frac{\rho}{2}}P_\O(x,y,t)|$. Fixed $\rho>0$
        such that $s\le \rho<\min\{1,2s\}$. Let $(x,y)\in \O\times \O$ with $x\neq y$ and $t>0$ be fixed. Then, we know that
        \begin{equation}
            \begin{array}{lll}
            {    (-\Delta)_x ^{\frac{\rho}{2}}P_\O(x,y,t)}&=
                \Int_{\mathbb{R}^N} \dfrac{P_\O(x,y,t)-P_\O(z,y,t)}{|x-z|^{N+\rho}}dz\vspace{0.2cm}\\
                &=\Int_{\mathbb{R}^N\setminus\Omega} \dfrac{P_\O(x,y,t)-P_\O(z,y,t)}{|x-z|^{N+\rho}}dz\vspace{0.2cm}+\Int_{\Omega} \dfrac{P_\O(x,y,t)-P_\O(z,y,t)}{|x-z|^{N+\rho}}dz\vspace{0.2cm}\\
                &=I_1(x,y,t)+I_2(x,y,t).
            \end{array}
        \end{equation}
        We start by estimating $I_1$, since $P_\O(z,y,t)\equiv 0$ when $z\in \mathbb{R}^N\setminus\Omega$, then
        \begin{equation}\label{I001}
            \begin{array}{lll}
                I_1(x,y,t)&=\Int_{\mathbb{R}^N\setminus\Omega} \dfrac{P_\O(x,y,t)-P_\O(z,y,t)}{|x-z|^{N+\rho}}dz=\Int_{\mathbb{R}^N\setminus\Omega} \dfrac{P_\O(x,y,t)}{|x-z|^{N+\rho}}dz\\
                &\leq C(\O)\dfrac{P_\O(x,y,t)}{\delta^\rho(x)}.\\
            \end{array}
        \end{equation}

        \

        \

        We deal now with $I_2$ which is more involved. Let $\s\in \re$ be such that $N+\s>\max\{2s,1\}$, then
        \begin{equation*}
            \begin{array}{llll}
                \left(t^{\frac{1}{2s}}+|x-y|\right)^{N+\sigma}{I}_2(x,y,t)&=&\Int_{\Omega} \dfrac{\left(t^{\frac{1}{2s}}+|x-y|\right)^{N+\s}P_\O(x,y,t)-
                    \left(t^{\frac{1}{2s}}+|x-y|\right)^{N+\s}P_\O(z,y,t)}{|x-z|^{N+\rho}}dz,\vspace{0.2cm}\\
                &=&\Int_{\Omega} P_\O(z,y,t)\dfrac{\left(t^{\frac{1}{2s}}+|z-y|\right)^{N+\s}-\left(t^{\frac{1}{2s}}+|x-y|\right)^{N+\s}}{|x-z|^{N+\rho}}dz,\vspace{0.2cm}\\
                &+&\Int_{\Omega}\dfrac{\left(t^{\frac{1}{2s}}+|x-y|\right)^{N+\s}P_\O(x,y,t)-\left(t^{\frac{1}{2s}}+|z-y|\right)^{N+\s}P_\O(z,y,t)}{|x-z|^{N+\rho}}dz,\vspace{0.2cm}\\
                &=&{I}_{21}(x,y,t)+{I}_{22}(x,y,t).
            \end{array}
        \end{equation*}
        We begin with ${I}_{21}$
        \begin{equation*}
            \begin{array}{lll}
                |{I}_{21}(x,y,t)|&=\bigg|\Int_{\Omega}P_\O(z,y,t)\dfrac{\left(t^{\frac{1}{2s}}+|z-y|\right)^{N+\s}-\left(t^{\frac{1}{2s}}+|x-y|\right)^{N+\s}}{|x-z|^{N+\rho}}dz\bigg|\vspace{0.2cm}\\
                &\leq
                \Int_{\Omega}\dfrac{P_\O(z,y,t)}{|x-z|^{N+\rho}}\Int_{0}^{1} \bigg|\langle z-x,\nabla\left((t^{\frac{1}{2s}}+|\tau x+(1-\tau)z-y|)^{N+\s}\right)\rangle\bigg|d\tau
                dz\vspace{0.2cm}\\
                &\leq C  \Int_{\Omega}\Int_{0}^{1}\dfrac{P_\O(z,y,t)}{|x-z|^{N+\rho-1}} \left(t^{\frac{1}{2s}}+|\tau x+(1-\tau)z-y|\right)^{N+\s-1}
                d\tau dz.
            \end{array}
        \end{equation*}
        Notice that, for $0<\tau<1$  and for all $x,y,z\in \Omega$, we have
        \begin{eqnarray*}
            \left(t^{\frac{1}{2s}}+|\tau x +(1-\tau )z-y|\right)^{N+\s-1}\leq& \left(t^{\frac{1}{2s}}+\tau |x-y|+(1-\tau)|z-y|\right)^{N+\s-1}\\
            \le & C \left(t^{\frac{1}{2s}}+|z-y|\right)^{N+\s-1}+C|x-y|^{N+\s-1}.
        \end{eqnarray*}

        Hence
        \begin{equation*}
            \begin{array}{lll}
                |{I}_{21}(x,y,t)|&\leq C\dyle \int_{\Omega}\dfrac{P_\O(z,y,t) \left(t
                    ^{\frac{1}{2s}}+|z-y|\right)^{N+\s-1}}{|x-z|^{N+\rho-1}} dz+C \Int_{\Omega}\dfrac{P_\O(z,y,t) \,|x-y|^{N+\s-1}}{|x-z|^{N+\rho-1}}
                dz\\ \\
                &={I}_{211}(x,y,t)+{I}_{212}(x,y,t).
            \end{array}
        \end{equation*}
        We begin by estimating ${I}_{211}$. Using estimates
        \eqref{first_prop}, \eqref{third_prop} in Lemma \ref{lemma_P} and
        according to  Lemma \ref{Gr}, we obtain that
        \begin{equation}\label{{I211}}
            \begin{array} {lll}
                {I}_{211}(x,y,t)&\leq & C(\frac{\d^s(y)}{\sqrt{t}}\wedge 1)\Int_{\Omega} \dfrac{1}{|x-z|^{N+\rho-1}|\,(t^{\frac{1}{2s}}+|z-y|)^{1-\s}}dz\vspace{0.2cm}\\
                &\leq & C(\frac{\d^s(y)}{\sqrt{t}}\wedge 1)\Int_{\Omega}\dfrac{1}{|x-z|^{N-(1-\rho)}\,|z-y|^{1-\s}}dz\vspace{0.2cm}\vspace{0.2cm}\\
                & \leq& C(\frac{\d^s(y)}{\sqrt{t}}\wedge 1) \times
                \left\{
                \begin{array}{lll}
                    \bigg(1+|\log|x-y||)\bigg)
                    &\mbox{  if  }\s=\rho,\\
                    \bigg(1+|x-y|^{\s-\rho})\bigg) &\mbox{  if  }\s\neq \rho
                \end{array}
                \right.
            \end{array}
        \end{equation}
        We deal now with ${I}_{212}$. We have
        \begin{equation*}
            \begin{array}{lll}
                {I}_{212}(x,y,t) &\le & C
                |x-y|^{N+\s-1}\dyle\Int_{\Omega}\frac{P_\O(z,y,t)}{|x-z|^{N+\rho-1}}dz\\
                & \le & \dyle C(\frac{\d^s(y)}{\sqrt{t}}\wedge 1)|x-y|^{N+\s-1}
                \io\dfrac{t dz}{(t^{\frac{1}{2s}}+|z-y|)^{N+2s}|x-z|^{N+\rho-1}}\\
                & \le & \dyle C(\frac{\d^s(y)}{\sqrt{t}}\wedge 1)|x-y|^{N+\s-1}
                \int_{\ren}\dfrac{P_{\ren}(z,y,t)}{|x-z|^{N+\rho-1}}dz,
            \end{array}
        \end{equation*}
        where $P_{\ren}$ is the fractional heart kernel in $\ren$. Setting
        $$
        W(x,y,t)=\int_{\ren}\dfrac{P_{\ren}(z,y,t)}{|x-z|^{N+\rho-1}}dz,
        $$
        then, $W$ solves the equation
        \begin{equation}\label{CCC}
            \left\{
            \begin{array}{llll}
                W_t(x,y,t)+(-\Delta)_y^s W(x,y,t)&=& 0& \text{in}\quad\ren\times (0,T),\vspace{0.2cm}\\
                W(x,y,0)&=&\dfrac{1}{|x-y|^{N+\rho-1}}&\text{in} \quad
                \ren,
            \end{array}
            \right.
        \end{equation}
        see \cite{BSV} {for more details}. Since $\frac{1}{|x-y|^{N+\rho-1}}$ is a
        supersolution to \eqref{CCC}, then $W(x,y,t)\le
        \frac{1}{|x-y|^{N+\rho-1}}$. Thus
        \begin{equation}\label{I212}
            {I}_{212}(x,y,t)\le C(\frac{\d^s(y)}{\sqrt{t}}\wedge
            1)|x-y|^{\s-\rho}.
        \end{equation}
        Hence,  from \eqref{{I211}} and \eqref{I212},  we get
        \begin{equation}\label{{I}{21}}
            |{I}_{21}(x,y,t)|\leq C(\frac{\d^s(y)}{\sqrt{t}}\wedge
            1)\times
            \left\{
            \begin{array}{lll}
                \bigg(1+|\log|x-y||)\bigg)
                &\mbox{  if  }\s=\rho,\\
                \bigg(1+|x-y|^{\s-\rho})\bigg) &\mbox{  if  }\s\neq \rho
            \end{array}
            \right.
        \end{equation}

        \

        Now, going back to  ${I}_{22}$, we have
        \begin{equation*}
            \begin{array}{lll}
                {I}_{22}(x,y,t)&=
                \Int_{\Omega}\dfrac{\left(t^{\frac{1}{2s}}+|x-y|\right)^{N+\s}P_\O(x,y,t)-\left(t^{\frac{1}{2s}}+|z-y|\right)^{N+\s}P_\O(z,y,t)}{|x-z|^{N+\rho}}dz\vspace{0.2cm}\\
                &=\Int_{\Omega}\dfrac{\Theta_{(y,t)}(x)-\Theta_{(y,t)}(z)}{|x-z|^{N+\rho}}dz,
            \end{array}
        \end{equation*}
        where, for $t>0, y\in \O$ fixed, $\Theta$ is defined by
        $$\Theta_{(y,t)}(x)=
        (t^{\frac{1}{2s}}+|x-y|)^{N+\s}P_\O(x,y,t).
        $$

        Notice that, using \eqref{first_prop} and \eqref{third_prop}, it holds that
        $\Theta_{(y,t)}(x)=0$ if $x\in \mathbb{R}^N\setminus \O$ for a.e. $y\in \O$ and $t>0$. Moreover, we find that
        $$
        \Theta_{(y,t)}(x)\le C\left(1\wedge
        \frac{\delta^s(x)}{\sqrt{t}}\right)\times\left(1\wedge
        \frac{\delta^s(y)}{\sqrt{t}}\right)\times \dfrac{t}{(t^{\frac{1}{2s}}+|x-y|)^{2s-\s}}.
        $$
        Notice that
        $$
        \n_x \Theta_{(y,t)}(x)=\left(t^{\frac{1}{2s}}+|x-y|\right)^{N+\s}\n_x
        P_\O(x,y,t)+P_\O(x,y,t)\n_x \left(t^{\frac{1}{2s}}+|x-y|\right)^{N+\s}.$$
        Then, using now \eqref{first_prop} and  \eqref{second_prop}, it follows that

        \begin{eqnarray*}
            |\n_x \Theta_{(y,t)}(x)|
            &\le & C
            \left(t^{\frac{1}{2s}}+|x-y|\right)^{N+\s}P_\O(x,y,t)\bigg(\frac{1}{\d(x)\wedge
                t^{\frac{1}{2s}}}\bigg) +C
            P_\O(x,y,t)\left(t^{\frac{1}{2s}}+|x-y|\right)^{N+\s-1}.
        \end{eqnarray*}
        Notice that
        \begin{eqnarray*}
            |\n_x \Theta_{(y,t)}(x)|
            &\le & \dfrac{C}{(t^{\frac{1}{2s}}+|x-y|)^{2s-\s}}\bigg(\frac{\sqrt{t}}{(\d(x))^{1-s}}\chi_{\{\d(x)<t^{\frac{1}{2s}}\}}+
            t^{\frac{2s-1}{2s}}\chi_{\{\d(x)\ge t^{\frac{1}{2s}}\}}\bigg)+\frac{C \d^s(x)\d^s(y)}{(t^{\frac{1}{2s}}+|x-y|)^{2s-\s+1}} .
        \end{eqnarray*}
        Thus, for $(y,t)$ fixed, we have  $\Theta_{(y,t)}\in \W^{1,p}_0(\O)$ for all $p<\frac{1}{1-s}$, for a.e.  $y\in \O$ and $t>0$.
        Therefore, using Lemma 3.1 in \cite{AFTY}, we obtain that, for $(y,t)$ fixed and for a.e.  $x,z \in \mathbb{R}^N$,
        $$
        \Theta_{(y,t)}(x)-\Theta_{(y,t)}(z)=\int_0^1 \langle x-z, \nabla \Theta_{(y,t)}(\tau x + (1-\tau)z) \rangle d\tau.
        $$
        Hence, for $t>0$ and for a.e. $x,y\in \O$, we have
        $${I}_{22}(x,y,t)=\io \bigg(\int_0^1
        \langle x-z, \nabla \Theta_{(y,t)}(\tau x + (1-\tau)z) \rangle
        d\tau\bigg)\frac{dz}{|x-z|^{N+\rho}}.
        $$
        Setting $\xi_\tau=\tau x+(1-\tau) z$ and define $I_\O=\{\tau\in
        [0,1]: \xi_\tau\in \O\}$, then
        \begin{equation*}
            \begin{array}{lll}
                |{I}_{22}(x,y,t)|&\leq  \Int_{\Omega}\Int_{0}^{1}\dfrac{|\nabla \Theta_{(y,t)}(\tau x+(1-\tau)z)|}{|x-z|^{N+\rho-1}}d\tau dz\vspace{0.2cm}\\
                &\leq C(\O)\Int_{\Omega}\Int_{0}^{1}\dfrac{(t^{\frac{1}{2s}}+|\tau x+(1-\tau) z-y|)^{N+\s-1}}{|x-z|^{N+\rho-1}}P_\O(\tau x+(1-\tau) z,y,t)d\tau dz\vspace{0.2cm}\\
                &+\Int_{\Omega}\Int_{0}^{1}\dfrac{(t^{\frac{1}{2s}}+|\tau x+(1-\tau) z-y|)^{N+\s}}{|x-z|^{N+\rho-1}} |\nabla P_\O(\tau x+(1-\tau) z,y,t)|d\tau dz\vspace{0.2cm}\\
                &\leq C(\O)\Int_{\Omega}\Int_{I_\O}\dfrac{(t^{\frac{1}{2s}}+|\xi_\tau-y |)^{N+\s-1}}{|x-z|^{N+\rho-1}}P_\O(\xi_\tau,y,t)d\tau dz
                +\Int_{\Omega}\Int_{I_\O}\dfrac{(t^{\frac{1}{2s}}+|\xi_\tau-y|)^{N+\s}}{|x-z|^{N+\rho-1}} |\nabla P_\O(\xi_\tau,y,t)|d\tau dz\vspace{0.2cm}\\
                &={I}_{221}(x,y,t)+{I}_{222}(x,y,t).
            \end{array}
        \end{equation*}
        Since $\O$ is a bounded domain, we get the existence of $R>>1$ such that $\O\subset B_R(0)$ and $\xi_\tau\in B_R(0)$ for all $\tau\in [0,1]$. Hence, using
        \eqref{first_prop} and \eqref{third_prop}, we get
        \begin{equation*}
            \begin{array}{lll}
                {I}_{221}(x,y,t)&\le \dyle\int_{B_R(0)}\bigg(\int_0^1(t^{\frac{1}{2s}}+|\xi_\tau-y|)^{N+\s-1}P_\O(\xi_\tau,y,t)d\tau\bigg) \frac{dz}{|z-x|^{N+\rho-1}}\vspace{0.2cm}\\
                &\leq (\frac{\d^s(y)}{\sqrt{t}}\wedge 1)\dyle\int_0^1\bigg( \Int_{B_R(0)}\dfrac{1}{(t^{\frac{1}{2s}}+|\xi_\tau-y|)^{1-\s} |z-x|^{N+\rho-1}} dz\bigg)d\tau\vspace{0.2cm}\\
                &\leq (\frac{\d^s(y)}{\sqrt{t}}\wedge 1)\dyle\int_0^1\bigg( \Int_{B_R(0)}\dfrac{1}{|\xi_\tau-y|^{1-\s} |z-x|^{N+\rho-1}} dz\bigg)d\tau\vspace{0.2cm}\\
                &\leq (\frac{\d^s(y)}{\sqrt{t}}\wedge 1)\dyle\int_0^1\bigg( \Int_{B_R(0)}\dfrac{1}{|z-(1-\tau)^{-1}(y-\tau x)|^{1-\s} |z-x|^{N+\rho-1}} dz\bigg)\frac{d\tau}{(1-\tau)^{1-\s}}\vspace{0.2cm}\\
                &\leq C(\frac{\d^s(y)}{\sqrt{t}}\wedge 1) \left\{
                \begin{array}{lll}
                    \dyle\int_0^1 \frac{d\tau}{(1-\tau)^{1-\s}}
                    (1+|\ln(\frac{|x-y|}{1-\tau}|)\mbox{  if  }\s=\rho,\\
                    \dyle\int_0^1 \frac{d\tau}{(1-\tau)^{1-\rho}}
                    (1+|x-y|^{\s-\rho}) \mbox{  if  }\s\neq \rho,\\
                \end{array}
                \right.
                \\
                \\
                &\le C(\frac{\d^s(y)}{\sqrt{t}}\wedge 1) \left\{
                \begin{array}{lll}
                    (1+|\ln|x-y||)\mbox{  if  }\s=\rho,\\
                    (1+|x-y|^{\s-\rho}) \mbox{  if  }\s\neq \rho.\\
                \end{array}
                \right.
            \end{array}
        \end{equation*}
        Hence
        \begin{equation}\label{I221}
            I_{221}\le C(\frac{\d^s(y)}{\sqrt{t}}\wedge 1) \left\{
            \begin{array}{lll}
                (1+|\ln|x-y||)\mbox{  if  }\s=\rho,\\
                (1+|x-y|^{\s-\rho}) \mbox{  if  }\s\neq \rho.\\
            \end{array}
            \right.
        \end{equation}

        We estimate now $I_{222}$. Using \eqref{second_prop}, it follows
        that
        \begin{equation*}
            \begin{array}{lll}
                {I}_{222}(x,y,t)& = & \Int_{\Omega}\Int_{I_\O}\dfrac{(t^{\frac{1}{2s}}+|\xi_\tau-y|)^{N+\s}}{|x-z|^{N+\rho-1}}
                |\nabla P_\O(\xi_\tau,y,t)|d\tau dz\vspace{0.2cm}\\ &\le &
                C\Int_{\Omega}\Int_{I_\O}\dfrac{(t^{\frac{1}{2s}}+|\xi_\tau-y|)^{N+\s}}{|x-z|^{N+\rho-1}}\left(\dfrac{1}{\delta(\xi_\tau)\wedge t^{\frac{1}{2s}}}\right)P_\O(\xi_\tau,y,t)  d\tau dz\vspace{0.2cm}\\
                &\leq & C\Int_{\Omega}\Int_{I_\O}\dfrac{(t^{\frac{1}{2s}}+|\xi_\tau-y|)^{N+\s}}{|x-z|^{N+\rho-1}}\dfrac{P_\O(\xi_\tau,y,t)}{\delta(\xi_\tau)} \chi_{\{\delta(\xi_\tau)< t^{\frac{1}{2s}}\}} d\tau dz\vspace{0.2cm}\\
                &+&  C\Int_{\Omega}\Int_{I_\O}\dfrac{(t^{\frac{1}{2s}}+|\xi_\tau-y|)^{N+\s}}{|x-z|^{N+\rho-1}}\dfrac{P_\O(\xi_\tau,y,t)}{ t^{\frac{1}{2s}}}\chi_{\{\delta(\xi_\tau)\geq
                    t^{\frac{1}{2s}}\}}d\tau dz
            \end{array}
        \end{equation*}
        Now, using \eqref{first_prop} and \eqref{third_prop}, we obtain that
        \begin{equation*}\label{{I222}}
            \begin{array}{lll}
                {I}_{222}(x,y,t)& \le & C(\frac{\d^s(y)}{\sqrt{t}}\wedge 1)\Int_{\Omega}\Int_{I_\O}
                \dfrac{1}{|x-z|^{N+\rho-1}}\dfrac{\sqrt{t}}{\delta^{1-s}(\xi_\tau)(t^{\frac{1}{2s}}+|\xi_\tau-y|)^{2s-\s}}
                d\tau dz\vspace{0.2cm}\\ &+ &
                C(\frac{\d^s(y)}{\sqrt{t}}\wedge 1)\Int_{\Omega}\Int_{I_\O}\dfrac{t^{\frac{2s-1}{2s}}}{|x-z|^{N+\rho-1}(t^{\frac{1}{2s}}+|\xi_\tau-y|)^{2s-\s}}
                \chi_{\{\delta(\xi_\tau)\geq t^{\frac{1}{2s}}\}}d\tau dz\vspace{0.2cm}\\
                &= & C(\frac{\d^s(y)}{\sqrt{t}}\wedge 1)\bigg({J}_{1}(x,y,t)+{J}_{2}(x,y,t)\bigg).
            \end{array}
        \end{equation*}

        We begin by estimating $J_1$,{ we consider}  the set $\Omega'=\{z\in \O:
        \xi_\tau=\tau x+(1-\tau)z\in \O.\}$ and define the change of variable $Y(z)=\xi_\tau$,
        then $Y(\O')\subset \O$. {Hence,  we get}
        \begin{equation*}
            \begin{array}{rcll}
                J_1(x,y,t) &=&\Int_0^1\Int_{\Omega'}\dfrac{1}{|x-z|^{N+\rho-1}}
                \dfrac{\sqrt{t}}{\delta^{1-s}(\xi_\tau)(t^{\frac{1}{2s}}+|\xi_\tau-y|)^{2s-\s}} dzd\tau\vspace{0.2cm}\\
                &=&\Int_0^1\Int_{Y(\Omega')}\dfrac{1}{|x-\xi_\tau|^{N+\rho-1}}
                \dfrac{\sqrt{t}}{\delta^{1-s}(\xi_\tau)(t^{\frac{1}{2s}}+|\xi_\tau-y|)^{2s-\s}}d\xi \dfrac{d\tau}{(1-\tau)^{1-\rho}}\vspace{0.2cm}\\
                &\le & C \Int_0^1 \dfrac{d\tau}{(1-\tau)^{1-\rho}}
                \io \dfrac{1}{|x-\xi|^{N+\rho-1}}
                \dfrac{\sqrt{t}}{\delta^{1-s}(\xi)(t^{\frac{1}{2s}}+|\xi-y|)^{2s-\s}} d\xi\\
                &\le & C\Int_{\Omega}\dfrac{1}{|x-\xi|^{N+\rho-1}}\dfrac{\sqrt{t}}{\delta^{1-s}(\xi)(t^{\frac{1}{2s}}+|\xi-y|)^{2s-\s}}d\xi \vspace{0.2cm}\\
                &\le & Ct^{\frac{\sigma-s}{2s}}\Int_{\Omega}\dfrac{d\xi}{|x-\xi|^{N+\rho-1}\delta^{1-s}(\xi)}.
            \end{array}
        \end{equation*}
        Hence,  using Lemma \ref{Tobias}, we deduce that
        \begin{equation}\label{J11}
            J_1(x,y,t)\leq Ct^{\frac{\s-s}{2s}} \left\{
            \begin{array}{lll}
                |\log({\delta(x)})| &\mbox{  if  } \rho=s,\\
                \delta^{s-\rho}(x) &\mbox{  if   } \rho\neq s.
            \end{array}
            \right.
        \end{equation}

        We treat now the term $J_2$. As above, we have
        \begin{equation*}
            \begin{array}{lll}
                J_2(x,y,t)&=&C\Int_{\Omega}\Int_{I_\O}\dfrac{t^{\frac{2s-1}{2s}}}{|x-z|^{N+\rho-1}(t^{\frac{1}{2s}}+|\xi_\tau-y|)^{2s-\s}}
                \chi_{\{\delta(\xi_\tau)\geq t^{\frac{1}{2s}}\}}d\tau dz\vspace{0.2cm}\\
                &\leq&
                C\Int_{\Omega}\dfrac{t^{\frac{2s-1}{2s}}}{|x-\xi|^{N+\rho-1}(t^{\frac{1}{2s}}+|\xi-y|)^{2s-\s}}
                d\xi\vspace{0.2cm}\\
                &\le & C t^{\frac{2s-1}{2s}}\Int_{\Omega}\dfrac{d\xi}{|x-\xi|^{N+\rho-1}(|\xi-y|)^{2s-\s}}.\\
            \end{array}
        \end{equation*}

        Using Lemma \ref{Gr}, il holds that

        \begin{equation*}
            J_2(x,y,t)\leq Ct^{\frac{2s-1}{2s}}\left\{
            \begin{array}{lll}
                1+|\log|x-y|| &\mbox{  if  }& 1+\s=2s+\rho,\\
                1+|x-y|^{1+\s-2s-\rho} &\mbox{  if   }& 1+\s\neq 2s+\rho.\\
            \end{array}
            \right.
        \end{equation*}
        Using the fact that $x,y\in \O$, a bounded domain, we have
        $1+|\log|x-y||\le C(\O)\log\frac{D}{|x-y|}$. Thus, we conclude that
        \begin{equation}\label{J2-sr}
            J_2(x,y,t)\leq Ct^{\frac{2s-1}{2s}}\left\{
            \begin{array}{lll}
                1&\mbox{  if  }& 1+\s>2s+\rho,\\
                \log\frac{D}{|x-y|} &\mbox{  if  }& 1+\s=2s+\rho,\\
                \frac{1}{|x-y|^{2s+\rho-1-\s}} &\mbox{  if   }& 1+\s<2s+\rho.\\
            \end{array}
            \right.
        \end{equation}
        Combining \eqref{J11} and \eqref{J2-sr}, we deduce that:

        $\bullet$ If $\rho=s$, then
        \begin{equation}\label{J2-rho}
            {I}_{222}(x,y,t)\leq
            C(\O)(\frac{\d^s(y)}{\sqrt{t}}\wedge 1)
            \left\{
            \begin{array}{lll}
                t^{\frac{2s-1}{2s}}+ t^{\frac{\s-s}{2s}}|\log(\d(x))|&\mbox{  if  }& 1+\s>3s,\\
                t^{\frac{2s-1}{2s}}\log\frac{D}{|x-y|}+t^{\frac{\s-s}{2s}}|\log(\d(x))| &\mbox{  if  }& 1+\s=3s,\\
                t^{\frac{2s-1}{2s}} \frac{1}{|x-y|^{3s-1-\s}}+t^{\frac{\s-s}{2s}}|\log(\d(x))| &\mbox{  if   }& 1+\s<3s.\\
            \end{array}
            \right.
        \end{equation}

        $\bullet$ If $\rho\neq s$, then
        \begin{equation}\label{J2-s}
            {I}_{222}(x,y,t)\leq
            C(\O)(\frac{\d^s(y)}{\sqrt{t}}\wedge 1)
            \left\{
            \begin{array}{lll}
                t^{\frac{2s-1}{2s}}+ t^{\frac{\s-s}{2s}}\d^{s-\rho}(x)&\mbox{  if  }& 1+\s>2s+\rho,\\
                t^{\frac{2s-1}{2s}}\log\frac{D}{|x-y|}+t^{\frac{\s-s}{2s}}\d^{s-\rho}(x) &\mbox{  if  }& 1+\s=2s+\rho,\\
                t^{\frac{2s-1}{2s}} \frac{1}{|x-y|^{2s+\rho-1-\s}}+t^{\frac{\s-s}{2s}}\d^{s-\rho}(x) &\mbox{  if   }& 1+\s<2s+\rho.\\
            \end{array}
            \right.
        \end{equation}

        Thus, from  \eqref{J2-rho}, \eqref{J2-s} and   \eqref{I221}, we
        obtain that

        $\bullet$ If $\rho=s$ and $\s=\rho$, then
        \begin{equation}\label{TT1}
            |{I}_{22}(x,y,t)|\leq
            C(\O)(\frac{\d^s(y)}{\sqrt{t}}\wedge
            1)\times
            \left\{
            \begin{array}{lll}
                t^{\frac{2s-1}{2s}}+ t^{\frac{\s-s}{2s}}|\log(\d(x))|+\log\frac{D}{|x-y|}&\mbox{  if  }& 2s<1,\\
                (1+t^{\frac{2s-1}{2s}})\log\frac{D}{|x-y|}+t^{\frac{\s-s}{2s}}|\log(\d(x))| &\mbox{  if  }& 2s=1,\\
                {t^{\frac{2s-1}{2s}} \frac{1}{|x-y|^{2s-1}}+|\log(\d(x))| +\log\frac{D}{|x-y|}}&\mbox{  if   }& 2s>1.\\
            \end{array}
            \right.
        \end{equation}

        $\bullet$ If $\rho=s$ and $\s\neq \rho$, then
        \begin{equation}\label{TT2}
            |{I}_{22}(x,y,t)|\leq C(\O)(\frac{\d^s(y)}{\sqrt{t}}\wedge
            1)\times
            \left\{
            \begin{array}{lll}
                1+t^{\frac{2s-1}{2s}}+ t^{\frac{\s-s}{2s}}|\log(\d(x))|+|x-y|^{\s-\rho}&\mbox{  if  }& 3s<1+\s,\\
                1+t^{\frac{2s-1}{2s}}\log\frac{D}{|x-y|}+t^{\frac{\s-s}{2s}}|\log(\d(x))|+|x-y|^{\s-\rho} &\mbox{  if  }& 3s=1+\s,\\
                1+t^{\frac{2s-1}{2s}}{\frac{1}{|x-y|^{3s-1-\s}}}+t^{\frac{\s-s}{2s}}|\log(\d(x))|+|x-y|^{\s-\rho} &\mbox{  if   }& 3s>1+\s.\\
            \end{array}
            \right.
        \end{equation}

        $\bullet$ If $\rho\neq s$ and $\s=\rho$, then
        \begin{equation}\label{TT3}
            |{I}_{22}(x,y,t)|\leq
            C(\O)(\frac{\d^s(y)}{\sqrt{t}}\wedge
            1)\times \left\{
            \begin{array}{lll}
                t^{\frac{2s-1}{2s}}+ t^{\frac{\s-s}{2s}}\d^{s-\rho}(x)+\log\frac{D}{|x-y|}&\mbox{  if  }& 2s<1,\\
                \log\frac{D}{|x-y|}+t^{\frac{\s-s}{2s}}\d^{s-\rho}(x)&\mbox{  if  }& 2s=1,\\
                t^{\frac{2s-1}{2s}} \frac{1}{|x-y|^{2s-1}}+t^{\frac{\s-s}{2s}}\d^{s-\rho}(x)+\log\frac{D}{|x-y|}&\mbox{  if   }& 2s>1.\\
            \end{array}
            \right.
        \end{equation}

        $\bullet$ If $\rho\neq s$ and $\s\neq\rho$, then
        \begin{equation}\label{TT4}
            |{I}_{22}(x,y,t)|\leq C(\O)(\frac{\d^s(y)}{\sqrt{t}}\wedge
            1)
            \times
            \left\{
            \begin{array}{lll}
                1+t^{\frac{2s-1}{2s}}+ t^{\frac{\s-s}{2s}}\d^{s-\rho}(x)+|x-y|^{\s-\rho}&\mbox{  if  }& 2s+\rho<1+\s,\\
                1+t^{\frac{2s-1}{2s}}\log\frac{D}{|x-y|}+t^{\frac{\s-s}{2s}}\d^{s-\rho}(x)+|x-y|^{\s-\rho}&\mbox{  if  }& 2s+\rho=1+\s,\\
                1+ t^{\frac{2s-1}{2s}} \frac{1}{|x-y|^{2s+\rho-1-\s}}+t^{\frac{\s-s}{2s}}\d^{s-\rho}(x)+|x-y|^{\s-\rho} &\mbox{  if   }& 2s+\rho>1+\s.\\
            \end{array}
            \right.
        \end{equation}

        Going back to estimates \eqref{{I}{21}}, \eqref{TT1},\eqref{TT2},
        \eqref{TT3} and \eqref{TT4}, it holds that

        $\bullet$ If $\rho=s$ and $\s=\rho$, then
        \begin{equation}\label{UU1}
            {I}_2(x,y,t)\leq \dfrac{C(\O)(\frac{\d^s(y)}{\sqrt{t}}\wedge
                1)}{(t^{\frac{1}{2s}}+|x-y|)^{N+\s}}
            \times\left\{
            \begin{array}{lll}
                t^{\frac{2s-1}{2s}}+|\log(\d(x))|+\log\frac{D}{|x-y|}&\mbox{  if  }& 2s<1,\\
                \log\frac{D}{|x-y|}+|\log(\d(x))|&\mbox{  if  }& 2s=1,\\
                \frac{t^{\frac{2s-1}{2s}}}{|x-y|^{2s-1}}+|\log(\d(x))| +\log\frac{D}{|x-y|}&\mbox{  if   }& 2s>1.\\
            \end{array}
            \right.
        \end{equation}

        $\bullet$ If $\rho=s$ and $\s\neq \rho$, then
        \begin{equation}\label{UU2}
            {I}_2(x,y,t)\leq
            \dfrac{C(\O)(\frac{\d^s(y)}{\sqrt{t}}\wedge
                1)}{(t^{\frac{1}{2s}}+|x-y|)^{N+\s}}\times
            \left\{
            \begin{array}{lll}
                1+t^{\frac{2s-1}{2s}}+ t^{\frac{\s-s}{2s}}|\log(\d(x))|+|x-y|^{\s-\rho}&\mbox{  if  }& 3s<1+\s,\\
                1+t^{\frac{2s-1}{2s}}\log\frac{D}{|x-y|}+t^{\frac{\s-s}{2s}}|\log(\d(x))|+|x-y|^{\s-\rho} &\mbox{  if  }& 3s=1+\s,\\
                1+t^{\frac{2s-1}{2s}} \frac{1}{|x-y|^{3s-1-\s}}+t^{\frac{\s-s}{2s}}|\log(\d(x))|+|x-y|^{\s-\rho} &\mbox{  if   }& 3s>1+\s.\\
            \end{array}
            \right.
        \end{equation}

        $\bullet$ If $\rho\neq s$ and $\s=\rho$, then
        \begin{equation}\label{UU3}
            {I}_2(x,y,t)\leq
            \dfrac{C(\O)(\frac{\d^s(y)}{\sqrt{t}}\wedge
                1)}{(t^{\frac{1}{2s}}+|x-y|)^{N+\s}}\times
            \left\{
            \begin{array}{lll}
                t^{\frac{2s-1}{2s}}+ t^{\frac{\s-s}{2s}}\d^{s-\rho}(x)+\log\frac{D}{|x-y|}&\mbox{  if  }& 2s<1,\\
                \log\frac{D}{|x-y|}+t^{\frac{\s-s}{2s}}\d^{s-\rho}(x)&\mbox{  if  }& 2s=1,\\
                t^{\frac{2s-1}{2s}} \frac{1}{|x-y|^{2s-1}}+t^{\frac{\s-s}{2s}}\d^{s-\rho}(x)+\log\frac{D}{|x-y|}&\mbox{  if   }& 2s>1.\\
            \end{array}
            \right.
        \end{equation}

        $\bullet$ If $\rho\neq s$ and $\s\neq\rho$, then
        \begin{equation}\label{UU4}
            {I}_2(x,y,t)\leq
            \dfrac{C(\O)(\frac{\d^s(y)}{\sqrt{t}}\wedge
                1)}{(t^{\frac{1}{2s}}+|x-y|)^{N+\s}}\times \left\{
            \begin{array}{lll}
                1+t^{\frac{2s-1}{2s}}+ t^{\frac{\s-s}{2s}}\d^{s-\rho}(x)+|x-y|^{\s-\rho}&\mbox{  if  }& 2s+\rho<1+\s,\\
                1+t^{\frac{2s-1}{2s}}\log\frac{D}{|x-y|}+t^{\frac{\s-s}{2s}}\d^{s-\rho}(x)+|x-y|^{\s-\rho}&\mbox{  if  }& 2s+\rho=1+\s,\\
                1+ t^{\frac{2s-1}{2s}} \frac{1}{|x-y|^{2s+\rho-1-\s}}+t^{\frac{\s-s}{2s}}\d^{s-\rho}(x)+|x-y|^{\s-\rho} &\mbox{  if   }& 2s+\rho>1+\s.\\
            \end{array}
            \right.
        \end{equation}

        As a consequence, from \eqref{I001}, \eqref{UU1}, \eqref{UU2},
        \eqref{UU3} and \eqref{UU4}, we reach that

        $\bullet$ If $\rho=s$ and $\s=\rho$, then
        \begin{equation}\label{RR1}
        {    |(-\Delta)_x ^{\frac{s}{2}}P_\O(x,y,t)|}\leq \dfrac{C
                (\frac{\d^s(y)}{\sqrt{t}}\wedge 1)}{
                (t^{\frac{1}{2s}}+|x-y|)^{N+s}} \\
            \times \left\{
            \begin{array}{lll}
                t^{\frac{2s-1}{2s}}+ |\log(\d(x))|+\log\frac{D}{|x-y|}&\mbox{  if  }& 2s<1,\\
                \log\frac{D}{|x-y|}+|\log(\d(x))|&\mbox{  if  }& 2s=1,\\
                \frac{t^{\frac{2s-1}{2s}}}{|x-y|^{2s-1}}+|\log(\d(x))|+\log\frac{D}{|x-y|}&\mbox{  if   }& 2s>1.\\
            \end{array}
            \right.
        \end{equation}

        $\bullet$ If $\rho=s$ and $\s\neq \rho$, then
        \begin{equation}\label{RR2}
            \begin{array}{lll}
                &{|(-\Delta)_x ^{\frac{\rho}{2}}P_\O(x,y,t)|}\leq \dfrac{C (\frac{\d^s(y)}{\sqrt{t}}\wedge 1)}{
                    (t^{\frac{1}{2s}}+|x-y|)^{N+\s}} \\
                &\times
                \left\{
                \begin{array}{lll}
                    \frac{1}{(t^{\frac{1}{2s}}+|x-y|)^{s-\s}}+1+t^{\frac{2s-1}{2s}}+ t^{\frac{\s-s}{2s}}|\log(\d(x))|+|x-y|^{\s-\rho}&\mbox{  if  }& 3s<1+\s,\\
                    \frac{1}{(t^{\frac{1}{2s}}+|x-y|)^{s-\s}}+1+t^{\frac{2s-1}{2s}}\log\frac{D}{|x-y|}+t^{\frac{\s-s}{2s}}|\log(\d(x))|+|x-y|^{\s-\rho} &\mbox{  if  }& 3s=1+\s,\\
                    \frac{1}{(t^{\frac{1}{2s}}+|x-y|)^{s-\s}}+ 1+t^{\frac{2s-1}{2s}}
                {    \frac{1}{|x-y|^{3s-1-\s}}}+t^{\frac{\s-s}{2s}}|\log(\d(x))|+|x-y|^{\s-\rho}
                    &\mbox{  if   }& 3s>
                    1+\s.\\
                \end{array}
                \right.
            \end{array}
        \end{equation}

        $\bullet$ If $\rho\neq s$ and $\s=\rho$, then
        \begin{equation}\label{RR3}
            \begin{array}{lll}
                &{|(-\Delta)_x ^{\frac{\rho}{2}}P_\O(x,y,t)|}\leq \dfrac{C (\frac{\d^s(y)}{\sqrt{t}}\wedge 1)}{
                    (t^{\frac{1}{2s}}+|x-y|)^{N+\rho}} \\
                &\times
                \left\{
                \begin{array}{lll}
                    \frac{\d^{s-\rho}(x)}{(t^{\frac{1}{2s}}+|x-y|)^{s-\rho}}+1+t^{\frac{2s-1}{2s}}+ t^{\frac{\rho-s}{2s}}\d^{s-\rho}(x)+\log\frac{D}{|x-y|}&\mbox{  if  }& 2s<1,\\
                    \frac{\d^{s-\rho}(x)}{(t^{\frac{1}{2s}}+|x-y|)^{s-\rho}}+\log\frac{D}{|x-y|}+t^{\frac{\rho-s}{2s}}\d^{s-\rho}(x)&\mbox{  if  }& 2s=1,\\
                    \frac{\d^{s-\rho}(x)}{(t^{\frac{1}{2s}}+|x-y|)^{s-\rho}}+t^{\frac{2s-1}{2s}} \frac{1}{|x-y|^{2s-1}}+t^{\frac{\rho-s}{2s}}\d^{s-\rho}(x)+\log\frac{D}{|x-y|} &\mbox{  if   }& 2s>1.\\
                \end{array}
                \right.
            \end{array}
        \end{equation}

        $\bullet$ If $\rho\neq s$ and $\s\neq\rho$, then
        \begin{equation}\label{RR4}
            \begin{array}{lll}
                &{|(-\Delta)_x ^{\frac{\rho}{2}}P_\O(x,y,t)|}\leq \dfrac{C (\frac{\d^s(y)}{\sqrt{t}}\wedge 1)}{
                    (t^{\frac{1}{2s}}+|x-y|)^{N+\s}} \\
                &\times
                \left\{
                \begin{array}{lll}
                    \frac{\d^{s-\rho}(x)}{(t^{\frac{1}{2s}}+|x-y|)^{s-\s}}+1+t^{\frac{2s-1}{2s}}+
                    t^{\frac{\s-s}{2s}}\d^{s-\rho}(x)+|x-y|^{\s-\rho}&\mbox{  if  }&
                    2s+\rho<
                    1+\s,\\
                    \frac{\d^{s-\rho}(x)}{(t^{\frac{1}{2s}}+|x-y|)^{s-\s}}+1+t^{\frac{2s-1}{2s}}\log\frac{D}{|x-y|}+t^{\frac{\s-s}{2s}}\d^{s-\rho}(x)+|x-y|^{\s-\rho}&\mbox{
                        if  }& 2s+\rho=
                    1+\s,\\
                    \frac{\d^{s-\rho}(x)}{(t^{\frac{1}{2s}}+|x-y|)^{s-\s}}+1+
                    t^{\frac{2s-1}{2s}}
                    \frac{1}{|x-y|^{2s+\rho-1-\s}}+t^{\frac{\s-s}{2s}}\d^{s-\rho}(x)+|x-y|^{\s-\rho}
                    &\mbox{  if   }& 2s+\rho>
                    1+\s.\\
                \end{array}
                \right.
            \end{array}
        \end{equation}
        Now, choosing $\s$ such that
        $$
        \s= \left\{\begin{array}{lll}
            2s+\rho-1 &\mbox{  if  }s\le \frac 12,\\
            \rho & \mbox{  if  }s>\frac 12, \end{array} \right.
        $$
        it follows that:

        \begin{enumerate}
            \item If $s\le \frac 12$ and $\s=2s+\rho-1$, then

            $\bullet$ If $s=\frac 12$, we have $\s=\rho$, hence using
            \eqref{RR1} and \eqref{RR3} (for $s=\frac 12$), it follows that
            \begin{equation*}
            {    |(-\Delta)_x ^{\frac{s}{2}}P_\O(x,y,t)|}\leq \dfrac{C
                    (\frac{\d^s(y)}{\sqrt{t}}\wedge 1)}{
                    (t+|x-y|)^{N+s}}
                \bigg(\log\frac{D}{|x-y|}+|\log(\d(x))|\bigg)  \mbox{  if
                } \rho=\frac 12,
            \end{equation*}
            and
            \begin{equation*}
                \begin{array}{lll}
            {   |(-\Delta)_x ^{\frac{\rho}{2}}P_\O(x,y,t)|} &\leq & \dfrac{C
                        (\frac{\d^s(y)}{\sqrt{t}}\wedge 1)}{
                        (t+|x-y|)^{N+\rho}}
                    \bigg(\frac{\d^{s-\rho}(x)}{(t+|x-y|)^{s-\rho}}+\log\frac{D}{|x-y|}+t^{\frac{\rho-s}{2s}}\d^{s-\rho}(x)\bigg)
                    \mbox{ if  }\rho\neq \frac 12.
                \end{array}
            \end{equation*}

            \item If $s<\frac 12$, then $\s=2s+\rho-1<\rho$. Hence
            \begin{enumerate}
                \item[$i)$] If $\rho=s$, then from \eqref{RR2}, it follows that
                \begin{equation}\label{PP3}
                    \begin{array}{lll}
                    {    |(-\Delta)_x ^{\frac{s}{2}}P_\O(x,y,t)|}&\leq & \dfrac{C
                            (\frac{\d^s(y)}{\sqrt{t}}\wedge 1)}{
                            (t^{\frac{1}{2s}}+|x-y|)^{N+\s}}\\
                        &\times &
                        \bigg(\frac{1}{(t^{\frac{1}{2s}}+|x-y|)^{s-\s}}+t^{\frac{2s-1}{2s}}\log\frac{D}{|x-y|}+t^{\frac{2s-1}{2s}}|\log(\d(x))|+|x-y|^{2s-1}\bigg).
                    \end{array}
                \end{equation}

                \item[$i)$] If $\rho\neq s$, then from \eqref{RR4}, it follows
                that
                \begin{equation}\label{PP4}
                    \begin{array}{lll}
                    {    |(-\Delta)_x ^{\frac{\rho}{2}}P_\O(x,y,t)|} &\leq & \dfrac{C
                            (\frac{\d^s(y)}{\sqrt{t}}\wedge 1)}{
                            (t^{\frac{1}{2s}}+|x-y|)^{N+\s}} \\
                        &\times &
                        \bigg(\frac{\d^{s-\rho}(x)}{(t^{\frac{1}{2s}}+|x-y|)^{s-\s}}+t^{\frac{2s-1}{2s}}\log\frac{D}{|x-y|}+t^{\frac{s+\rho-1}{2s}}\d^{s-\rho}(x)+|x-y|^{2s-1}\bigg).
                    \end{array}
                \end{equation}
            \end{enumerate}
        \end{enumerate}
        Therefore, we deduce that if $s\le \frac 12$, with
        $\s=2s+\rho-1$, we have
        \begin{equation}\label{QQQ1}
            \begin{array}{lll}
                & &{|(-\Delta)_x ^{\frac{\rho}{2}}P_\O(x,y,t)|}\leq\dfrac{C
                    (\frac{\d^s(y)}{\sqrt{t}}\wedge 1)}{
                    (t^{\frac{1}{2s}}+|x-y|)^{N+2s+\rho-1}} \\
                &\times &
                \bigg(\frac{\d^{s-\rho}(x)}{(t^{\frac{1}{2s}}+|x-y|)^{1-s-\rho}}+t^{\frac{2s-1}{2s}}\log\frac{D}{|x-y|}+t^{\frac{s+\rho-1}{2s}}\d^{s-\rho}(x)+
                t^{\frac{2s-1}{2s}}|\log(\d(x))|+|x-y|^{2s-1}\bigg).
            \end{array}
        \end{equation}

        Now, if $s>\frac 12$, then in this case we have $\s=\rho$. Thus
        \begin{enumerate}
            \item If $\rho=s$, using \eqref{RR1}, we have
            \begin{equation}\label{OO1}
         {   |(-\Delta)_x ^{\frac{s}{2}}P_\O(x,y,t)|}\leq \dfrac{C
                    (\frac{\d^s(y)}{\sqrt{t}}\wedge 1)}{
                    (t^{\frac{1}{2s}}+|x-y|)^{N+s}} \bigg(\frac{t^{\frac{2s-1}{2s}}}{|x-y|^{2s-1}}+ \log\frac{D}{|x-y|}+|\log(\d(x))|\bigg)
            \end{equation}

            \item If $\rho\neq s$, using \eqref{RR3}, we have
            \begin{equation}\label{OO2}
                \begin{array}{lll}
                {    |(-\Delta)_x ^{\frac{\rho}{2}}P_\O(x,y,t)|} &\leq & \dfrac{C
                        (\frac{\d^s(y)}{\sqrt{t}}\wedge 1)}{
                        (t+|x-y|)^{N+\rho}}
                    \bigg( \frac{\d^{s-\rho}(x)}{
                        (t^{\frac{1}{2s}}+|x-y|)^{s-\rho}}+\frac{t^{\frac{2s-1}{2s}}}{|x-y|^{2s-1}}+t^{\frac{\rho-s}{2s}}\d^{s-\rho}(x)+
                    \log\frac{D}{|x-y|}\bigg).
                \end{array}
            \end{equation}
        \end{enumerate}

        Hence, as a conclusion we deduce that if $s>\frac 12$, we have
        \begin{equation}\label{OO2LL}
            \begin{array}{lll}
            {    |(-\Delta)_x ^{\frac{\rho}{2}}P_\O(x,y,t)|} &\leq & \dfrac{C
                    (\frac{\d^s(y)}{\sqrt{t}}\wedge 1)}{
                    (t^{\frac{1}{2s}}+|x-y|)^{N+\rho}}
                \bigg( \frac{\d^{s-\rho}(x)}{
                    (t^{\frac{1}{2s}}+|x-y|)^{s-\rho}}+\frac{t^{\frac{2s-1}{2s}}}{|x-y|^{2s-1}}+t^{\frac{\rho-s}{2s}}\d^{s-\rho}(x)+
                \log\frac{D}{|x-y|}+|\log(\d(x))|\bigg).
            \end{array}
        \end{equation}
        {Therefore,} the result follows.
    \end{proof}
    As a {first} application of   we get the next estimate on the
    fractional gradient of $\mathcal{G}_s(x,y)$, the Green's function
    of the fractional Laplacian with Dirichlet condition.

    \begin{Corollary}
        Assume that $s>\frac 14$, then for all $s\le \rho<\min\{1,2s\}$,
        we have
        \begin{equation}\label{ellip}
            \begin{array}{lll}
                |(-\Delta)_x^{\frac{\rho}{2}}\mathcal{G}_s(x,y)|&\leq \dfrac{C}{
                    |x-y|)^{N-(2s-\rho)}}\left(\dfrac{\d^{s-\rho}(x)}{|x-y|^{s-\rho}}+\log(\frac{D}{|x-y|})
                +|\log(\delta(x))|\right).
            \end{array}
        \end{equation}
    \end{Corollary}
    \begin{proof}
        We give the proof in the case where $\frac 14<s\le \frac 12$, the author case
        follows in the same way.

        Let $\mathcal{G}_s(x,y)$ is the Green's function of the fractional
        Laplacian with Dirichlet condition, then
        \begin{equation*}
            \mathcal{G}_s(x,y)=\Int_{0}^{+\infty} P_{\Omega}(x,y,t) dt.
        \end{equation*}
        Hence
        $$
        |(-\Delta)_x^{\frac{\rho}{2}}\mathcal{G}_s(x,y)|\le
        \Int_{0}^{+\infty} |(-\Delta)_x^{\frac{\rho}{2}}P_{\Omega}(x,y,t)|
        dt.
        $$
        Recall that, since $s\le \frac{1}{2}$, then $\s=2s+\rho-1$ and
        \begin{equation*}
            \begin{array}{lll}
                & & {|(-\Delta)_x ^{\frac{\rho}{2}}P_\O(x,y,t)|}\leq\dfrac{C}{
                    (t^{\frac{1}{2s}}+|x-y|)^{N+2s+\rho-1}} \\
                &\times &
                \bigg(\frac{\d^{s-\rho}(x)}{(t^{\frac{1}{2s}}+|x-y|)^{1-s-\rho}}+t^{\frac{2s-1}{2s}}\log\frac{D}{|x-y|}+t^{\frac{s+\rho-1}{2s}}\d^{s-\rho}(x)+
                t^{\frac{2s-1}{2s}}|\log(\d(x))|+|x-y|^{2s-1}\bigg).
            \end{array}
        \end{equation*}
        Then
        \begin{equation*}
            \begin{array}{llll}
                & & |(-\Delta)_x^{\frac{\rho}{2}}\mathcal{G}_s(x,y)|\leq C\dyle
                \d^{s-\rho}(x)\bigg(\int_0^\infty\dfrac{dt}{
                    (t^{\frac{1}{2s}}+|x-y|)^{N+s}} dt+
                \int_0^\infty\dfrac{t^{\frac{s+\rho-1}{2s}}dt}{
                    (t^{\frac{1}{2s}}+|x-y|)^{N+2s+\rho-1}} dt\bigg)\\
                &+&\dyle (
                |\log(\d(x))|+\log\frac{D}{|x-y|})\int_0^\infty\dfrac{t^{\frac{2s-1}{2s}}dt}{
                    (t^{\frac{1}{2s}}+|x-y|)^{N+2s+\rho-1}} dt\\
                &+& \dyle|x-y|^{2s-1}\int_0^\infty\dfrac{dt}{
                    (t^{\frac{1}{2s}}+|x-y|)^{N+2s+\rho-1}} dt.
            \end{array}
        \end{equation*}
        Since $s>\frac 14$, then the above integrals are converging near
        the zero. Now, setting $\hat{t}=\frac{t}{|x-y|^{2s}}$, it holds
        that
        $$
        \int_0^\infty\dfrac{dt}{ (t^{\frac{1}{2s}}+|x-y|)^{N+s}} dt+
        \int_0^\infty\dfrac{t^{\frac{s+\rho-1}{2s}}dt}{
            (t^{\frac{1}{2s}}+|x-y|)^{N+2s+\rho-1}}\le
        \frac{C(\O)}{|x-y|^{N-s}},
        $$
        $$
        \int_0^\infty\dfrac{t^{\frac{2s-1}{2s}}dt}{
            (t^{\frac{1}{2s}}+|x-y|)^{N+2s+\rho-1}} dt\le
        \frac{C(\O)}{|x-y|^{N-(2s-\rho)}},
        $$
        and
        $$
        |x-y|^{2s-1}\int_0^\infty\dfrac{dt}{
            (t^{\frac{1}{2s}}+|x-y|)^{N+2s+\rho-1}} dt\le
        \frac{C(\O)}{|x-y|^{N-(2s-\rho)}}.$$ Thus,
        \begin{equation*}
            \begin{array}{lll}
                |(-\Delta)_x^{\frac{\rho}{2}}\mathcal{G}_s(x,y)|&\leq &\dyle
                \dfrac{C}{|x-y|^{N-(2s-\rho)}}\bigg(\frac{\d^{s-\rho}(x)}{|x-y|^{s-\rho}}+\log\frac{D}{|x-y|}+|\log(\d(x)|
                \bigg).
            \end{array}
        \end{equation*}
    \end{proof}
    \begin{remarks}
    Notice that in \cite{AFTY}, the estimate \eqref{ellip} is proved
    without any restriction on $s$.
    \end{remarks}

  Now, let us recall that for $\phi\in \mathcal{C}^\infty_0(\ren)$, the Riesz
    gradient is defined by
    $$
    \nabla^s \phi(x):=\int_{ \mathbb{R}^N}
    \frac{\phi(x)-\phi(y)}{|x-y|^s}\frac{x-y}{|x-y|}\frac{dy}{|x-y|^N},
    \quad \forall\ x \in \mathbb{R}^N.
    $$
    Then following the same computations as in the proof of Theorem
    \ref{regulaP}, it holds that
    \begin{Corollary}
        Assume that the conditions of Theorem \ref{regulaP} are satisfied,
        then for all $s\le \rho<\min\{2s,1\}$, we have

        $\bullet$ If $s\le \frac 12$, then

        \begin{equation*}
            \begin{array}{lll}
                & &{ |\n_x^\rho P_\O(x,y,t)|}\leq\dfrac{C(\O)
                    (\frac{\d^s(y)}{\sqrt{t}}\wedge 1)}{
                    (t^{\frac{1}{2s}}+|x-y|)^{N+2s+\rho-1}} \\
                &\times &
                \bigg(\frac{\d^{s-\rho}(x)}{(t^{\frac{1}{2s}}+|x-y|)^{1-s-\rho}}+t^{\frac{2s-1}{2s}}\log\frac{D}{|x-y|}+t^{\frac{s+\rho-1}{2s}}\d^{s-\rho}(x)+
                t^{\frac{2s-1}{2s}}|\log(\d(x))|+|x-y|^{2s-1}\bigg).
            \end{array}
        \end{equation*}

        $\bullet$ If $s>\frac 12$, then

        \begin{equation*}
            \begin{array}{lll}
            {    |\n_x^\rho P_\O(x,y,t)|} &\leq & \dfrac{C
                    (\frac{\d^s(y)}{\sqrt{t}}\wedge 1)}{
                    (t^{\frac{1}{2s}}+|x-y|)^{N+\rho}}
                \bigg( \frac{\d^{s-\rho}(x)}{
                    (t^{\frac{1}{2s}}+|x-y|)^{s-\rho}}+\frac{t^{\frac{2s-1}{2s}}}{|x-y|^{2s-1}}+t^{\frac{\rho-s}{2s}}\d^{s-\rho}(x)+
                \log\frac{D}{|x-y|}+|\log(\d(x))|\bigg).
            \end{array}
        \end{equation*}

    \end{Corollary}

    \section{Application to the regularity of the fractional heat
        equation.}\label{Regularity_heat_Equation_Section}

    The main goal of this section is to establish the regularity of the solution $w$ to Problem $(FHE)$  in suitable Bessel potential and fractional Sobolev spaces, depending on the regularity of the data $h$  and $ w_0$. This will be achieved by means of the representation formula and the sharp estimate previously obtained for the heat kernel.

    To simplify our presentation, we will consider separately the two cases:\\
    $\bullet$ Case: $h=0$, $w_0\in L^\s(\O)$;\\
    $\bullet$ Case:  $h\in L^m(\O_T) $, $w_0=0$. \vspace{0.2cm}

    \subsection{\textbf{Case:  $w_0\in L^\s(\O)$ and $h=0$.}}
    \
    \

   In this case, the problem $(FHE)$  takes the form
    \begin{equation}\label{eq_linear_01}
        \left\{
        \begin{array}{llll}
            w_t+(-\Delta)^s w&=& 0
            & \text{in}\quad\Omega_T,\vspace{0.2cm}\\
            w(x,t)&=&0&\text{in} \quad (\mathbb{R}^N\setminus\Omega)\times(0,T),\vspace{0.2cm}\\
            w(x,0)&=&w_0(x)&  \text{in} \quad \Omega.
        \end{array}
        \right.
    \end{equation}
    Before going further, we state in the next theorem, a regularity result related to the
    regularity of a class of hyper-singular integrals. This result will be used
    systematically throughout this section.

    \begin{Theorem}\label{first_integrals_regularity}
        Assume that $g\in L^m(\O)$ with {$m\geq1$}. For $x\in \O$ and $t>0$, we define
        $$
        G_{\l}(x,t)=\io \frac{g(y)}{(t^{\frac{ 1}{2s}}+|x-y|)^{N+\l}}dy,
        $$
        where $\l\in \mathbb{R}$. Then,  there exists a positive
        constant depending on $\O$ and the data and it is
        independent of $t$ and $g$ such that:
        \

        $\bullet$ If $N+\l\le 0$, then $G_{\l}(x,t)\le
        C(\O)(t^{-\frac{N+\l}{2s}}+1)||g||_{L^1(\O)}$ a.e. for $t>0$ and
        $x\in \O$.

        $\bullet$ If $\l\ge 0$, then for $p>m$, we have
        \begin{equation}\label{hyper10}
            ||G_{\l}(.,t)||_{L^p(\O)}\le  C
            t^{\frac{-\l}{2s}-\frac{N}{2s}(\frac
                1m-\frac{1}{\ell})}||g||_{L^m(\O)}.
        \end{equation}
        $\bullet$ If $-N<\l<0$, then considering $m_0$ such that $1\le
        m_0<\min\{m,-\frac{N}{\l}\}$, then for all $p>m_0$ and for all
        $\nu>0$, we have
        \begin{equation}\label{hyper100}
            ||G_{\l}(.,t)||_{L^p(\O)}\le  C(\O)\times
            \left\{
            \begin{array}{lll}
                t^{\frac{-\l}{2s}-\frac{N}{2s}(\frac
                    {1}{m_0}-\frac{1}{p})}||g||_{L^{m}(\O)} &\mbox{  if
                }p>\frac{m_0N}{N+m_0\l},\\
                t^{-\nu\frac{N}{2s}}||g||_{L^{m}(\O)}&\mbox{  if
                }p\le \frac{m_0N}{N+m_0\l}.
            \end{array}
            \right.
        \end{equation}
        Furthermore, in the case where $\l<0$, then for all $p>m$, we have
        \begin{equation}\label{YY00}
            ||G_{\l}(.,t)||_{L^p(\O)}\le  C(t^{-\frac{\l}{2s}}+1)
            t^{-\frac{N}{2s}(\frac
                1p-\frac{1}{m})}||g||_{L^m(\O)}.
        \end{equation}
    \end{Theorem}
    We refer the reader to the Appendix for the proof.

    \medbreak

    Now, we are ready to state our first  regularity result about  the fractional regularity of the solution to Problem \eqref{eq_linear_01}.
    In order to clarify the structure of the proof, we proceed by distinguishing two separate cases :
    \begin{itemize}
    \item[$\bullet$] Inside regularity: $x\in \O$ and $t\in (0,T)$;
     \item[$\bullet$] Outside regularity: $x\in \mathbb{R}^N\setminus \O$ and $t\in
    (0,T)$.
  \end{itemize}

    Let us begin by the next {theorem}.
    \begin{Theorem}\label{reg_frac_u0_t}
        Assume that $s\le \rho<\max\{1,2s\}$. Suppose that $w_0\in
        L^\s(\O)$ with $\s\ge 1$ and define $\widehat{\s}\le
        \min\{\s,\frac{1}{\rho-s}\}$. Let $w$ be the unique weak solution
        to Problem \eqref{eq_linear_01}, then
        \begin{enumerate}
            \item If $2s+\rho\ge 1$, then for all $\eta>0$ small enough and for
            all $p>\widehat{\s}$, we have
            $$
            \begin{array}{lll}
                & &||(-\Delta)^{\frac{\rho}{2}} w(.,t)||_{L^p(\O)}\leq C(\O)
                t^{-\frac{N}{2s}(\frac{1}{\widehat{\s}}-\frac{1}{p})-\frac{1}{2}}\\
                &\times & \bigg( t^{-\frac{N}{2s}(1+\eta)(\rho-s)}+
                t^{-\frac{\rho-s}{2s}}+t^{-\frac{\rho+\eta-s}{2s}}
                +t^{-\frac{N\eta}{2sp(1+\eta)}-\frac{\rho-s}{2s}}
                \bigg)||w_0||_{L^{\widehat{\s}}(\O)}.
            \end{array}
            $$
            \item If $2s+\rho<1$, then setting $\s_0<\min\{\widehat{\s},
            \frac{N}{1-2s-\rho}\}$, for all
            $p>\frac{\s_0N}{N-\s_0(1-2s-\rho)}$, we have
            $$
            \begin{array}{lll}
                & &||(-\Delta)^{\frac{\rho}{2}} w(.,t)||_{L^p(\O)}\leq C(\O)
                t^{-\frac{N}{2s}(\frac{1}{\s_0}-\frac{1}{p})-\frac{1}{2}}\\
                &\times & \bigg( t^{-\frac{N}{2s}(1+\eta)(\rho-s)}+
                t^{-\frac{\rho+\eta-s}{2s}}+
                t^{-\frac{N\eta}{2sp(1+\eta)}-\frac{\rho-s}{2s}}
                +t^{-\frac{\rho-s}{2s}}\bigg)||w_0||_{L^{\widehat{\s}}(\O)}.
            \end{array}
            $$
        \end{enumerate}
    \end{Theorem}

    \begin{proof}
        Let $w$ be the solution of Problem \eqref{eq_linear_01}, then
        $$w(x,t)=\int_{\O} w_0(y)P_\O(x,y,t)dy.$$
        Hence
 {   $$|(-\Delta)^{\frac{\rho}{2}}w(x,t)|\leq \int_{\O}
w_0(y)|(-\Delta)_x^{\frac{\rho}{2}}P_\O(x,y,t)|dy.$$}

        According to the value of $s$ and the sign of $2s+\rho-1$, we will
        consider three main cases.

{    {\bf Case:  $2s+\rho\ge 1$ and $s\le \frac 12$.}}

        Using estimate \eqref{gradp12} on
{    $|(-\Delta)_x^{\frac{\rho}{2}}P_\O(x,y,t)|$}, we obtain that
        \begin{equation}\label{nancy1}
            |(-\Delta)^{\frac{\rho}{2}}w(x,t)|\leq C\left(I_1(x,t)+t^{\frac{2s-1}{2s}} I_2(x,t)+t^{\frac{s+\rho-1}{2s}}I_3(x,t)+
            t^{\frac{2s-1}{2s}}I_4(x,t)+I_5(x,t)\right),
        \end{equation}
        where
        $$I_1(x,t):=\d^{s-\rho}(x)\Int_{\O} \dfrac{w_0(y)}{ (t^{\frac{1}{2s}}+|x-y|)^{N+s}}dy,\:\:I_2(x,t):=\Int_{\O} \dfrac{w_0(y)}{ (t^{\frac{1}{2s}}+|x-y|)^{N+(2s+\rho-1)}}\log(\frac{D}{|x-y|})dy,$$
        $$ I_3(x,t):= \d^{s-\rho}(x)\Int_{\O} \dfrac{w_0(y)}{ (t^{\frac{1}{2s}}+|x-y|)^{N+(2s+\rho-1)}}dy,\:\: I_4(x,t):=|\log(\d(x))|\Int_{\O} \dfrac{w_0(y)}{
            (t^{\frac{1}{2s}}+|x-y|)^{N+(2s+\rho-1)}}dy,$$ and
        $$I_5(x,t):=\Int_{\O} \dfrac{w_0(y)}{|x-y|^{1-2s}(t^{\frac{1}{2s}}+|x-y|)^{N+(2s+\rho-1)}}dy.$$

        Define
        $$
        \widehat{I}(x,t)=\Int_{\O} \dfrac{w_0(y)}{
            (t^{\frac{1}{2s}}+|x-y|)^{N+s}}dy,
        $$
        and
        $$
        \widetilde{I}(x,t)=\Int_{\O} \dfrac{w_0(y)}{
            (t^{\frac{1}{2s}}+|x-y|)^{N+2s+\rho-1}}dy.
        $$
        Using Theorem \ref{first_integrals_regularity} with $\l=s$, and $\l=2s+\rho-1\ge 0$ respectively, we reach that, for all $p>\s$,
        \begin{equation}\label{Main000}
            ||\widehat{I}(.,t)||_{L^p(\O)}\leq C
            t^{-\frac{N}{2s}(\frac{1}{\s}-\frac{1}{p})-\frac{1}{2}}
            ||w_0||_{L^\s(\O)}
        \end{equation}
        and
        \begin{equation}\label{Main001}
            ||\widetilde{I}(.,t)||_{L^p(\O)}\leq C
            t^{-\frac{N}{2s}(\frac{1}{\s}-\frac{1}{p})-\frac{2s+\rho-1}{2s}}
            ||w_0||_{L^\s(\O)}.
        \end{equation}
        Let us begin by estimating $I_1$. We have $I_1(x,t)=\d^{s-\rho}(x)
        \widehat{I}(x,t)$.

        Since $\d^{-\beta}\in L^1(\O)$ if and only if $\beta<1$, then for
        $\widehat{\s}<p<\frac{1}{\rho-s}$ fixed, for $\eta>0$ small
        enough, using H\"older's inequality, it holds that
        \begin{equation}\label{vacance1}
            ||{I}_1(.,t)||_{L^p(\O)}\leq C
            t^{-\frac{N}{2s}(\frac{1}{\widehat{\s}}-\frac{1}{p}+(1+\eta)(\rho-s))-\frac{1}{2}}
            ||w_0||_{L^{\widehat{\s}}(\O)}.
        \end{equation}
        Notice that $I_3(x,t)=\delta^{s-\rho}(x))\widetilde{I}(x,t)$. Then
        as above, we deduce that $I_3(.,t)\in L^p(\O)$ for all $p>\s$ and
        \begin{equation}\label{I33}
            ||I_3(.,t)||_{L^p(\O)}\leq C
            t^{-\frac{N}{2s}(\frac{1}{\widehat{\s}}-\frac{1}{p}+(1+\eta)(\rho-s))-\frac{2s+\rho-1}{2s}}
            ||w_0||_{L^\s(\O)}.
        \end{equation}
        We deal with the term $I_2$. We know that
        $\log(\frac{D}{|x-y|})\leq C|x-y|^{-\eta}$ for all $\eta>0$ and $x,y\in \O$.
        Thus, fixed $\eta$ to be chosen later, we have
        $$I_2(x,t)\leq C \Int_{\O} \dfrac{w_0(y)|x-y|^{-\eta}}{ (t^{\frac{1}{2s}}+|x-y|)^{N+2s+\rho-1}}dy.$$
        To estimate $I_2$, we use a duality argument. {Indeed, let $\phi\in\mathcal{C}^\infty_0(\Omega)$ such that}
        \begin{equation*}
            \begin{array}{lll}
                ||I_2(.,t)||_{L^{p }(\Omega)}&=&\dyle\sup\limits_{\{|| \phi|| _{L^{p'}(\Omega)\leq 1}\}}{\int_{\O}}  I_2(x,t)\phi(x) dx dt,\vspace{0.2cm}\\
                &\leq &\dyle\sup\limits_ {\{|| \phi|| _{L^{p'}(\Omega)\leq 1}\}}\int_{\O} |\phi(x)|\int_{\Omega}|w_0(y)|{H}(x-y,t)dy  dx,\\
            \end{array}
        \end{equation*}
        where
        \begin{equation*}
            {H}(x,t):=\dfrac{|x|^{-\eta}}{(t^{\frac{1}{2s}}+|x|)^{N+2s+\rho-1}}.
        \end{equation*}
        By  using Young's inequality, we obtain
        \begin{equation}
            \begin{array}{lll}
                ||I_2(.,t)||_{L^{p}(\Omega)}&\leq& \sup\limits_{\{|| \phi||
                    _{L^{p'}(\Omega)\leq 1}\}} || \phi || _{L^{p'}(\Omega)}||
                w_0||_{L^{\s}(\Omega)} || H(.,t)||_{L^{a}(\Omega)},
            \end{array}
        \end{equation}
        with
        \begin{equation}
            \dfrac{1}{a}+\dfrac{1}{p'}+\dfrac{1}{\s}=2.
        \end{equation}
        Thus, choosing $\eta$ small enough such that $\eta a<N$, we
        find {that}
        \begin{equation*}
            || H(.,t)||_{L^{a}(\Omega)}\leq C
            t^{\frac{N}{2sa}-\frac{\eta}{2s}-\frac{N+2s+\rho-1}{2s}}.
        \end{equation*}

        Hence
        \begin{equation}\label{fifth_estimation}
            ||I_2(.,t)||_{L^{p}(\Omega)}\leq C
            t^{\frac{-N}{2s}(\frac{1}{\s}-\frac{1}{p})-\frac{2s+\rho+\eta-1}{2s}}
            ||w_0||_{L^{\s}(\Omega)}.
        \end{equation}
        In the same way, choosing $\eta=1-2s$, it follows that
        \begin{equation}\label{I55}
            ||I_5(.,t)||_{L^{p}(\Omega)}\leq C
            t^{\frac{-N}{2s}(\frac{1}{\s}-\frac{1}{p})-\frac{\rho}{2s}}
            ||w_0||_{L^{\s}(\Omega)}.
        \end{equation}
        Related to $I_4$, we have $$I_4(x,t)=|\log(\delta(x))|\widetilde{I}(x,t).$$ Since $|\log(\delta(x))| \in L^{\theta}(\O)$ for all $\theta\in
        [1,\infty)$, we deduce that $I_4(.,t)\in L^p(\O)$ for all $p>\s$
        with
        \begin{equation}\label{I33-33}
            ||I_4(.,t)||_{L^p(\O)}\leq C
            t^{-\frac{N}{2s}(\frac{1}{\s}-\frac{1}{p(1+\eta)})-\frac{2s+\rho-1}{2s}}
            ||w_0||_{L^\s(\O)}.
        \end{equation}
        Thus, combining {the estimates \eqref{vacance1}, \eqref{I33}, \eqref{fifth_estimation}, \eqref{I55} and \eqref{I33-33}} , we deduce that
        $$
        \begin{array}{lll}
            & &||(-\Delta)^{\frac{\rho}{2}} w(.,t)||_{L^p(\O)}\leq C(\O)
            t^{-\frac{N}{2s}(\frac{1}{\widehat{\s}}-\frac{1}{p})-\frac{1}{2}}\\
            &\times & \bigg( t^{-\frac{N}{2s}(1+\eta)(\rho-s)}+
            t^{-\frac{\rho+\eta-s}{2s}}+
            t^{-\frac{N\eta}{2sp(1+\eta)}-\frac{\rho-s}{2s}}
            +t^{-\frac{\rho-s}{2s}}\bigg)||w_0||_{L^{\widehat{\s}}(\O)}.
        \end{array}
        $$

{    {\bf Case: $2s+\rho\ge 1$ and $s>\frac 12$.}}   From
        \eqref{gradp22}, it holds that
        \begin{equation}\label{nancy2}
            |(-\Delta)^{\frac{\rho}{2}}w(x,t)|\leq C\left(K_1(x,t)+t^{\frac{2s-1}{2s}}K_2(x,t)+t^{\frac{\rho-s}{2s}}K_3(x,t)+
            K_4(x,t)+K_5(x,t)\right),
        \end{equation}
        where
        $$K_1(x,t):=\d^{s-\rho}(x)\Int_{\O} \dfrac{w_0(y)}{
            (t^{\frac{1}{2s}}+|x-y|)^{N+s}}dy,\:\:K_2(x,t):=\Int_{\O} \dfrac{w_0(y)}{|x-y|^{2s-1}(t^{\frac{1}{2s}}+|x-y|)^{N+\rho}}dy,$$

        $$K_3(x,t):=\d^{s-\rho}(x)\Int_{\O} \dfrac{w_0(y)}{
            (t^{\frac{1}{2s}}+|x-y|)^{N+\rho}}dy,\:\:K_4(x,t):= \Int_{\O} \dfrac{w_0(y)}{ (t^{\frac{1}{2s}}+|x-y|)^{N+\rho}}\log\frac{D}{|x-y|}dy,$$
        and
        $$K_5(x,t):=|\log(\d(x))|\Int_{\O} \dfrac{w_0(y)}{(t^{\frac{1}{2s}}+|x-y|)^{N+\rho}}dy.$$
        It is clear that $K_1=I_1$. Hence, for
        $\widehat{\s}<p<\frac{1}{\rho-s}$ fixed, for $\eta>0$ small
        enough, we have
        \begin{equation}\label{K01}
            ||{K}_1(.,t)||_{L^p(\O)}\leq C
            t^{-\frac{N}{2s}(\frac{1}{\widehat{\s}}-\frac{1}{p}+(1+\eta)(\rho-s))-\frac{1}{2}}
            ||w_0||_{L^{\widehat{\s}}(\O)}.
        \end{equation}
        Define
        $$
        \widetilde{K}(x,t)=\Int_{\O} \dfrac{w_0(y)}{
            (t^{\frac{1}{2s}}+|x-y|)^{N+\rho}}dy.
        $$
        Using Theorem \ref{first_integrals_regularity} with $\l=\rho$, we deduce that for $p>\s$,
        \begin{equation}\label{KMain001}
            ||\widetilde{K}(.,t)||_{L^p(\O)}\leq C
            t^{-\frac{N}{2s}(\frac{1}{\s}-\frac{1}{p})-\frac{\rho}{2s}}
            ||w_0||_{L^\s(\O)}.
        \end{equation}
        Notice that { $K_3(x,t)=\d^{s-\rho}(x)\widetilde{K}(.,t)$, }
     as in the estimating of $K_1$, for
        $\widehat{\s}<p<\frac{1}{\rho-s}$ fixed, for $\eta>0$ small
        enough, we have
        \begin{equation}\label{K03}
            ||{K}_3(.,t)||_{L^p(\O)}\leq C
            t^{-\frac{N}{2s}(\frac{1}{\widehat{\s}}-\frac{1}{p}+(1+\eta)(\rho-s))-\frac{\rho}{2s}}
            ||w_0||_{L^{\widehat{\s}}(\O)}.
        \end{equation}

        Since $K_5(x,t)=|\log(\d(x))|\widetilde{K}(.,t)$, then as for the
        term $I_3$, we reach that
        \begin{equation}\label{K5}
            ||{K}_5(.,t)||_{L^p(\O)}\leq C
            t^{-\frac{N}{2s}(\frac{1}{\s}-\frac{1}{p(1+\eta)})-\frac{\rho}{2s}}
            ||w_0||_{L^\s(\O)}.
        \end{equation}
        To estimate $K_2$ and $K_4$, we use the same duality argument used
        for estimating $I_2$. Hence by a direct computation, it follows that
        \begin{equation}\label{K222}
            ||K_2(.,t)||_{L^{p}(\Omega)}\leq C
            t^{\frac{-N}{2s}(\frac{1}{\s}-\frac{1}{p})-\frac{2s+\rho-1}{2s}}
            ||w_0||_{L^{\s}(\Omega)},
        \end{equation}
        and
        \begin{equation}\label{K444}
            ||K_4(.,t)||_{L^{p}(\Omega)}\leq C
            t^{\frac{-N}{2s}(\frac{1}{\s}-\frac{1}{p})-\frac{\rho+\eta}{2s}}
            ||w_0||_{L^{\s}(\Omega)}.
        \end{equation}
        {Therefore, } we conclude that
        $$
        \begin{array}{lll}
            & &||(-\Delta)^{\frac{\rho}{2}} w(.,t)||_{L^p(\O)}\leq C(\O)
            t^{-\frac{N}{2s}(\frac{1}{\widehat{\s}}-\frac{1}{p})-\frac{1}{2}}\\
            &\times & \bigg( t^{-\frac{N}{2s}(1+\eta)(\rho-s)}+
            t^{-\frac{\rho-s}{2s}}+t^{-\frac{\rho+\eta-s}{2s}}
            +t^{-\frac{N\eta}{2sp(1+\eta)}-\frac{\rho-s}{2s}}
            \bigg)||w_0||_{L^\s(\O)}.
        \end{array}
        $$
{   Hence,  the} result follows in this case.

{    {\bf Case:
$2s+\rho<1$.} In this case we have trivially that $s\le \frac 12$.
Fixed $1\le
            \s_0<\min\{\widehat{\s},\frac{N}{1-2s-\rho}\}$, then using the
            second point in Theorem \ref{first_integrals_regularity}, it holds
            that for all $p>\frac{\s_0N}{N-\s_0(1-2s-\rho)}$, we have
            $$
            || \widehat{I}(.,t)||_{L^p(\O)}\le  C(\O)t^{-\frac{N}{2s}(\frac
                {1}{\s_0}-\frac{1}{p})+\frac{1-2s-\rho}{2s}}||w_0||_{L^{\s}(\O)}
            $$
            Thus, fixed $\s_0$ and $p>\frac{\s_0N}{N-\s_0(1-2s-\rho)}$, we
            reach that $I_3(.,t)\in L^p(\O)$ for all
            $p>\frac{\s_0N}{N-\s_0(1-2s-\rho)}$ and
            \begin{equation}\label{I33ss}
                ||I_3(.,t)||_{L^p(\O)}\leq C
                t^{-\frac{N}{2s}(\frac{1}{\s_0}-\frac{1}{p}+(1+\eta)(\rho-s))-\frac{2s+\rho-1}{2s}}
                ||w_0||_{L^\s(\O)}.
            \end{equation}
            In the same way we obtain that
            $$
            ||I_2(.,t)||_{L^{p}(\Omega)}\leq C
            t^{\frac{-N}{2s}(\frac{1}{\s_0}-\frac{1}{p})-\frac{2s+\rho+\eta-1}{2s}}
            ||w_0||_{L^{\s}(\Omega)},
            $$
            and $$ ||I_4(.,t)||_{L^p(\O)}\leq C
            t^{-\frac{N}{2s}(\frac{1}{\s_0}-\frac{1}{p(1+\eta)})-\frac{2s+\rho-1}{2s}}
            ||w_0||_{L^\s(\O)}.
            $$
            Combining the above estimates, for $p>\frac{\s_0N}{N-\s_0(1-2s-\rho)}$,
            we get
            \begin{equation}\label{negat1}
                \begin{array}{lll}
                    & &||(-\Delta)^{\frac{\rho}{2}} w(.,t)||_{L^p(\O)}\leq C(\O)
                    t^{-\frac{N}{2s}(\frac{1}{\s_0}-\frac{1}{p})-\frac{1}{2}}\\
                    &\times & \bigg( t^{-\frac{N}{2s}(1+\eta)(\rho-s)}+
                    t^{-\frac{\rho+\eta-s}{2s}}+
                    t^{-\frac{N\eta}{2sp(1+\eta)}-\frac{\rho-s}{2s}}
                    +t^{-\frac{\rho-s}{2s}}\bigg)||w_0||_{L^{\widehat{\s}}(\O)}.
                \end{array}
        \end{equation}}

    \end{proof}

    \begin{remarks}
        In the case $2s+\rho<1$, taking into consideration that
        $$
        \Int_{\O} \dfrac{w_0(y)}{
            (t^{\frac{1}{2s}}+|x-y|)^{N+(2s+\rho-1)}}dy \le
        C(t^{\frac{1-2s-\rho}{2s}}+1)\Int_{\O} \dfrac{w_0(y)}{
            (t^{\frac{1}{2s}}+|x-y|)^{N}}dy,
        $$
        Estimating in the same way all the integrals where the exponent $2s+\rho-1$ appears and using
        the third point in Theorem \ref{first_integrals_regularity}, we
        obtain that
        \begin{equation}\label{negat2}
            \begin{array}{lll}
                & &||(-\Delta)^{\frac{\rho}{2}} w(.,t)||_{L^p(\O)}\leq
                C(\O)t^{-\frac{N}{2s}(\frac{1}{\widehat{\s}}-\frac{1}{p})-\frac{1}{2}}
                \bigg(t^{-\frac{N}{2s}(1+\eta)(\rho-s)}+t^{-\frac{\rho-s}{2s}}\bigg)||w_0||_{L^{\widehat{\s}}(\O)}\\
                &+ &
                C(t^{\frac{1-2s-\rho}{2s}}+1)t^{-\frac{N}{2s}(\frac{1}{\widehat{\s}}-\frac{1}{p})}\bigg(t^{\frac{2s-1-\eta}{2s}}+
                t^{-\frac{N}{2s}(1+\eta)(\rho-s)+\frac{s+\rho-1}{2s}}+t^{\frac{2s-1}{2s}}\bigg)||w_0||_{L^{\widehat{\s}}(\O)}
            \end{array}
        \end{equation}
        Notice that estimate \eqref{negat2} is equivalent to estimate
        \eqref{negat1} for $t>>1$.
    \end{remarks}

    As a consequence, we have the next result.
    \begin{Corollary}\label{CRR}
        Assume that the hypotheses of Theorem \ref{reg_frac_u0_t}    are
        all satisfied, and let $w$ be the unique weak solution to Problem
        \eqref{eq_linear_01}. Then

        \begin{enumerate}
            \item If $s\le \rho<s+\frac{s}{N}$ and $\s=1$, then
            $(-\Delta)^{\frac{\rho}{2}} w\in L^{p}(\Omega_T)$ for all
            $p<\frac{N+2s}{N+s+N(\rho-s)}$ and
            $$ ||(-\Delta)^{\frac{\rho}{2}}
            w||_{L^{p}(\Omega_T)}\le
            C(\O_T)||w_0||_{L^1(\O)},$$
            where $C(\O_T)\to 0$ as $T\to 0$.

            \item Assume that $\rho=s$, define
            $$\widehat{p_s}:=
            \left\{
            \begin{array}{lll}
                \frac{\s(N+2s)}{N+\s
                    s} \mbox{  if      }\:\:3s>1\\ \\
                \frac{\s_0(N+2s)}{N+\s_0
                    s} \mbox{  if      }\:\: 3s<1,
            \end{array}
            \right.
            $$
            where $\s_0<\min\{\s,\frac{N}{1-3s}\}$. Then for all $p<\widehat{p_s}$, there exists a positive constant $C=C(\O)$
            such that
            \begin{equation}
                ||(-\Delta)^{\frac{s}{2}} w||_{L^{p}(\Omega_T)}\le
                C(\O)\bigg(
                T^{-\frac{N}{2s}(\frac{1}{\s}-\frac{1}{p})-\frac{1}{2}+\frac{1}{p}}+
                T^{-\frac{N}{2s}(\frac{1}{\s}-\frac{1}{p(1+\eta)})-\frac{1}{2}+\frac{1}{p}}+
                T^{-\frac{N}{2s}(\frac{1}{\s}-\frac{1}{p})-\frac{s+\eta}{2s}+\frac{1}{p}}\bigg)
                ||w_0||_{L^\s(\Omega)},
            \end{equation}
            where $\eta>0$ is small enough and chosen such that
 $p<\min\bigg\{\frac{\s(N+2s)}{N+\s(s+\eta)}, \frac{\s(\frac{N}{1+\eta}+2s)}{N+\s s}\bigg\}$.

            In particular, if $\s=1$, then for all $p<\frac{N+2s}{N+s}$, we have $||(-\Delta)^{\frac{s}{2}}
            w||_{L^{p}(\Omega_T)}\le C(\O_T)||w_0||_{L^1(\O)}$ where
            $C(\O_T)\to 0$ as $T\to +\infty$.
        \end{enumerate}
    \end{Corollary}
    To complete our global regularity result in this case, we have to
    estimate the fractional gradient of $w$ outside $\O$. More
    precisely, we have the next proposition :
    \begin{Proposition}\label{ext1}
        Suppose that   the assumptions of Theorem \ref{reg_frac_u0_t} are satisfied. Let $w$ be the unique weak solution to Problem
        \eqref{eq_linear_01}. Assume that $s\le \rho<\max\{2s,1\}$, $p<\frac{1}{\rho-s}$ and  fix $\a$ small enough
        such that $p<\frac{1}{\rho-s+\a}$.
        Let $p_0\ge p$ be such that $p_0>\frac{Np}{N+\a p-pN(\rho-s+\a)}$, then, there exists a
        positive constant $C$ depends only on $\O$ and the data and
        it is independent of $t$, such that
        \begin{equation}\label{exter001}
            ||(-\Delta)^{\frac{\rho}{2}}
            w(.,t)||_{L^{p}(\mathbb{R}^N\backslash \Omega)}\le
            C(\O)   {\bigg\|\dfrac{w}{\delta^s}(.,t)\bigg \|_{L^{p_0}(\Omega)}}\le
            C(\O)t^{-\frac{N}{2s}(\frac{1}{\s}-\frac{1}{p_0})-\frac{1}{2}}\|w_0\|_{L^{\s}(\O)}.
        \end{equation}
        Moreover, for $\rho=s$, then $|(-\Delta)^{\frac{s}{2}} w| \in
        {L^p((\mathbb{R}^N\backslash \Omega) \times (0,T))}$ for all
        $p<\widehat{p_s}:=\dfrac{\s(N+2s)}{N+\s s}$.
    \end{Proposition}
    \begin{proof}
        Without loss of generality,  we  assume that $w_0\gneqq 0$. Thus,  $w\gneqq 0$ in $\mathbb{R}^N$. Hence,
        for $x\in \mathbb{R}^N\backslash \Omega$, we have
        $$
        |(-\Delta)^{\frac{\rho}{2}} w(x,t)|=\bigg|\Int _{\Omega } \dfrac {-w(y,t)}{|x-y|^{N+\rho}} dy\bigg|=\Int _{\Omega } \dfrac {w(y,t)}{|x-y|^{N+\rho}} dy.
        $$
        Notice that, for $y\in \Omega$, we have $|x-y|\ge \delta (y)$. Let $p\geq \s $ be fixed, we consider the exterior set $\Omega_1=\{x\in  \mathbb{R}^N\backslash \O;\, \text{dist}(x,\partial\Omega)>>1\}$, then $|x-y|\ge \frac{|x|+1}{2}$ for all $y\in \Omega$. Thus, for $x\in \O_1$, we get
        $$
        |(-\Delta)^{\frac{\rho}{2}} w(x,t)|\le \frac{2}{(|x|+1)^{N+\rho}}\int _{\Omega } |w(y,t)|dy\le\frac{C}{(|x|+1)^{N+\rho}}\bigg\|\frac{w(.,t)}{\d^s}\bigg\|_{L^1(\Omega)}.
        $$
        Therefore,  we conclude that $|(-\Delta)^{\frac{\rho}{2}} w(.,t)|\in L^1(\Omega_1)\cap
        L^\infty(\Omega_1)$ and, for all $\theta\ge 1$, for all $r>\s$, we have

        \begin{equation}\label{eqrr}
            ||(-\Delta)^{\frac{\rho}{2}} w(,t)||_{L^\theta(\O_1)}\le C(\O,r)\bigg\|\frac{w(.,t)}{\d^s}\bigg\|_{L^r(\O)}\le C(\O,r)t^{-\frac{N}{2s}(\frac{1}{\s}-\frac{1}{r})-\frac 12}||w_0||_{L^\s(\O)}.
        \end{equation}

        Now,  we treat the integral in the set $\Omega_2=\mathbb{R}^N\backslash (\Omega_1\cup \Omega)$. Without loss of generality,  we can assume that $0<\text{dist}(x, \partial\Omega)\le 2$ for all $x\in \O_2$.

        \vspace{0.2cm}
        So, for $\alpha>0$ small enough to be chosen later and for $x\in \O_2$, we have
        \begin{equation*}
            \begin{array}{lll}
                |(-\Delta)^{\frac{\rho}{2}} w(x,t)|&=&\Int _{\Omega } \dfrac {w(y,t)}{|x-y|^{N+\rho}} dy\vspace{0.2cm} \\
                &\le & \dfrac{1}{\d^{\rho-s+\alpha}(x)}\Int _{\Omega} \dfrac {w(y,t)}{\delta^s(y)}\dfrac{1}{|x-y|^{N-\alpha}}dy=
                \frac{R(x,t)}{(\text{dist}(x, \partial\Omega))^{\rho-s+\alpha}},
            \end{array}
        \end{equation*}
        where
        $$
        R(x,t)=\int_{\O}\dfrac {w(y,t)}{\delta^s(y)}\frac{1}{|x-y|^{N-\alpha}} dy dt.
        $$
        Since $w_0\in L^\s(\Omega)$ with $\s\geq 1$, then using Proposition \ref{pro:lp2}, we obtain that
        $\dfrac{w(y,t)}{\delta^s(y)}\in L^{p_0}(\O)$ for all $p_0\geq \s$, to be chosen later, and
        $$\bigg\|\dfrac{w}{\delta^s}(.,t)\bigg\|_ {L^{p_0}(\Omega)}\leq Ct^{-\frac{N}{2s}(\frac{1}{\s}-\frac{1}{p_0})-\frac{1}{2}} ||w_0||_{L^{\s}(\Omega)}.$$
        For $t\in (0,T)$ fixed, since $\O$ is bounded, {by }using Theorem \ref{stein1}, it follows that $R(.,t)\in L^{\frac{p_0N}{N-p_0\a}}(\ren)\cap L^1(\ren)$
        and
        $$
        ||R(.,t)||_{L^{\frac{p_0N}{N-p_0\a}}(\mathbb{R}^N)}\le C(\O)\bigg\|\dfrac{w(.,t)}{\delta^s}\bigg\|_{L^{p_0}(\O)}.
        $$
        Fixed now $p_0$ such that
        $p_0>\frac{Np}{N+\a p-pN(\rho-s+\a)}$, then by the generalized H\"older inequality, we deduce
        that $|(-\Delta)^{\frac{\rho}{2}} w(,t)|\in L^a(\O_2)$
        with $a=\frac{Np_0}{(\rho-s+\a)Np_0+N-\a p_0}$
        and
        $$
        \begin{array}{lll}
            \dyle \bigg(\int_{\Omega_2}|(-\Delta)^{\frac{\rho}{2}} w(x,t)|^a
            dx\bigg)^{\frac 1a}  &\le &\dyle
            C\bigg\|\frac{w(.,t)}{\d^s}\bigg\|_{L^{p_0}(\O)}\le
            C(\O)t^{-\frac{N}{2s}(\frac{1}{\s}-\frac{1}{p_0})-\frac{1}{2}}\|w_0\|_{L^{\s}(\O)}.
        \end{array}
        $$
        Notice that, taking into consideration the value of $p_0$, we
        deduce that $p\le a$. Hence,  choosing $\theta=p$ and $r=p_0$ in
        \eqref{eqrr}, we deduce that
        $|(-\Delta)^{\frac{\rho}{2}} w(x,t)|\chi_{\{x\in \mathbb{R}^N\backslash \Omega\}}\in {L^{p}(\mathbb{R}^N\backslash\Omega)}$ and
        $$
        ||(-\Delta)^{\frac{\rho}{2}} w(.,t)||_{L^{p}(\mathbb{R}^N\backslash \Omega)}\le C(\O)  \bigg\|\dfrac{w}{\delta^s}(.,t)\bigg\|_{L^{p_0}(\Omega)}\le C t^{-\frac{N}{2s}(\frac{1}{\s}-\frac{1}{p_0})-\frac{1}{2}}||w_0||_{L^\s(\Omega)}.
        $$
        Now, if $\rho=s$, then $\frac{Np}{N+\a
            p-pN(\rho-s+\a)}=\frac{Np}{N-\a p(N-1)}\to p$ as $\a\to 0$.
        Hence,
        we deduce that for all $p>\s$ and for all $\eta>0$, we have
        $$
        ||(-\Delta)^{\frac{\rho}{2}} w(.,t)||_{L^{p}(\mathbb{R}^N\backslash \Omega)}\le C(\O)  \bigg\|\dfrac{w}{\delta^s}(.,t)\bigg\|_{L^{p+\eta}(\Omega)}\le C t^{-\frac{N}{2s}(\frac{1}{\s}-\frac{1}{p+\eta})-\frac{1}{2}}||w_0||_{L^\s(\Omega)}.
        $$
        Therefore, if $p<\widehat{p_s}:=\dfrac{\s(N+2s)}{N+\s s}$, then
        $|(-\Delta)^{\frac{s}{2}} w| \in {L^p((\mathbb{R}^N\backslash
            \Omega) \times (0,T))}$ and

{    $$
        ||(-\Delta)^{\frac{\rho}{2}}
        w(.,t)||_{L^{p}(\mathbb{R}^N\backslash \Omega)}\le
        C(\O)t^{-\frac{N}{2s}(\frac{1}{\s}-\frac{1}{p_0})-\frac{1}{2}}\|w_0\|_{L^{\s}(\O)},
        $$
    and
        $$
        ||(-\Delta)^{\frac{\rho}{2}} w||_{L^{p}((\mathbb{R}^N\backslash \Omega)\times (0,T))}\le C(\O_T)
        ||w_0||_{L^\s(\Omega)}.
        $$}
    \end{proof}

    By combining Theorem  \ref{reg_frac_u0_t} and   Proposition
    \ref{ext1}, we get the next fractional regularity result in the
    Bessel potential space.
    \begin{Theorem}\label{globalu_0_t}
        Assume that $w_0\in L^{\s}(\Omega)$ with $\s\geq 1$,  let $w$ be the unique solution to  Problem \eqref{eq_linear_01}. Then,
        $w(.,t)\in  \mathbb{L}_0^{s,p}(\Omega)$ for all $p\geq \s$ and
        $$
        ||w(.,t)||_{\mathbb{L}_0^{s,p}(\Omega)}\le
        \left(t^{-\frac{N}{2s}(\frac{1}{\s}-\frac{1}{p})-\frac{1}{2}}+
        t^{-\frac{N}{2s}(\frac{1}{\s}-\frac{1}{p(1+\eta)})-\frac{\rho}{2s}}+t^{\frac{-N}{2s}(\frac{1}{\s}-\frac{1}{p})-\frac{s+\eta}{2s}}\right)
        ||w_0||_{L^\s(\O)};
        $$
        \begin{enumerate}
            \item If $3s\ge 1$, then $w(.,t)\in  \mathbb{L}_0^{s,p}(\Omega)$
            for all $p>\s$. Moreover, for all $\eta>0$ small enough, we have
            $$
            ||w(.,t)||_{\mathbb{L}_0^{s,p}(\Omega)}\le C(\O)
            t^{-\frac{N}{2s}(\frac{1}{{\s}}-\frac{1}{p})-\frac{1}{2}}\\
            \bigg(t^{-\frac{\eta}{2s}} +t^{-\frac{N\eta}{2sp(1+\eta)}} +
            t^{-\frac{N}{2s}(\frac 1p-\frac1{p_0})}
            \bigg)||w_0||_{L^{{\s}}(\O)}.
            $$

            \item If $3s<1$, then setting $\s_0<\min\{{\s}, \frac{N}{1-3s}\}$,
            $w(.,t)\in  \mathbb{L}_0^{s,p}(\Omega)$ for all
            $p>\frac{\s_0N}{N+\s_0(3s-1)}$. In addition, we have
            $$
            ||w(.,t)||_{\mathbb{L}_0^{s,p}(\Omega)}\le C(\O)
            t^{-\frac{N}{2s}(\frac{1}{{\s_0}}-\frac{1}{p})-\frac{1}{2}}\\
            \bigg(t^{-\frac{\eta}{2s}} +t^{-\frac{N\eta}{2sp(1+\eta)}} +
            t^{-\frac{N}{2s}(\frac 1p-\frac1{p_0})}
            \bigg)||w_0||_{L^{\widehat{\s}}(\O)}.
            $$
        \end{enumerate}
        Moreover, if $1\le \s<\frac{N}{(1-3s)_+}$, then $w\in L^p(
        0,T;\mathbb{L}_0^{s,p}(\Omega))$ for all $p<\frac{\s(N+2s)}{N+\s
            s}$ and
        $$
        ||w||_{L^p( 0,T;\mathbb{L}_0^{s,p}(\Omega))}\le
        C(\O_T)||w_0||_{L^\s(\O)}.$$
    \end{Theorem}

    \begin{remarks}
        According to the {   regularity} results obtained in Theorem
        \ref{reg_frac_u0_t} and   Proposition \ref{ext1}, we get the existence of $\e>0$ small such that the results of
        Theorem \ref{globalu_0_t} holds in the Bessel potential space
        $\mathbb{L}_0^{s+\e,p}(\Omega)$. More precisely, we have
        {$$
        ||w(.,t)||_{\mathbb{L}_0^{s+\e,p}(\Omega)}\le C(t,\O)||w_0||_{L^\s(\O)}.
        $$}
        Furthermore, if $1\le \s<\frac{N}{(1-3s)_+}$, then $w\in L^p(
        0,T;\mathbb{L}_0^{s+\e,p}(\Omega))$ for all
        $p<\frac{\s(N+2s)}{N+\s s}$ and
     {   $$||w||_{L^p(
            0,T;\mathbb{L}_0^{s+\e,p}(\Omega))}\le C(\O_T)||w_0||_{L^\s(\O)},$$
        where $C(\O_T)\to 0$ as $T\to 0$.}
    \end{remarks}

    Now, taking into consideration the relation between the {potential
    Bessel space and the fractional Sobolev space}, we get the following
    result.

    \begin{Corollary}\label{coryy}
        Assume that $w_0\in L^{\s}(\Omega)$ with $1\le \s<\frac{N}{(1-3s)_+}$, then
        $w\in L^p(0,T; \W^{s,p}_0(\O))$ for all $p<\frac{\s(N+2s)}{N+\s
            s}$ and $||w||_{L^p(
            0,T;\W_0^{s,p}(\Omega))}\le C(\O_T)||w_0||_{L^\s(\O)}$.
    \end{Corollary}

    \vspace{0.2cm}
    \subsection{\textbf{Case: $w_0=0$ and $h\in L^m(\O_T)$}}
    For this case, we treat mainly  the following {Problem}
    \begin{equation}\label{eq_linear_02}
        \left\{
        \begin{array}{llll}
            w_t+(-\Delta)^s w&=& h
            & \text{in}\quad\Omega_T,\vspace{0.2cm}\\
            w(x,t)&=&0&\text{in} \quad (\mathbb{R}^N\setminus\Omega)\times(0,T),\vspace{0.2cm}\\
            w(x,0)&=&0&\text{in}\quad    \Omega,
        \end{array}
        \right.
    \end{equation}
    As in the first case, we start this part by giving the next
    theorem which provides the regularity of some kind of integrals.

    \begin{Theorem}\label{integrals regularityII}
        Assume that $g\in L^m(\O_T)$ with {$m\geq1$} and define
        $$
        G_{\a,\l}(x,t)=\int_0^t\io
        \frac{g(y,\tau)(t-\tau)^{\a}}{((t-\tau)^{\frac{
                    1}{2s}}+|x-y|)^{N+\l}} dy d\tau,
        $$
        where $\a,\l \in \mathbb{R}$ are such that $\a>-1,
        2s(\a+1)>\l$ and $N+\l>\max\{0,2s\a\}$. We have:

        \noindent $\bullet$ If $\l\ge 0$, then:
        \begin{enumerate}
            \item If $m>\frac{N+2s}{2s(\a+1)-\l}$, then $G_{\a,\l}\in
            L^\infty(\O_T)$ and
            $$
            ||G_{\a,\l}||_{L^\infty(\O_T)}\le
            CT^{\frac{2s(\a+1)-\l}{2s}-\frac{1}{m}\frac{N+2s}{2s}}||g||_{L^m(\O_T)}.
            $$
            \item If $m\le \frac{N+2s}{2s(\a+1)-\l}$, then $G_{\a,\l}\in
            L^r(\O_T)$ for $\ell$ satisfies
            $m<r<\frac{m(N+2s)}{N+2s-m(2s(\a+1)-\l)}$. Moreover,
            \begin{equation*}
                || G_{\a,\l}||_{L^{r}(\Omega_T)}
                \leq C(\O) T^{\frac{2s(\a+1)-\l}{2s}-\frac{N+2s}{2s}(\frac{1}{m}-\frac{1}{r})}||g||_{L^{m}(\O_T)}.
            \end{equation*}
        \end{enumerate}
        $\bullet$ If $\l<0$, for all $r$ satisfies
        $m<r<\frac{m(N+2s)}{N+2s-2sm(\a+1)}$, we have
        \begin{equation*}\label{ZA1}
            || G_{\a,\l}||_{L^{r}(\Omega_T)}
            \leq C(T^{-\frac{\l}{2s}}+1)T^{\a+1-\frac{N+2s}{2s}(\frac{1}{m}-\frac{1}{r})}||g||_{L^{m}(\O_T)}.
        \end{equation*}
        $\bullet$ If $\l<0$ and $|\l|<N(\a+1)$, choosing $m_0$ such that
        $1\le m_0<\min\{m,-\frac{N}{\l}\}$, for
        $\frac{m_0N}{N+m_0\l}<r<\frac{m_0(N+2s)}{(N+2s-m_0(2s(\a+1)-\l))_+}$,
        we have
        \begin{equation*}
            || G_{\a,\l}||_{L^{r}(\Omega_T)}
            \leq
            C(\O)T^{\frac{2s(\a+1)-\l}{2s}-\frac{N+2s}{2s}(\frac{1}{m_0}-\frac{1}{r})}
            ||g||_{L^{m_0}(\Omega_T)}\le
            C(\O)T^{\frac{1}{m_0}-\frac{1}{m}+\frac{2s(\a+1)-\l}{2s}-\frac{N+2s}{2s}(\frac{1}{m_0}-\frac{1}{r})}
            ||g||_{L^{m}(\Omega_T)}.
        \end{equation*}
    \end{Theorem}
    We refer to the Appendix for a detailed proof. See also Remark
    \ref{lastr} for the case $\l\ge 2s(\a+1)$.

    \

    Now, we are in position to state our first main result about the
    regularity of the fractional gradient of $w$, the solution {of}
    \eqref{eq_linear_02}.
    \begin{Theorem}\label{regug}
        Assume that $s>\frac 14$ and fixed $s\le \rho<\max\{2s,1\}$ be
        such that $\rho-s<\min\{\frac{s}{(N+2s)},
        \frac{4s-1}{(N+2s-1)}\}$.

        Let $h\in L^{m}(\Omega_T)$ where $m\ge 1$ and define $w$ to be the
        unique solution to Problem \eqref{eq_linear_02}. Then, for $\eta>0$ small enough, we
        have:
        \begin{enumerate}
            \item If $2s+\rho\ge 1$, then, \\
            $\bullet$   If $1\leq m\le \frac{N+2s}{2s-\rho}$, then  $(-\Delta)^{\frac{\rho}{2}} w\in {L^{r}(\Omega_T)}$  for all $m<r<\overline{\overline{m}}_{s,\rho}:=
            \frac{m(N+2s)}{(N+2s)(m(\rho-s)+1)-ms}$. Moreover, we have
            \begin{equation*}
                ||(-\Delta)^{\frac{\rho}{2}} w||_{L^{r}(\Omega_T)}\leq CT^{\frac
                    12-\frac{N+2s}{2s}(\frac{1}{m}-\frac{1}{r})}(1+
                T^{-\frac{\rho+\eta-s}{2s}}+T^{-\frac{\rho-s}{2s}})||h||_{L^m(\Omega_T)}.
            \end{equation*}
            $\bullet$   If $m>\frac{N+2s}{2s-\rho}$, then
            $(-\Delta)^{\frac{\rho}{2}} w\in {L^{r}(\Omega_T)}$  for all
            $r<\frac{1}{\rho-s}$ and
            $$
            ||(-\Delta)   ^{\frac{\rho}{2}} w||_ {L^{r}(\Omega_T)}\le
            CT^{\frac
                12-\frac{N+2s}{2sm}}(1+
            T^{-\frac{\rho+\eta-s}{2s}}+T^{-\frac{\rho-s}{2s}})||h||_{L^m(\Omega_T)}.
            $$
            \item If $2s+\rho<1$, then

            $\bullet$   If $m\le \frac{N+2s}{2s-\rho}$, then $(-\Delta)^{\frac{\rho}{2}} w\in
            {L^{r}(\Omega_T)}$ for all
            $m<r<\frac{m(N+2s)}{(N+2s)(m(\rho-s)+1)-m(3s+\rho-1)}$and
            \begin{equation*}
                \begin{array}{lll}
                    ||(-\Delta)^{\frac{\rho}{2}} w||_{L^{r}(\Omega_T)}&\leq &
                    C(T^{\frac{1-2s-\rho}{2s}}+1)T^{-\frac{N+2s}{2s}(\frac{1}{m}-\frac{1}{r})}\\
                    &\times & \bigg(T^{\frac 12} + T^{\frac{4s-1-\eta}{2s}}+
                    T^{\frac{3s+\rho-1}{2s}} + T^{\frac{4s-1}{2s}}
                    +T^{\frac{2s-\rho}{2s}}\bigg) ||h||_{L^m(\Omega_T)}.
                \end{array}
            \end{equation*}
            $\bullet$ If $m>\frac{N+2s}{4s-1}$, then $(-\Delta)^{\frac{\rho}{2}} w\in {L^{r}(\Omega_T)}$
            all $r<\frac{1}{\rho-s}$, moreover, for $\eta>0$ small enough, we have
            \begin{equation*}
                \begin{array}{lll}
                    ||(-\Delta)^{\frac{\rho}{2}} w||_{L^{r}(\Omega_T)}&\leq &
                    C(T^{\frac{1-2s-\rho}{2s}}+1)T^{\frac{1}{r}-\frac{N+2s}{2sm}}\\
                    &\times & \bigg(T^{\frac 12} + T^{\frac{4s-1-\eta}{2s}}+
                    T^{\frac{3s+\rho-1}{2s}} + T^{\frac{4s-1}{2s}}
                    +T^{\frac{2s-\rho}{2s}}\bigg) ||h||_{L^m(\Omega_T)}.
                \end{array}
            \end{equation*}
        \end{enumerate}
        In any case, we have
        \begin{equation*}
            ||(-\Delta)^{\frac{\rho}{2}} w||_{L^{r}(\Omega_T)}\leq C(\O,
            T)||h||_{L^m(\Omega_T)}.
        \end{equation*}
     \end{Theorem}
    \begin{proof}
        Without loss of generality and using the linearity of the problem,
        we can assume that $h\gneqq 0$. We give a detailed proof in the
        first case $1\leq m\le \frac{N+2s}{2s-\rho}$, the other case
        follows in the same way.

        According to the sign of $2s+\rho-1$ and the value of $s$, we will
        divide the proof into three cases:

{   {\bf Case: $2s+\rho\ge 1$ and $\frac 14<s\le \frac 12$.}}

        Let $w$ be the unique solution to Problem \eqref{eq_linear_02},
        then $w$ is
        given by
        \begin{equation*}
            w(x,t)=\iint_{\Omega_t}h(y,\tau)P_\O(x,y,t-\tau)dyd\tau.
        \end{equation*}
        Hence
 {   \begin{equation}
            |(-\Delta)   ^{\frac{\rho}{2}} w(x,t)|\leq\iint_{\Omega_t}h(y,\tau)|(-\Delta)_x ^{\frac{\rho}{2}}P_\O(x,y,t-\tau)|dyd\tau.\vspace{0.2cm}\\
        \end{equation}}

        Using estimate \eqref{gradp12}, we reach that

        \begin{equation*}
            |(-\Delta)   ^{\frac{\rho}{2}} w(x,t)|\leq C \left(K_1(x,t)+K_2(x,t)+K_3(x,t)+K_4(x,t)+K_5(x,t)\right),
        \end{equation*}
        with
        $$K_1(x,t):=\iint_{\Omega_t} \dfrac{h(y,\tau)}{\left((t-\tau)^{\frac{1}{2s}}+|x-y|\right)^{N+s}}
        d\tau dy, $$
        $$K_2(x,t):=\iint_{\Omega_t} \dfrac{h(y,\tau)(t-\tau)^{\frac{2s-1}{2s}}}{\left((t-\tau)^{\frac{1}{2s}}+|x-y|\right)^{N+2s+\rho-1}}
         {\log\bigg(\frac{D}{|x-y|}\bigg)} d\tau dy,$$
        $$K_3(x,t):=\d^{s-\rho}(x)\iint_{\Omega_t} \dfrac{h(y,\tau)
            (t-\tau)^{\frac{s+\rho-1}{2s}}}{
            \left((t-\tau)^{\frac{1}{2s}}+|x-y|\right)^{N+2s+\rho-1}} d\tau
        dy,$$
        $$K_4(x,t):=|\log(\d(x))|\iint_{\Omega_t} \dfrac{h(y,\tau)
            (t-\tau)^{\frac{2s-1}{2s}}}{
            \left((t-\tau)^{\frac{1}{2s}}+|x-y|\right)^{N+2s+\rho-1}} d\tau
        dy,$$ and
        $$K_5(x,t):=\iint_{\Omega_t} \dfrac{h(y,\tau)}{|x-y|^{1-2s}(
            (t-\tau)^{\frac{1}{2s}}+|x-y|)^{N+2s+\rho-1}} d\tau dy,$$

        Since $s>\frac 14$, then $\frac{2s-1}{2s},\frac{s+\rho-1}{2s}>-1$. To estimate $K_1$, we use Theorem \ref{integrals
            regularityII}
        with $\a=0$ and $\l=s$, then $2s(\a+1)-\l=s>0$. Thus,
        $K_1\in L^r(\O_T)$ for all
        $m<r<\overline{m}_{s}:=\frac{m(N+2s)}{N+2s-ms}$ and
        $$
        ||K_1||_{L^r(\O_T)}\le
        T^{\g_1+\frac{1}{a}}||h||_{L^m(\O_T)},
        $$
        where $\gamma_1=\frac{N}{2sa}-\frac{N+s}{2s}$ and
        $a=\frac{m'r}{m'+r}$. Hence
        \begin{equation}\label{KG1}
            ||K_1||_{L^r(\O_T)}\le
            T^{\frac 12-\frac{N+2s}{2s}(\frac{1}{m}-\frac{1}{r})}||h||_{L^m(\O_T)},
        \end{equation}
        We now consider the term $K_2$. Notice that for all $0<\eta<<1$,
        there exists a positive constant $C$ depends only on $\O$ such
        that
        $ \log(\frac{D}{|x-y|})\le \frac{C}{|x-y|^\eta}$  {a.e}  in $\O\times
        \O$. We follow the duality argument used in the proof of
        Theorem \ref{integrals regularityII}. More precisely, we have
        \begin{equation*}
            \begin{array}{lll}
                ||K_2||_ {L^{r}(\Omega_T)}&=\dyle \sup\limits_{\{|| \psi||_{L^{r'}(\Omega_T)}\leq 1\}}{\iint_{\Omega_T}}\iint_{\Omega_t}
                |\psi(x,t)|\dfrac{h(y,\tau)(t-\tau)^{\frac{2s-1}{2s}}}{ \left((t-\tau)^{\frac{1}{2s}}+|x-y|\right)^{N+2s+\rho-1}}
             {   \log\bigg(\frac{D}{|x-y|}\bigg)}d\tau dydxdt,\vspace{0.2cm}\\
                & \leq C\sup\limits_{\{|| \psi|| _{L^{r'}(\Omega_T)}\leq 1\}}\dyle
                \iint_{\Omega_T}\iint_{\Omega_t}|\psi(x,t)|  \dfrac{h(y,\tau)(t-\tau)^{\frac{2s-1}{2s}}}{|x-y|^{\eta}
                    ((t-\tau)^{\frac{1}{2s}}+|x-y|)^{N+2s+\rho-1}}
                d\tau dydxdt.\vspace{0.2cm}\\
            \end{array}
        \end{equation*}
        Setting
        $$\widehat{H}(x,t)=\dfrac{t^{\frac{2s-1}{2s}}|x|^{{-\eta}}}{(t^{\frac{1}{2s}}+|x|)^{N+2s+\rho-1}}.$$
        Then
        $$
        ||K_2||_ {L^{r}(\Omega_T)}\le \sup\limits_{\{|| \psi|| _{L^{r'}(\Omega_T)}\leq 1\}}\iint_{\Omega_T}\iint_{\Omega_t}|\psi(x,t)| h(y,\tau)\widehat{H}(t-\tau, x-y)d\tau dydxdt.
        $$
        By using Young's inequality, it holds that
        \begin{equation*}
            ||K_2||_ {L^{r}(\Omega_T)}\leq  \sup\limits_{\{|| \psi|| _{L^{r'}(\Omega_T)}\leq 1\}}\Int_0^T ||\psi(.,t)||_{L^{r'}(\Omega) }  \Int_0^t ||h(.,\tau)||_{L^{m}(\Omega)} ||\widehat{H}(.,t-\tau)||_{L^{\bar{a}}(\Omega)}d\tau dt,
        \end{equation*}
        with
        \begin{equation*}\label{l33}
            \dfrac{1}{r'}+\dfrac{1}{m}+\dfrac{1}{\bar a}=2.
        \end{equation*}
        Notice that $\overline{a}=\frac{rm'}{r+m'}=a${, choosing }${\eta}$ small enough such that $\bar a{\eta}<N$, it holds that
        $$ ||\widehat{H}(.,t-\tau)||_{L^{\bar{a}}(\Omega)}\leq
        C (t-\tau)^{\frac{2s-1}{2s}+\frac{N}{2s\bar{a}}-\frac{N+2s+\rho-1}{2s}-\frac{{\eta}}{2s}} \le
        C(t-\tau)^{\frac{N}{2s\bar{a}}-\frac{N+\rho}{2s}-\frac{{\eta}}{2s}}.$$
        Thus, using again Young's inequality, we deduce that
        \begin{eqnarray*}
            ||K_2||_ {L^{r}(\Omega_T)} &\leq & \sup\limits_{\{|| \psi||_{L^{r'}(\Omega_T)}\leq 1 \}}\Int_0^T\Int_0^T ||\psi(.,t)||_{L^{r'}(\Omega) }
            ||h(.,\tau)||_{L^{m}(\Omega)} |t-\tau|^{\frac{N}{2s\bar{a}}-\frac{N+\rho}{2s}-\frac{{\eta}}{2s}}d\tau dt\\
            &\le & \sup\limits_{\{|| \psi||_{L^{r'}(\Omega_T)}\leq 1\}}||\psi||_{L^{r'}(\Omega_T)}
            ||h||_{L^{m}(\Omega_T)} \bigg(\int_0^T t^{\bar{a}(\frac{N}{2s\bar{a}}-\frac{N+\rho+\eta}{2s})}dt\bigg)^{\frac{1}{a}}.
        \end{eqnarray*}
        Since, for $\eta$ small enough, choosing $r<\overline{m}_{s,\rho}:=\frac{m(N+2s)}{N+2s-m(2s-\rho)}$, we reach that
        $\bar{a}|\frac{N}{2s\bar{a}}-\frac{N+\rho+\eta}{2s}|<1$.
        Then,  the last integral is finite. Thus, $K_2\in L^r(\O_T)$ for all $m<r<\overline{m}_{s,\rho}$ and
        \begin{equation}
            ||K_2||_ {L^{r}(\Omega_T)}\leq C
            T^{{\g_2}+\frac{1}{\bar{a}}}||h||_{L^{m}(\Omega_T)},
        \end{equation}
        where
        ${\g_2}=\frac{N}{2s\overline{a}}-\frac{(N+\rho+{\eta})}{2s}=\g_1-\frac{\rho+\eta-s}{2s}$.

        We deal now with $K_3$. Notice that
        $K_3(x,t)=\d^{s-\rho}(x)\widehat{K}_3(x,t)$ where
        \begin{equation}\label{kchap}
            \widehat{K}_3(x,t)=\iint_{\Omega_t} \dfrac{h(y,\tau)
                (t-\tau)^{\frac{s+\rho-1}{2s}}}{
                \left((t-\tau)^{\frac{1}{2s}}+|x-y|\right)^{N+2s+\rho-1}} d\tau
            dy. \end{equation} Using Theorem \ref{integrals regularityII}
        with $\a=\frac{s+\rho-1}{2s}$ and $\l=2s+\rho-1$, then $2s(\a+1)-\l=s>0$. Thus,
        $\widehat{K}_3\in L^r(\O_T)$ for all {$r<\overline{m}_{s, \rho}$} and
        \begin{equation}\label{KG3}
            ||\widehat{K}_3||_{L^r(\O_T)}\le
            T^{\g_1+\frac{1}{a}}||h||_{L^m(\O_T)}.
        \end{equation}
        Now, since $\d^{s-\rho}\in L^\beta(\O)$ for all
        $\beta<\frac{1}{\rho-s}$, using the generalized H\"older's
        inequality, we deduce that $K_3\in L^{r_2}(\O_T)$ for all
        $r_2<\overline{\overline{m}}_{s,\rho}:=
        \frac{m(N+2s)}{(N+2s)(m(\rho-s)+1)-ms}<\overline{m}_{s,\rho}$.

        Since $\rho-s<\frac{s}{(N+2s)}$, then
        $\overline{\overline{m}}_{s,\rho}>m$. Hence, in this case we obtain
        that, for all $m<r<\overline{\overline{m}}_{s,\rho}$,
        \begin{equation}\label{K033}
            ||K_3||_{L^{r}(\Omega_T)}\leq C
            T^{\g_1+\frac{1}{a}}||h||_{L^m(\O_T)}.
        \end{equation}
        Respect to $K_4$, we have
        $K_4(x,t)=|\log(\d(x))|\widehat{K}_4(x,t)$ where
        $$
        \widehat{K}_4(x,t)=\iint_{\Omega_t} \dfrac{h(y,\tau)
            (t-\tau)^{\frac{2s-1}{2s}}}{
            \left((t-\tau)^{\frac{1}{2s}}+|x-y|\right)^{N+2s+\rho-1}} d\tau
        dy.$$ Using Theorem \ref{integrals regularityII}
        with $\a=\frac{2s-1}{2s}$ and $\l=2s+\rho-1$, then $2s(\a+1)-\l=2s-\rho>0$.
        Thus,
        $\widehat{K}_4\in L^r(\O_T)$ for all $r<\overline{m}_{s,\rho}$ and
        \begin{equation}\label{KG4}
            ||\widehat{K}_4||_{L^r(\O_T)}\le
            T^{\g_3+\frac{1}{a}}||h||_{L^m(\O_T)},
        \end{equation}
        where ${\g_3}=\frac{N}{2s{a}}-\frac{N+\rho}{2s}=\g_1-\frac{\rho-s}{2s}$.

        Since $\log(\delta(x)) \in L^{\xi}(\O)$ for all $\xi\in
        [1,\infty)$, we deduce that $K_4\in L^r(\O_T)$ for all
        $m<r<{\overline{m}_{s,\rho}}$ and
        \begin{equation}\label{K400}
            ||K_4||_{L^r(\O_T)}\leq C T^{\g_3+\frac{1}{a}} ||h||_{L^m(\O_T)}.
        \end{equation}

        To estimate $K_5$, we use the same duality argument used for
        $K_2$. Hence, we reach that Thus, $K_5\in L^r(\O_T)$ for all
        $m<r<\overline{m}_{s,\rho}$ and
        \begin{equation}\label{K55}
            ||K_5||_ {L^{r}(\Omega_T)}\leq C
            T^{{\g_3}+\frac{1}{{a}}}||h||_{L^{m}(\Omega_T)}.
        \end{equation}
        Combining the above estimates and taking into consideration that
        $\overline{\overline{m}}_{s,\rho}<{\overline{m}}_{s,\rho}<{\overline{m}}_{s}$,
        we conclude that, for all $m<r<\overline{\overline{m}}_{s,\rho}$, we
        have
        \begin{equation*}
            ||(-\Delta)^{\frac{\rho}{2}} w||_{L^{r}(\Omega_T)}\leq CT^{\frac
                12-\frac{N+2s}{2s}(\frac{1}{m}-\frac{1}{r})}(1+
            T^{-\frac{\rho+\eta-s}{2s}}+T^{-\frac{\rho-s}{2s}})||h||_{L^m(\Omega_T)}.
        \end{equation*}
  Hence, we conclude.

        As a consequence, we deduce that for all $1\le
        r<\overline{\overline{m}}_{s,\rho}${, we have}
        \begin{equation*}
            ||(-\Delta)^{\frac{\rho}{2}} w||_{L^{r}(\Omega_T)}\leq
            C(\O_T)||h||_{L^m(\Omega_T)},
        \end{equation*}
        where $C(\O_T)\to 0$ as $T\to 0$.

  {    {\bf Case:  $2s+\rho\ge 1$ and $s>\frac 12$.}} By
        estimate \eqref{gradp22}, we reach that
        \begin{equation*}
            |(-\Delta)   ^{\frac{\rho}{2}} w(x,t)|\leq C \left(K_1(x,t)+L_2(x,t)+L_3(x,t)+L_4(x,t)+L_5(x,t)\right),
        \end{equation*}
        where $K_1$ is already given in the first case,
        $$L_2(x,t):=\iint_{\Omega_t} \dfrac{h(y,\tau)(t-\tau)^{\frac{2s-1}{2s}}}{|x-y|^{2s-1}\left((t-\tau)^{\frac{1}{2s}}+|x-y|\right)^{N+\rho}}
        d\tau dy,$$
        $$L_3(x,t):=\d^{s-\rho}(x)\iint_{\Omega_t} \dfrac{h(y,\tau)
            (t-\tau)^{\frac{\rho-s}{2s}}}{
            \left((t-\tau)^{\frac{1}{2s}}+|x-y|\right)^{N+\rho}} d\tau dy,$$
        $$L_4(x,t):=\iint_{\Omega_t} \dfrac{h(y,\tau)}{ \left((t-\tau)^{\frac{1}{2s}}+|x-y|\right)^{N+\rho}}{\log\bigg(\frac{D}{|x-y|}\bigg) } d\tau
        dy,$$ and
        $$L_5(x,t):=|\log(\d(x))|\iint_{\Omega_t} \dfrac{h(y,\tau)}{(
            (t-\tau)^{\frac{1}{2s}}+|x-y|)^{N+\rho}} d\tau dy.$$

        Recall that $K_1\in L^r(\O_T)$ for all {$m<r<\overline{m}_{s,\rho}$} and
        \begin{equation*}
            ||K_1||_{L^r(\O_T)}\le
            T^{\g_1+\frac{1}{a}}||h||_{L^m(\O_T)}.
        \end{equation*}
        Now, using the same duality argument as in estimating $K_2$, it
        holds that $L_2\in L^r(\O_T)$ for all $m<r<\overline{m}_{s,\rho}$
        and
        \begin{equation}
            ||L_2||_ {L^{r}(\Omega_T)}\leq C
            T^{{\g_3}+\frac{1}{\bar{a}}}||h||_{L^{m}(\Omega_T)},
        \end{equation}
        We deal now with $L_3$. We set
        $$
        \widehat{L}_3(x,t)=\iint_{\Omega_t} \dfrac{h(y,\tau)
            (t-\tau)^{\frac{\rho-s}{2s}}}{
            \left((t-\tau)^{\frac{1}{2s}}+|x-y|\right)^{N+\rho}} d\tau dy.$$
        Using Theorem \ref{integrals regularityII}
        with $\a=\frac{\rho-s}{2s}$ and $\l=\rho$, then $2s(\a+1)-\l=s>0$.
        Thus, $\widehat{L}_3\in L^r(\O_T)$ for all {$m<r<\overline{m}_{s, \rho}$ } and
        \begin{equation}\label{LL3}
            ||\widehat{L}_3||_{L^r(\O_T)}\le
            T^{\g_1+\frac{1}{a}}||h||_{L^m(\O_T)}.
        \end{equation}
        Notice that $L_3(x,t)=\d^{s-\rho}(x)\widehat{L}_3(x,t)$. Since
        $\d^{s-\rho}\in L^\beta(\O)$ for all $\beta<\frac{1}{\rho-s}$,
        we deduce that $L_3\in L^{r_1}(\O_T)$ for all
        $m<r_1<\overline{\overline{m}}_{s,\rho}:=
        \frac{m(N+2s)}{(N+2s)(m(\rho-s)+1)-ms}<\overline{m}_{s,\rho}$,  and
        \begin{equation}\label{L033}
            ||L_3||_{L^{r}(\Omega_T)}\leq C
            T^{\g_1+\frac{1}{a}}||h||_{L^m(\O_T)}.
        \end{equation}

        Respect to $L_4$ and $L_5$, using the same approach as in
        estimating of $K_2$, it holds that for $L_4\in L^r(\O_T)$ for all
        $m<r<\overline{m}_{s,\rho}$ and
        \begin{equation}
            ||L_4||_ {L^{r}(\Omega_T)}\leq C
            T^{{\g_2}+\frac{1}{\bar{a}}}||h||_{L^{m}(\Omega_T)},
        \end{equation}
{   and
        \begin{equation}\label{L5}
            ||L_5||_{L^r(\O_T)}\le
            T^{\g_3+\frac{1}{a}}||h||_{L^m(\O_T)}.
        \end{equation}}

        Now, combining the above estimates, we conclude.

{\bf Case:   $2s+\rho<1$ and $\frac 14<s\le\frac 12$. } Since
        $\rho\ge s$, then, under the above condition{, we obtain} that
        $3s<1$. We follow closely the second point in Theorem
        \ref{integrals regularityII}. Notice that, related to $K_2$, we
        have
        $$K_2(x,t)\le C(T^{\frac{1-2s-\rho}{2s}}+1)
        \iint_{\Omega_t} \dfrac{h(y,\tau)(t-\tau)^{\frac{2s-1}{2s}}}{\left((t-\tau)^{\frac{1}{2s}}+|x-y|\right)^{N}}
{   \log\bigg(\frac{D}{|x-y|}\bigg)} d\tau dy.$$
        We apply now the duality argument used previously to reach that
        for $m<r<\frac{m(N+2s)}{N+2s-m(4s-1)}$, choosing $\eta>0$ small
        enough such that $r<\frac{m(N+2s)}{N+2s-m(4s-1-\eta)}$,  we have
        \begin{equation*}
            ||K_2||_ {L^{r}(\Omega_T)}\leq C(T^{\frac{1-2s-\rho}{2s}}+1)
            T^{\frac{4s-1-\eta}{2s}-\frac{N+2s}{2s}(\frac{1}{m}-\frac{1}{r})}||h||_{L^{m}(\Omega_T)},
        \end{equation*}
        In the same way, using H\"older's inequality and the fact that
        $\rho-s<\frac{4s-1}{(N+2s-1)}$, we reach that $K_3\in L^r(\O_T)$
        for all $m<r<\frac{m(N+2s)}{(N+2s)(m(\rho-s)+1)-m(3s+\rho-1)}$ and
        \begin{equation*}
            ||K_3||_ {L^{r}(\Omega_T)}\leq C(T^{\frac{1-2s-\rho}{2s}}+1)
            T^{\frac{3s+\rho-1}{2s}-\frac{N+2s}{2s}(\frac{1}{m}-\frac{1}{r})}||h||_{L^{m}(\Omega_T)},
        \end{equation*}
        For $K_4$, as in estimating of $K_2$, we have for all
        $m<r<\frac{m(N+2s)}{N+2s-m(4s-1)}$,
        \begin{equation*}
            ||K_4||_ {L^{r}(\Omega_T)}\leq C(T^{\frac{1-2s-\rho}{2s}}+1)
            T^{\frac{4s-1}{2s}-\frac{N+2s}{2s}(\frac{1}{m}-\frac{1}{r})}||h||_{L^{m}(\Omega_T)},
        \end{equation*}
        Related to $K_5$, for $m<r<\frac{m(N+2s)}{N+2s-m(2s-\rho)}$, we
        have
        \begin{equation*}
            ||K_5||_ {L^{r}(\Omega_T)}\leq C(T^{\frac{1-2s-\rho}{2s}}+1)
            T^{\frac{2s-\rho}{2s}-\frac{N+2s}{2s}(\frac{1}{m}-\frac{1}{r})}||h||_{L^{m}(\Omega_T)},
        \end{equation*}
    {   Thus,  combining} the above {estimates,  we} deduce that if $\frac
        14<s\le \frac 12$, for all
        $m<r<\frac{m(N+2s)}{(N+2s)(m(\rho-s)+1)-m(3s+\rho-1)}$, we have
        \begin{equation*}
            \begin{array}{lll}
                ||(-\Delta)^{\frac{\rho}{2}} w||_{L^{r}(\Omega_T)}&\leq &
                C(T^{\frac{1-2s-\rho}{2s}}+1)T^{-\frac{N+2s}{2s}(\frac{1}{m}-\frac{1}{r})}\\
                &\times & \bigg(T^{\frac 12} + T^{\frac{4s-1-\eta}{2s}}+
                T^{\frac{3s+\rho-1}{2s}} + T^{\frac{4s-1}{2s}}
                +T^{\frac{2s-\rho}{2s}}\bigg) ||h||_{L^m(\Omega_T)}.
            \end{array}
        \end{equation*}
    Hence, we conclude.
    \end{proof}

    It is clear that, under the condition of the theorem, we have
    \begin{equation*}
        ||(-\Delta)^{\frac{\rho}{2}} w||_{L^{r}(\Omega_T)}\leq
        C(\O,T)||h||_{L^m(\Omega_T)},
    \end{equation*}
    where $C(\O,T)\to 0$ as $T\to 0$.

    \begin{remarks}
        In the case $\frac 14<s<\frac 12$ and $2s+\rho-1<0$, we can closely follow
        the proof of the third point in Theorem \ref{integrals
            regularityII} to improve the above estimate on
        $(-\D)^{\frac{\rho}{2}}w$ under additional conditions on $N,s$ and
        $\rho$. More precisely, fixed $\rho,s$ such that
        $N>\frac{2s(1-2s-\rho)}{4s-1}$ and
        $\rho-s<\frac{N(4s-1)+2s(3s-1)}{(N+2s)(N-1)}$.

        Let $m_0$ be such that $1\le m_0<\min\{m,\frac{N}{1-2s-\rho}\}$,
        to estimate $K_2$ in this case, we follow the same duality
        argument as above. Choosing $\a=\frac{2s-1}{2s},
        \l=2s+\rho+\eta-1$, for $\eta$ small enough, then under the above
        condition on $N,s,\rho$, we have $N(\a+1)>|\l|$. Hence, we reach
        that for
        $\frac{m_0N}{N-m_0(1-2s-\rho)}<r<\frac{m_0(N+2s)}{(N+2s-m_0(2s-\rho)_+}=\overline{m}_{0,s,\rho}$
        and for all $\eta>0$ small enough, we have $K_2\in L^r(\O_T)$ and
        $$ ||K_2||_ {L^{r}(\Omega_T)}\le
        C(\O)T^{\frac{1}{m_0}-\frac{1}{m}+\frac{2s-\rho-\eta}{2s}-\frac{N+2s}{2s}(\frac{1}{m_0}-\frac{1}{r})}
        ||h||_{L^{m}(\Omega_T)}.
        $$
        Related to $\widehat{K}_3${, using} again Theorem \ref{integrals
            regularityII} with $\a=\frac{s+\rho-1}{2s}$ and $\l=2s+\rho-1$,
        then $2s(\a+1)-\l=s>0$. {Hence, } for
        $\frac{m_0N}{N-m_0(1-2s-\rho)}<r<\frac{m_0(N+2s)}{(N+2s-m_0s)_+}=\overline{m}_{0,s}$,
        we have $\widehat{K}_3\in L^{r}(\O_T)$ and
        \begin{equation}\label{00KG3}
            ||\widehat{K}_3||_{L^{r}(\O_T)}\le
            C(\O)T^{\frac{1}{m_0}-\frac{1}{m}+\frac{1}{2}-\frac{N+2s}{2s}(\frac{1}{m_0}-\frac{1}{r})}
            ||h||_{L^{m}(\Omega_T)}.
        \end{equation}
        Since $\d^{s-\rho}\in L^\beta(\O)$ for all
        $\beta<\frac{1}{\rho-s}$, using the generalized H\"older's
        inequality{, we deduce} that $K_3\in L^{r_2}(\O_T)$ for all
        $r_2<\overline{\overline{m}}_{0,s,\rho}:=
        \frac{m_0(N+2s)}{(N+2s)(m_0(\rho-s)+1)-m_0s}<\overline{m}_{0,s,\rho}$.
        Notice that under the above conditions on $N,s,\rho$, we have
        $$m_0<\frac{m_0N}{N-m_0(1-2s-\rho)}<\overline{\overline{m}}_{0,s,\rho}.$$
        Therefore, it holds that for all
        $\frac{m_0N}{N-m_0(1-2s-\rho)}<r<\overline{\overline{m}}_{0,s,\rho}$,
        \begin{equation}\label{00K033}
            ||K_3||_{L^{r}(\Omega_T)}\leq
            C(\O)T^{\frac{1}{m_0}-\frac{1}{m}+\frac{1}{2}-\frac{N+2s}{2s}(\frac{1}{m_0}-\frac{1}{r})}
            ||h||_{L^{m}(\Omega_T)}.
        \end{equation}
        Respect to $\widehat{K}_4$, choosing $\a=\frac{2s-1}{2s}$ and
        $\l=2s+\rho-1$, then $2s(\a+1)-\l=2s-\rho>0$. {Hence, } by using
        Theorem \ref{integrals regularityII}, for all $r$ such that
        $\frac{m_0N}{N+m_0(2s+\rho-1)}<r<\overline{m}_{0,s,\rho}$, we have
        \begin{equation*}
            || \widehat{K}_4||_{L^{r}(\Omega_T)}
            \leq  C(\O)T^{\frac{1}{m_0}-\frac{1}{m}+\frac{2s-\rho}{2s}-\frac{N+2s}{2s}(\frac{1}{m_0}-\frac{1}{r})}
            ||h||_{L^{m}(\Omega_T)}.
        \end{equation*}
        Since $\log(\delta(x)) \in L^{\xi}(\O)$ for all $\xi\in
        [1,\infty)$, we deduce that $K_4\in L^r(\O_T)$ for all
        $\frac{m_0N}{N+m_0\l}<r<\overline{m}_{0,s,\rho}$ and
        \begin{equation}\label{00K400}
            ||K_4||_{L^r(\O_T)}\leq  C(\O)T^{\frac{1}{m_0}-\frac{1}{m}+\frac{2s-\rho}{2s}-\frac{N+2s}{2s}(\frac{1}{m_0}-\frac{1}{r})}
            ||h||_{L^{m}(\Omega_T)}.
        \end{equation}
        Therefore, combining the above estimates and using the fact that
        $\overline{\overline{m}}_{0,s,\rho}<{\overline{m}}_{0,s,\rho}<{\overline{m}}_{0,s}$,
        we conclude that for all
        $\frac{m_0N}{N-m_0(1-2s-\rho)}<r<\overline{\overline{m}}_{0,s,\rho}$,
        we have
        \begin{equation*}
            \begin{array}{lll}
                & & ||(-\Delta)^{\frac{\rho}{2}} w||_{L^{r}(\Omega_T)}\leq C(\O)T^{\frac{1}{2}+\frac{1}{m_0}-\frac{1}{m}-\frac{N+2s}{2s}(\frac{1}{m_0}-\frac{1}{r})}\\
                & &\bigg(1+T^{\frac{N}{2s}(\frac{1}{m_0}-\frac{1}{m})} +
                T^{-\frac{\rho+\eta-s}{2s}} +T^{-\frac{\rho-s}{2s}}+
                T^{\frac{N}{2s}(\frac{1}{m_0}-\frac{1}{m})-\frac{\rho-s}{2s}}
                \bigg)||h||_{L^m(\Omega_T)}. \end{array}
        \end{equation*}
        Notice that if $s>\frac 13$ or $N>\frac{(1-3s)2s}{4s-1}$, then the
        above condition on $N$ holds.

    \end{remarks}

    \begin{remarks}
        Related to the case $\frac 14<s<\frac 13$, $2s+\rho<1$ and
        $m>\frac{N+2s}{4s-1}$. Under the above hypotheses, we have
        $\frac{N+2s}{4s-1}>\max\{\frac{N+2s}{2s-\rho},\frac{N+2s}{s}\}$.
    Hence, we can apply the first case in Theorem \ref{integrals
            regularityII} to conclude. Using the same argument, we can treat also the case
        $\max\{\frac{N+2s}{2s-\rho},\frac{N+2s}{s}\}<m<\frac{N+2s}{4s-1}$.
    \end{remarks}

    \begin{Corollary}\label{CorR}
        In the particular case $\rho=s$, we have
        \begin{enumerate}
            \item If $s>\frac{1}{3}$, for $1\leq m\le \frac{N+2s}{s}$, then
            $(-\Delta)^{\frac{s}{2}} w\in {L^{r}(\Omega_T)}$  for all
            $m<r<\overline{m}_s=
            \frac{m(N+2s)}{N+2s-ms}$, for all $\eta>0$ small enough,
            we have
            \begin{equation*}
                ||(-\Delta)^{\frac{s}{2}} w||_{L^{r}(\Omega_T)}\leq CT^{\frac
                    12-\frac{N+2s}{2s}(\frac{1}{m}-\frac{1}{r})}(1+
                T^{-\frac{\eta}{2s}})||h||_{L^m(\Omega_T)}.
            \end{equation*}

            \item If $s>\frac 13$ and $m>\frac{N+2s}{s}$, then
            $(-\Delta)^{\frac{s}{2}} w\in {L^{r}(\Omega_T)}$  for all $r<
            \infty$ and
            $$
            ||(-\Delta)   ^{\frac{s}{2}} w||_ {L^{r}(\Omega_T)}\le
            CT^{\frac
                12-\frac{N+2s}{2sm}}(1+
            T^{-\frac{\eta}{2s}})||h||_{L^m(\Omega_T)}.
            $$
            \item If $\frac 14<s\le \frac 13$, then for
            $m<r<\frac{m(N+2s)}{N+2s-m(4s-1)}$, we have
            \begin{equation*}
                \begin{array}{lll}
                    ||(-\Delta)^{\frac{s}{2}} w||_{L^{r}(\Omega_T)}\leq
                    C(T^{\frac{1-3s}{2s}}+1)T^{-\frac{N+2s}{2s}(\frac{1}{m}-\frac{1}{r})}\bigg(T^{\frac
                        12} + T^{\frac{4s-1-\eta}{2s}}+ T^{\frac{4s-1}{2s}}\bigg)
                    ||h||_{L^m(\Omega_T)}.
                \end{array}
            \end{equation*}
            \item If $\frac 14<s\le \frac 13$ and $m>\frac{N+2s}{4s-1}$, then
            for all $r<\infty$, we have
            \begin{equation*}
                \begin{array}{lll}
                    ||(-\Delta)^{\frac{s}{2}} w||_{L^{r}(\Omega_T)}&\leq &
                    C(T^{\frac{1-3s}{2s}}+1)T^{\frac{1}{r}-\frac{N+2s}{2sm}}\bigg(T^{\frac
                        12} + T^{\frac{4s-1-\eta}{2s}}+ T^{\frac{4s-1}{2s}}\bigg)
                    ||h||_{L^m(\Omega_T)}.
                \end{array}
            \end{equation*}

        \end{enumerate}

    \end{Corollary}
Regarding the regularity of $w$ outside of $\O$, following the
    same argument as in the proof of Proposition \ref{ext1}, we get
    the next estimates.
    \begin{Proposition}\label{exttt}
        Assume that $h\in L^{m}(\Omega_T)$ with $m\geq 1$ and let $w$ be the unique weak  solution to Problem \eqref{eq_linear_02}.
        Fixed $\rho\ge s$ such that $\rho-s<\min\{\frac{s}{(N+2s)},
        \frac{4s-1}{(N+2s-1)}\}$ and recall that
        $\overline{m}_s=\frac{m(N+2s)}{N+2s-ms}$. Then there exists a positive constant $C$ depends only on the data and it is
        independent  of $h$ and $w$ such that

 $\bullet$   If $m>\frac{N+2s}{s}$, then  $(-\Delta)^{\frac{\rho}{2}} w\in {L^{p}((\mathbb{R}^N\backslash \Omega)\times(0,T))}$
        for all $p<\frac{1}{\rho-s}$ and
        $$
        ||(-\Delta)^{\frac{\rho}{2}} w||_{L^{p} {((\mathbb{R}^N\backslash
            \Omega)\times (0,T))}}\le C(\O,N,s,p)T^{\frac{1}{p}+\frac
            12-\frac{N+2s}{2sm}}||h||_{L^m(\O_T)}.
        $$
$\bullet$   If $1\leq m\le \frac{N+2s}{s}$, for $\a>0$ small
enough, we define $\overbrace{m_s}$
        by
        $$
        \overbrace{m_s}:=\frac{N
            \overline{m}_s}{(\rho-s+\a)N\overline{m}_s+N-\a
            \overline{m}_s}=\frac{Nm(N+2s)}{m(N+2s)(N(\rho-s+\a)-\a)+N(N+2s-ms)},
        $$
Then, for all $p<\overbrace{m_s}$, $(-\Delta)^{\frac{\rho}{2}}
w\in {L^{p}((\mathbb{R}^N\backslash
            \Omega)\times(0,T))}$ and for all
        $r\in (p,\overline{m}_s)$, we have
        \begin{equation}\label{exter00}
            ||(-\Delta)^{\frac{\rho}{2}} w||_{L^{p}((\mathbb{R}^N\backslash
                \Omega)\times(0,T))}\le C(\O,N,s,p)T^{\frac{r-p}{pr}}
            \bigg\|\dfrac{w}{\delta^s}\bigg\|_{L^{r}(\Omega_T)}\le
            CT^{\frac{1}{p}-\frac{1}{m}+\frac{N}{2s}(\frac{1}{r}-\frac{1}{m})+\frac
                12}||h||_{L^m(\Omega_T)}.
        \end{equation}
    \end{Proposition}

   \begin{proof} We follow closely the proof of Proposition
        \ref{ext1}. For the reader's convenience, we include here some
        details. Let $\Omega_1=\{x\in \mathbb{R}^N\backslash \O;\,
        \text{dist}(x,\partial\Omega)>>1\}$, then, for $x\in \O_1$,
        \begin{equation}\label{RQ2}
            |(-\Delta)^{\frac{\rho}{2}} w(x,t)|\le \frac{2}{(|x|+1)^{N+\rho}}\int _{\Omega } \frac{|w(y,t)|}{\d^s(y)}dy.
        \end{equation}
        Since $\dfrac{|w|}{\d^s}\in L^r(\O_T)$ for all
        $r<\overline{m}_{s}=\frac{m(N+2s)}{N+2s-ms}$, then we conclude that $|(-\Delta)^{\frac{\rho}{2}} w|\in L^1(\Omega_1\times (0,T))\cap
        L^{\theta}(\Omega_1\times (0,T))$ and for all $\theta\le r$
        and
        \begin{equation}\label{eqrrpp}
            ||(-\Delta)^{\frac{\rho}{2}} w||_{L^\theta(\O_1\times
                (0,T))}\le C(\O)T^{\frac{r-\theta}{r\theta}}\bigg\|\frac{w}{\d^s}\bigg\|_{L^r(\O_T)}.
        \end{equation}
        We suppose now that $x\in \Omega_2=\mathbb{R}^N\backslash
        (\Omega_1\cup \Omega)$. Without loss of generality,  we can assume
        that $0<\text{dist}(x, \partial\Omega)\le 2$ for all $x\in \O_2$.

        Let $x\in \O_2$, then for $\a>0$ small enough to be chosen later,
        we have \begin{equation}\label{WR} |(-\Delta)^{\frac{\rho}{2}}
            w(x,t)|\le
            \frac{R(x,t)}{(\text{dist}(x,
                \partial\Omega))^{\rho-s+\alpha}},
        \end{equation}
        where
        $$
        R(x,t)=\int_{\O}\dfrac {w(y,t)}{\delta^s(y)}\frac{1}{|x-y|^{N-\alpha}} dy dt.
        $$
        Recall that, for a.e. $t\in (0,T)$, we have $\dfrac{w(.,t)}{\d^s(.)}\in L^r(\O)$ for all
        $r<\overline{m}_{s}$. Using Theorem \ref{stein1}, it follows that $R(.,t)\in L^{\frac{rN}{N-r\a}}(\ren)\cap L^1(\ren)$
        and
        $$
        ||R(.,t)||_{L^{\frac{rN}{N-r\a}}(\mathbb{R}^N)}\le C(\O)\bigg\|\dfrac{w(.,t)}{\delta^s}\bigg\|_{L^{r}(\O)}.
        $$
        Now, we fix $r<\overline{m}_s$, then by the generalized
        H\"older's
        inequality, we deduce
        that $|(-\Delta)^{\frac{\rho}{2}} w(,t)|\in L^p(\O_2)$
        for all $p<\frac{Nr}{(\rho-s+\a)Nr+N-\a r}$
        and
        $$
        \dyle \bigg(\int_{\Omega_2}|(-\Delta)^{\frac{\rho}{2}} w(x,t)|^p
        dx\bigg)^{\frac 1p}\le\dyle C\bigg\|\frac{w(.,t)}{\d^s}\bigg\|_{L^{r}(\O)}.
        $$
        It is clear that $p<r$, hence
        $$
        ||(-\Delta)^{\frac{\rho}{2}} w||_{L^p(\O_2\times (0,T))}\le C(\O)
        T^{\frac{r-p}{pr}}\bigg\|\frac{w}{\d^s}\bigg\|_{L^{r}(\O_T)}.
        $$
        Choosing $\theta=p$ in \eqref{eqrrpp}, we deduce that
        $|(-\Delta)^{\frac{\rho}{2}} w(x,t)|\chi_{\{x\in
            \mathbb{R}^N\backslash \Omega\}}\in
        L^{p}((\mathbb{R}^N\backslash\Omega))$ for all $p<\overbrace{m_s}$
        and
        $$
        ||(-\Delta)^{\frac{\rho}{2}} w||_{L^{p}((\mathbb{R}^N\backslash
            \Omega)\times (0,T)}\le C(\O)T^{\frac{r-p}{pr}}
        \bigg\|\dfrac{w}{\delta^s}\bigg\|_{L^{r}(\Omega_T)},.
        $$
        Going back to Theorem \ref{u_delta_s}, we conclude.

        Now if $m>\frac{N+2s}{s}$, since $p<\frac{1}{\rho-s}$, we get the
        existence of $\a>0$ small enough such that
        $p<\frac{1}{\rho-s+\a}$.

        Notice that $\dfrac {w(.,t)}{\delta^s(.)}\in L^\infty(\O)$ a.e. for $t\in (0,T)$. Hence
        $R(.,t)\in L^\infty(\O)$ and
        $$
        ||R(.,t)||_{L^{\infty}(\mathbb{R}^N)}\le
        C(\O)\bigg\|\dfrac{w(.,t)}{\delta^s}\bigg\|_{L^{\infty}(\O)}.
        $$
        Thus
        $$
        ||R||_{L^{\infty}(\mathbb{R}^N\times (0,T))}\le
        C(\O)\bigg\|\dfrac{w}{\delta^s}\bigg\|_{L^{\infty}(\O_T)}.
        $$
        Going back to \eqref{WR}, we obtain that
        $|(-\Delta)^{\frac{\rho}{2}} w(x,t)|\chi_{\{x\in
            \mathbb{R}^N\backslash \Omega\}}\in
        L^{p}((\mathbb{R}^N\backslash\Omega))$ for all
        $p<\frac{1}{\rho-s}$ and
        $$
        ||(-\Delta)^{\frac{\rho}{2}} w||_{L^{p}((\mathbb{R}^N\backslash
            \Omega)\times (0,T))}\le C(\O)T^{\frac{1}{p}}
        \bigg\|\dfrac{w}{\delta^s}\bigg\|_{L^{\infty}(\Omega_T)}\le C(\O)
        T^{\frac{1}{p}+\frac 12-\frac{N+2s}{2sm}}||h||_{L^m(\O_T)}.
        $$
    \end{proof}

    \begin{Corollary}\label{CorRR}
        Suppose that $\rho=s$, then

        $\bullet$ If $m\le \frac{N+2s}{s}$, then for all
        $m<p<\overline{m}_s$, choosing $r$ such that $p<r<\overline{m}_s$
        and $\frac{1}{r}=\frac{1}{p}-\eta$, for $\eta>0$ small enough, we
        deduce that
        \begin{equation*}\label{RQ1}
            ||(-\Delta)^{\frac{s}{2}} w||_{L^{p}((\mathbb{R}^N\backslash
                \Omega)\times(0,T))}\le C(\O)T^{-\frac{N+2s}{2s}(\frac{1}{m}-\frac{1}{p})+\frac 12-\frac{\eta N}{2s}}||h||_{L^m(\Omega_T)}.
        \end{equation*}

        $\bullet$   If $m>\frac{N+2s}{s}$, then
        $(-\Delta)^{\frac{\rho}{2}} w\in {L^{p}((\mathbb{R}^N\backslash
            \Omega)\times(0,T))}$ for all $p<\infty$ and
        $$
        ||(-\Delta)^{\frac{\rho}{2}} w||_{L^{p}((\mathbb{R}^N\backslash
            \Omega)\times (0,T))}\le C(\O)T^{\frac{1}{p}+\frac
            12-\frac{N+2s}{2sm}}||h||_{L^m(\O_T)}.
        $$

        $\bullet$ Taking into consideration that estimate \eqref{RQ2}
        holds for $|x|>>1$, we conclude that $(-\Delta)^{\frac{s}{2}} w\in
        L^{p}((\mathbb{R}^N\backslash\Omega)\times(0,T))$ for all
        $p<r<\overline{m}_s$ and
        $$ ||(-\Delta)^{\frac{s}{2}}
        w||_{L^{p}((\mathbb{R}^N\backslash \Omega)\times (0,T))}\le C(\O_T)
        \bigg\|\dfrac{w}{\delta^s}\bigg\|_{L^{r}(\Omega_T)}\le
        C(\O_T)||h||_{L^m(\O_T)},
        $$
        where $C(\O_T)\to 0$ as $T\to 0$.
    \end{Corollary}
    Now, combining the results of Theorem \ref{regug}, Corollary
    \ref{CorR}, Proposition
    \ref{exttt} and Corollary \ref{CorRR}, we get the next global regularity result.

    \begin{Theorem}\label{global}
        Assume that $s>\frac 14${, let} $h\in L^{m}(\Omega_T)$ with
        $m\geq 1$ and define $w$ to be the unique solution to  Problem
        \eqref{eq_linear_02}. Then{, there} exists a positive constant $C$
        depending on the data and it is independent of {$T$} and $h$ such
        that
        \begin{enumerate}
            \item If $s>\frac{1}{3}$, for $1\leq m\le \frac{N+2s}{s}$, then
            $w\in L^{r}(0,T; \mathbb{L}_0^{s,r}(\Omega))$ for all
            $r<\frac{m(N+2s)}{N+2s-ms}$
            and for all $\eta>0$ small enough. Moreover, we have
            $$||w||_{L^{r}(0,T; \mathbb{L}_0^{s,r}(\Omega))}\le
            CT^{\frac 12-\frac{N+2s}{2s}(\frac{1}{m}-\frac{1}{r})}(1+
            T^{-\frac{\eta}{2s}})||h||_{L^m(\Omega_T)}
            $$

            \item If $s>\frac 13$ and $m>\frac{N+2s}{s}$, then $w\in
            L^{r}(0,T; \mathbb{L}_0^{s,r}(\Omega))$ for all $r<\infty$
            and for all $\eta>0$ small enough, we have

            $$
            ||w||_{L^{r}(0,T; \mathbb{L}_0^{s,r}(\Omega))}\le C(\O)T^{\frac
                12-\frac{N+2s}{2sm}}(1+T^{\frac{1}{r}}+
            T^{-\frac{\eta}{2s}})||h||_{L^m(\Omega_T)}.
            $$
            \item If $\frac 14<s\le \frac 13$ and $1\le m\le
            \frac{N+2s}{4s-1}$, then for $m<r<\frac{m(N+2s)}{N+2s-m(4s-1)}$,
            we have
            $$
            ||w||_{L^{r}(0,T;\mathbb{L}_0^{s,r}(\Omega))} C(\O)
            T^{-\frac{N+2s}{2s}(\frac{1}{m}-\frac{1}{r})+\frac
                12}\bigg(1+T^{\frac{1-3s}{2s}}+T^{\frac{3s-1-\eta}{2s}}+T^{\frac{3s-1}{2s}}+T^{-\frac{\eta}{2s}}\bigg)
            ||h||_{L^m(\Omega_T)}.
            $$
            \item If $\frac 14<s\le \frac 13$ and $m>\frac{N+2s}{4s-1}$, then
            $w\in L^{r}(0,T; \mathbb{L}_0^{s,r}(\Omega))$ for all $r<\infty $
            and
            $$||w||_{L^{r}(0,T;\mathbb{L}_0^{s,r}(\Omega))}\le C(\O)
            (T^{\frac{1-3s}{2s}}+1)T^{\frac{1}{r}-\frac{N+2s}{2sm}}\bigg(T^{\frac
                12} + T^{\frac{4s-1-\eta}{2s}}+ T^{\frac{4s-1}{2s}}\bigg)
            ||h||_{L^m(\Omega_T)} .$$
        \end{enumerate}
    \end{Theorem}
    As a consequence and by taking into consideration the relation
    between $\mathbb{L}^{s,r}_0(\O)$ and $\W^{s,r}_0(\O)$, we get the next
    corollary that improves the regularity results obtained in
    \cite{AABP, APPS}.
    \begin{Corollary}
        Suppose that $s>\frac 14$, let $h\in L^{m}(\Omega_T)$ with $1\leq m$ and consider $w$, the unique solution to  Problem \eqref{eq_linear_02}.
        Then,
        we have
        \begin{enumerate}
            \item If $\frac 14<s\le \frac 13$ and $1\le m<\frac{N+2s}{4s-1}$,
            then $w\in L^r(0,T; \W^{s,r}_0(\O))$ for all $r<\frac{m(N+2s)}{N+2s-m(4s-1)}$\textcolor{magenta}    {;}
            \item If $\frac 13<s<1$ and $1\le m\le \frac{N+2s}{s}$, then $w\in L^r(0,T; \W^{s,r}_0(\O))$ for all $r<\frac{m(N+2s)}{N+s-ms}$.
        \end{enumerate}
        In particular, for $m=1$, we have $w\in L^r(0,T; \W^{s,r}_0(\O))$
        for all $r<\min\{\frac{{N+2s}}{N+1-2s},\frac{N+2s}{N+s}\}$.

        Furthermore, in any case, we have
        $$||w||_{L^r(0,T; \W^{s,r}_0(\O))}\le
        C(\O_T)||h||_{L^m(\O_T)}.
        $$
    \end{Corollary}

    \begin{remarks}\label{mainr1oo}
        Assume that $s>\frac 14$ and let $\rho$ be fixed such that $s\le
        \rho<\max\{2s,1\}$ {and} $\rho-s<\min\{\frac{s}{(N+2s)},
        \frac{4s-1}{(N+2s-1)}\}$.
        {Then, we have \\ \rm (i) For}  $s>\frac 13$ and $r<\frac{m(N+2s)}{N+2s-ms}$, we get the
        existence of $\rho>s$, close to $s$ such that $r<\frac{m(N+2s)}{(N+2s)(m(\rho-s)+1)-ms}${;}\\
{\rm (ii) For}   $\frac 14<s\le \frac 13$ and
$r<\frac{m(N+2s)}{N+2s-m(4s-1)}$,
        we get the existence of $\rho>s$, close to $s$ such that\\ $r<\frac{m(N+2s)}{(N+2s)(m(\rho-s)+1)-m(3s+\rho-1)}$.\\

        Therefore, according to Theorem \ref{regug} and Proposition
        \ref{exttt}, choosing $\rho>s$ as above, we deduce that if $w$ is
        the unique weak solution to Problem \eqref{eq_linear_02}, then
        $w\in L^{r}(0,T; \mathbb{L}_0^{\rho,r}(\Omega))$ and
        $$
        ||w(.,t)||_{L^{\rho}(0,T; \mathbb{L}_0^{\rho,r}(\Omega))}\le
        C(\O_T)||h||_{L^m(\O_T)}.
        $$
    \end{remarks}

    \

    Let now $w$ be the solution to the general Problem $(FHE)$, then
    combing the above regularity results for booth cases: $h=0, w_0\neq 0$ and $h\neq 0, w_0=0$, we get the next consequence:
    \begin{Corollary}\label{mainr1000}
        Assume that $s>\frac 14$ and fix be $\rho$ such that $s\le
        \rho<s+\min\{\frac{s}{(N+2s)},\frac{4s-1}{(N+2s-1)}\}$. Let $h\in
        L^1(\O_T)$ and $w_0\in L^1(\O)$. Define
        \begin{equation}\label{kapp}
            \widehat{\kappa}_{s,\rho}=\min\bigg\{\frac{N+2s}{(N+2s)(\rho-s)+N+s},\frac{N+2s}{(N+2s)(\rho-s)+N+1-s-\rho}\bigg\}.
        \end{equation}

        $\bullet$ If $w$ is a solution to Problem \eqref{eq_linear_02},
        then using Theorem \ref{regug} and Proposition \ref{exttt}, we
        obtain that if $s\le
        \rho<s+\min\{\frac{s}{(N+2s)},\frac{4s-1}{(N+2s-1)}\}$, then
        $(-\Delta)^{\frac{\rho}{2}} w\in L^r(\mathbb{R}^N\times (0,T))$
        for all $r<\widehat{\kappa}_{s,\rho}$ and
        $$
        ||w||_{L^r( 0,T;\mathbb{L}_0^{\rho,r}(\Omega))}\le
        C(\O_T)\bigg(||h||_{L^1(\O_T)}+\bigg\|\dfrac{w}{\delta^s}\bigg\|_{L^{r_0}(\Omega_T)}\bigg)\le
        C(T)||h||_{L^1(\O_T)},
        $$
        where $r_0>1$ is chosen such that $r<r_0<\frac{N+2s}{N+s}$.

        \

        $\bullet$ Now, going back to Problem $(FHE)$, assume that
        $(h,w_0)\in L^1(\O_T)\times L^1(\O)$, combining the regularity
        results in Theorems \ref{reg_frac_u0_t}, Corollary \ref{CRR},
        Theorem \ref{regug} and Proposition \ref{exttt}, we deduce that
        $(-\Delta)^{\frac{\rho}{2}} w\in L^r(\mathbb{R}^N\times (0,T))$
        for all $r<\widehat{\kappa}_{s,\rho}$ and
        $$
        ||w||_{L^r( 0,T;\mathbb{L}_0^{\rho,r}(\Omega))}\le
        C(\O_T)\bigg(||h||_{L^1(\O_T)}+{||w_0||_{L^{1}(\Omega)}}\bigg).
        $$
        In any case we have $C(\O_T)\to 0$ as $T\to \infty$.

        \noindent $\bullet$ If $\rho=s$, then
        $\widehat{\kappa}_{s,s}=\frac{N+2s}{N+s}$ if $s>\frac 13$ and
        $\widehat{\kappa}_{s,s}=\frac{N+2s}{N+1-2s}$ if $s\in(\frac
        14,\frac 13]$.
    \end{Corollary}

    \vspace{0.2cm}

    \section{Compactness results.}\label{Compactness_Section}
    Throughout this section{, we} assume that $s>\frac 14${ and $\rho$ be fixed } such that $s\le
    \rho<s+\min\{\frac{s}{(N+2s)},\frac{4s-1}{(N+2s-1)}\}$. Let $h\in
    L^1(\O_T)$ and $w_0\in L^1(\O)$. Recall that $w$ is the unique
    weak solution to the {problem}
    \begin{equation*}\label{linear-comp}
{   (FHE) \qquad    }    \left\{
        \begin{array}{llll}
            w_t+(-\Delta)^s w&=& h
            & \text{in}\quad\Omega_T,\vspace{0.2cm}\\
            w(x,t)&=&0&\text{in} \quad (\mathbb{R}^N\setminus\Omega)\times(0,T),\vspace{0.2cm}\\
            w(x,0)&=&w_0(x)&\text{in}\quad    \Omega.
        \end{array}
        \right.
    \end{equation*}
    According to the regularity result obtained in
    the previous section, we can define the operator\\ $\Phi: L^1(\O_T)\times L^1(\O)\to L^q(0,T; \mathbb{L}^{s,q}_0(\O))$
    for $q<\widehat{\kappa}_{s,\rho}$, defined in \eqref{kapp}, where
    $w=\Phi(h,w_0)$ is the unique solution to Problem
    $(FHE)$. Using the last point in Corollary \ref{mainr1000}, we deduce that
    for $q<\widehat{\kappa}_{s,s}=\min\{\frac{N+2s}{N+s},\frac{N+2s}{N+1-2s}\}$, fixed, there
    exists $\e>0$ small enough such that for all $\rho\in [s,s+\e]$, we
    have $w\in L^{q}(0,T; \mathbb{L}_0^{\rho,q}(\Omega))$ and
    $$||w||_{L^{q}(0,T; \mathbb{L}_0^{\rho,q}(\Omega))}\le
    C(\O_T)(||h||_{L^1(\O_T)}+||w_0||_{L^1(\O_T)}),
    $$
    with $C(\O_T)\to 0$ as $T\to 0$. The goal of this section is to
    show that $\Phi$ is a compact operator{, this }will be achieved by
    improving the previous regularity {that we have shown for} $w$.

    Taking into consideration the deep difference between the cases
    $s>\frac 12$ and $s\le \frac 12$. We will consider separately the
    above cases using different approach.

    The results obtained in this
    section can be seen as an improvement of previous compactness
    results obtained in \cite{CV,APPS,AFTY}.

    In the case where $s>\frac 12$, which is the more regular, we have
    the next proposition.
    \begin{Proposition}\label{s12}
        Assume that $s>\frac 12$ {and} $(h,w_0)\in
        L^{1}(\Omega_T)\times L^1(\O)${, let} $w$ be the unique weak
        solution to
        Problem $(FHE)$, then for all $\varrho>0$, for all $r<\frac{N+2s}{N}$ and for all
        $1<\a<2s$, we have
        \begin{equation}\label{control1}
        {   \varrho \,} \bigg|\bigg\{(x,t)\in \O_T\mbox{ with } |\n w(x,t)|\ge
            \varrho\bigg\}\bigg|^{\frac{r+\a-1}{\a r}}\le
            C\bigg(||h||_{L^1(\O_T)}+||w_0||_{L^1(\O)}\bigg)^{\frac{1}{\a}}||w||^{\frac{
                    (\a-1)}{\a}}_{L^r(\O_T)}. \end{equation}
    \end{Proposition}
    \begin{proof}
        Since $s>\frac 12$, then using Proposition 3.7 in \cite{APPS}, for
        all $1<\a<2s$ and for all $k>0$, we have
        $$
        ||T_k(w)||_{L^\a(0,T;{\W^{1,\a}_0(\O)})}\le C(\O, T)k^{\a-1}
        (||w_0||_{L^1(\O)}+||h||_{L^1(\O_T)}).
        $$
        Let $\varrho>0$, then
        $$
        \begin{array}{lll}
            & \bigg\{(x,t)\in \O_T\mbox{  with } |\n w(x,t)|\ge \varrho\bigg\}\subset \\
            & \bigg\{(x,t)\in \O_T\mbox{ with } |\n w(x,t)|\ge \varrho\mbox { and
            } |w(x,t)|\le k\bigg\}\cup \bigg\{(x,t)\in \O_T\mbox{ with }
            |w(x,t)|\ge k\bigg\}.
        \end{array}
        $$
        Hence
        $$
        \begin{array}{lll}
            & \bigg|\bigg\{(x,t)\in \O_T\mbox{  with } |\n w(x,t)|\ge \varrho\bigg\}\bigg|\le
            \\ \\& \bigg|\bigg\{(x,t)\in \O_T\mbox{ with } |\n T_k(w(x,t))|\ge \varrho\bigg\}\bigg|+\bigg|\bigg\{(x,t)\in \O_T\mbox{ with }
            |w(x,t)|\ge k\bigg\}\bigg|.
        \end{array}
        $$
        Notice that
        $$
        \bigg|\bigg\{(x,t)\in \O_T\mbox{ with } |\n T_k(w(x,t))|\ge
        \varrho\bigg\}\bigg|\le \iint_{\O_T}\frac{|\n T_k(w(x,t))|^\a}{\varrho^\a} dxdt\le
        \frac{C(\O_T)
            k^{\a-1}}{\varrho^\a}\bigg(||w_0||_{L^1(\O)}+||h||_{L^1(\O_T)}\bigg).
        $$
        On the other hand, since $w\in L^r(\O_T)$, we have
        $$
        \bigg|\bigg\{(x,t)\in \O_T\mbox{ with }
        |w(x,t)|\ge k\bigg\}\bigg|\le \frac{||w||^r_{L^r(\O_T)}}{k^r}.
        $$
        Thus
        $$
        \bigg|\bigg\{(x,t)\in \O_T\mbox{  with } |\n w(x,t)|\ge \varrho\bigg\}\bigg|\le
        \frac{C(\O_T)
            k^{\a-1}}{\varrho^\a}\bigg(||w_0||_{L^1(\O)}+||h||_{L^1(\O_T)}\bigg)+
        \frac{||w||^r_{L^r(\O_T)}}{k^r}.
        $$
        Minimizing the second expression in $k$, we reach that
        $$
        \bigg|\bigg\{(x,t)\in \O_T\mbox{  with } |\n w(x,t)|\ge \varrho\bigg\}\bigg|\le
        \frac{C(\O_T,\a)}{\varrho^{\frac{\a
                    r}{r+\a-1}}}||w||^{\frac{r(\a-1)}{r+\a-1}}_{L^r(\O_T)}
        \bigg(||w_0||_{L^1(\O)}+||h||_{L^1(\O_T)}\bigg)^{\frac{r}{r+\a-1}}.
        $$
        Hence, we conclude.
    \end{proof}

    We consider now the case $s\le \frac 12$, as we have already pointed out, the situation here becomes
    more complicated due to the fact that in general $|\n w|\notin
    L^1(\O_T)$ and $(-\D)^{\frac{s}{2}}w \chi_{\{|w|\le k\}}\neq
    (-\D)^{\frac{s}{2}}T_k(w)$. In order to get the corresponding
    compactness result, we need to improve the previous estimates in
    the corresponding fractional parabolic spaces.

    Fix $q>1$, we define
    $W(x,y,t):=\dfrac{w(x,t)-w(y,t)}{|x-y|^{\frac{N}{q}+s}}$. Then, we
    have the next estimate.

    \begin{Proposition}\label{compat11}
        Assume that $\frac 14<s\le \frac 12$ and fix
        $q<\min\{\frac{N+2s}{N+s},\frac{N+2s}{N+1-2s}\}$,
        $r<\frac{N+2s}{N}$. {Then, we have}
        \begin{equation}\label{control2} \varrho
            \bigg|\bigg\{(x,y,t)\in \O\times \O\times (0,T)\mbox{ with }
            |W(x,y,t)|\ge \varrho\bigg\}\bigg|^{\frac{2r+q}{2qr}}\le
            C\bigg(||h||_{L^1(\O_T)}+||w_0||_{L^1(\O)}\bigg)^{\frac 12}||w||^{\frac
                12}_{L^r(\O_T)}.
        \end{equation}
    \end{Proposition}
    \begin{proof}
        Let $\a>0$, then
        $$
        \begin{array}{lll}
            & \bigg\{(x,y,t)\in \O\times \O\times (0,T)\mbox{  with } |W(x,y,t)|\ge \varrho\bigg\}\subset \\
            & \bigg\{(x,y,t)\in \O\times \O\times (0,T)\mbox{  with }
            |W(x,y,t)|\ge \varrho\mbox { and
            } |w(x,t)|\ge k|x-y|^{\theta}\bigg\}\\
            & \cup
            \bigg\{(x,y,t)\in \O\times \O\times (0,T)\mbox{  with }
            |W(x,y,t)|\ge \varrho\mbox { and
            } |w(y,t)|\ge k|x-y|^{\theta}\bigg\}\\
            & \cup
            \bigg\{(x,y,t)\in \O\times \O\times (0,T)\mbox{  with }
            |W(x,y,t)|\ge \varrho\mbox { and
            } |w(x,t)|\le k|x-y|^{\theta}, |w(y,t)|\le
            k|x-y|^{\theta}\bigg\},
        \end{array}
        $$
        where $k>0$ and $\theta>0$ is chosen such that $\theta
        \frac{N+2s}{N}<<N$. Then
        $$
        \begin{array}{lll}
            & \bigg|\bigg\{(x,y,t)\in \O\times \O\times (0,T)\mbox{  with }
            |W(x,y,t)|\ge \varrho\bigg\}\bigg|\le
            \\ \\& \bigg|\bigg\{(x,y,t)\in \O\times \O\times (0,T)\mbox{  with }
            |W(x,y,t)|\ge \varrho\mbox { and
            } |w(x,t)|\ge k|x-y|^{\theta}\bigg\}\bigg|\\
            &+\bigg|\bigg
            \{(x,y,t)\in \O\times \O\times (0,T)\mbox{  with } |W(x,y,t)|\ge
            \varrho\mbox { and
            } |w(y,t)|\ge k|x-y|^{\theta}\bigg\}\bigg|\\ \\
            &+
            \bigg|\bigg\{(x,y,t)\in \O\times \O\times (0,T)\mbox{  with }
            |W(x,y,t)|\ge \varrho\mbox { and
            } |w(x,t)|\le k|x-y|^{\theta}, |w(y,t)|\le
            k|x-y|^{\theta}\bigg\}\bigg|.
        \end{array}
        $$
        Notice that
        $$
        \begin{array}{lll}
            \bigg|\bigg\{(x,y,t)\in \O\times \O\times (0,T)\mbox{ with }
            |w(x,t)|\ge k|x-y|^\theta\bigg\}\bigg|&\le & \dyle
            \frac{1}{k^{r}}\int_0^T\io\io\frac{|w(x,t)|^r}{|x-y|^{r\theta}}dxdydt\\
            &\le & \dyle \frac{1}{k^{r}}\int_0^T\io |w(x,t)|^r\bigg(\io \frac{dy
            }{|x-y|^{r\theta}}\bigg)dxdt\\
            &\le & \dyle \frac{1}{k^{r}}\int_0^T\io |w(x,t)|^r\bigg(\int_{B_R(x)}
            \frac{dy}{|x-y|^{r\theta}}\bigg)dxdt\\
        \end{array}
        $$
        where $R>0$ is chosen such that $\O\subset \subset B_R(x)$.\\
        Using the fact that $r\theta<<N$, then $\int_{B_R(x)}
        \frac{dy}{|x-y|^{r\theta}}\le C(\O)$. Hence, we deduce that
        \begin{equation}\label{estim1c}
            \bigg|\bigg\{(x,y,t)\in \O\times \O\times (0,T)\mbox{ with }
            |w(x,t)|\ge k|x-y|^\theta\bigg\}\bigg|\le
            C(\O)\frac{||w||^r_{L^r(\O_T)}}{k^r}.
        \end{equation}
        In the same way we obtain that
        $$
        \bigg|\bigg\{(x,y,t)\in \O\times \O\times (0,T)\mbox{ with }
        |w(y,t)|\ge k|x-y|^\theta\bigg\}\bigg|\le
        C(\O)\frac{||w||^r_{L^r(\O_T)}}{k^r}.
        $$
        We estimate now the term
        $$
        \bigg|\bigg\{(x,y,t)\in \O\times \O\times (0,T)\mbox{  with }
        |W(x,y,t)|\ge \varrho\mbox { and
        } |w(x,t)|\le k|x-y|^{\theta}, |w(y,t)|\le
        k|x-y|^{\theta}\bigg\}\bigg|.
        $$
        Since $\O$ is a bounded domain, then for all $(x,t)\in \O\times
        \O$, we have $|x-y|^\theta\le C(\O)$. Thus
        $$
        \begin{array}{lll}
            & \bigg\{(x,y,t)\in \O\times \O\times (0,T)\mbox{  with }
            |W(x,y,t)|\ge \varrho\mbox { and
            } |w(x,t)|\le k|x-y|^{\theta}, |w(y,t)|\le
            k|x-y|^{\theta}\bigg\}\subset\\
            &\bigg\{(x,y,t)\in \O\times \O\times (0,T)\mbox{  with }
            |W(x,y,t)|\ge \varrho\mbox { and
            } |w(y,t)|\le C(\O)k, |w(y,t)|\le
            C(\O)k\bigg\}.
        \end{array}
        $$
        Let $k_1=C(\O) k$, then, under the condition $|w(x,t)|\le C(\O)k, |w(y,t)|\le
        C(\O)k$, we have
        $$
        W(x,y,t)=\frac{T_{k_1}(w(x,t))-T_{k_1}(w(y,t))}{|x-y|^{\frac{N}{q}+s_1}}.
        $$
        Thus
        $$
        \begin{array}{lll}
            &\bigg|\bigg\{(x,y,t)\in \O\times \O\times (0,T)\mbox{  with }
            |W(x,y,t)|\ge \varrho\mbox { and
            } |w(x,t)|\le k|x-y|^{\theta}, |w(y,t)|\le
            k|x-y|^{\theta}\bigg\}\bigg|\le \\ \\
            & \bigg|\bigg\{(x,y,t)\in \O\times \O\times (0,T)\mbox{  with }
            \bigg|\dfrac{T_{k_1}(w(x,t))-T_{k_1}(w(y,t))}{|x-y|^{\frac{N}{q}+s}}\bigg|\ge
            \varrho \bigg\}\bigg|.
        \end{array}
        $$
        Notice that for $t\in (0,T)$ fixed, according to Theorem \ref{GUT}
        (see Theorem 1.2 in \cite{GU}), we reach that
        $$
        \begin{array}{lll}
            H_\varrho(t) &:= & \bigg|\bigg\{(x,y)\in \O\times \O\mbox{ with }
            \bigg|\dfrac{T_{k_1}(w(x,t))-T_{k_1}(w(y,t))}{|x-y|^{\frac{N}{q}+s}}\bigg|\ge
            \varrho \bigg\}\bigg| \\ \\ &\le &
            \dfrac{C}{\varrho^q}\bigg\|(1-\D)^{\frac{s}{2}}T_{k_1}(w(.,t))\bigg\|^q_{L^q(\ren)}\simeq
            \dfrac{C}{\varrho^q}
            \bigg\|(-\D)^{\frac{s}{2}}T_{k_1}(w(.,t))\bigg\|^q_{L^q(\ren)}.
        \end{array}
        $$
        Integrating in time, we reach that
        $$
        \begin{array}{lll}
            &\bigg|\bigg\{(x,y,t)\in \O\times \O\times (0,T)\mbox{ with }
            \bigg|\dfrac{T_{k_1}(w(x,t))-T_{k_1}(w(y,t))}{|x-y|^{\frac{N}{q}+s}}\bigg|\ge
            \varrho \bigg\}\bigg|\\
            &=\dyle \int_0^T H_\varrho(t)dt\le \dfrac{C}{\varrho^q}\int_0^T
            ||(-\D)^{\frac{s}{2}}T_{k_1}(w(.,t))||^q_{L^q(\ren)}dt\\ \\&\le
            \dfrac{C}{\varrho^q}\dyle \bigg(\int_0^T
            ||(-\D)^{\frac{s}{2}}T_{k_1}(w(.,t))||^q_{L^q(\O)}dt+\int_0^T
            ||(-\D)^{\frac{s}{2}}T_{k_1}(w(.,t))||^q_{L^q(\ren\backslash
                \O)}dt\bigg).
        \end{array}
        $$
        Using H\"older's inequality, it holds that
        $$
        \begin{array}{lll}
            \dyle \int_0^T ||(-\D)^{\frac{s}{2}}T_{k_1}(w(.,t))||^q_{L^q(\O)}dt
            &\le & \dyle C(\O)\int_0^T
            ||(-\D)^{\frac{s}{2}}T_{k_1}(w(.,t))||^{q}_{L^2(\O)}dt\\
            &\le & \dyle C(\O_T) \bigg(\int_0^T
            ||(-\D)^{\frac{s}{2}}T_{k_1}(w(.,t))||^{2}_{L^2(\O)}dt\bigg)^{\frac{q}{2}}\\
            &\le & \dyle C(\O_T) \int_0^T
            ||(-\D)^{\frac{s}{2}}T_{k_1}(w(.,t))||^{2}_{L^2(\ren)}dt\bigg)^{\frac{q}{2}}.
        \end{array}
        $$
        Using the fact that
        $$
        \begin{array}{lll}
            \dyle\int_0^T
            ||(-\D)^{\frac{s}{2}}T_{k_1}(w(.,t))||^2_{L^2(\ren)}dt &\simeq &
            \dyle\int_0^T \iint_{\ren\times \ren}
            \frac{|T_{k_1}(w(x,t))-T_{k_1}(w(y,t))|^2}{|x-y|^{N+2s}}dxdydt\\
            &\le & k_1\bigg (||h||_{L^1(\O_T)}+||w_0||_{L^1(\O)})\bigg),
        \end{array}
        $$
        we deduce that
        $$
        \int_0^T ||(-\D)^{\frac{s}{2}}T_{k_1}(w(.,t))||^q_{L^q(\O)}dt\le
        C(\O_T)k^{\frac{q}{2}}_1\bigg
        (||h||_{L^1(\O_T)}+||w_0||_{L^1(\O)}\bigg)^{\frac{q}{2}}.
        $$
        To estimate the term $\int_0^T
        ||(-\D)^{\frac{s}{2}}T_{k_1}(w(.,t))||^q_{L^q(\ren\backslash
            \O)}dt$, we follow closely the computations used in the proof of
        Proposition \ref{ext1}.

        Let $\Omega_1=\{x\in \mathbb{R}^N\backslash \O;\,
        \text{dist}(x,\partial\Omega)>>1\}$, then $|x-y|\ge
        \frac{|x|+1}{2}$ for all $x\in \O_1$ and $y\in \Omega$. Hence, using H\"older's inequality, we
        have
        $$
        \begin{array}{lll}
            \dyle \int_0^T
            ||(-\D)^{\frac{s}{2}}T_{k_1}(w(.,t)||^q_{L^q(\O_1)}dt &\le & \dyle\int_0^T
            \int_{\O_1}\frac{2}{(|x|+1)^{q(N+s)}}\bigg(\int _{\Omega}
            |T_{k_1}(w(y,t))|dy\bigg)^q dxdt\\
            &\le &\dyle C(\O_1)\int_0^T \bigg(\io |T_{k_1}(w(y,t))|dy\bigg)^q
            dt\\
            &\le & \dyle C(\O)\bigg(\int_0^T \int _{\Omega}
            \frac{|T_{k_1}(w(y,t))|^2}{\d^{2s}(y)}dydt\bigg)^{\frac{q}{2}}.
        \end{array}
        $$
        Let now $\Omega_2=\mathbb{R}^N\backslash (\Omega_1\cup \Omega)$.
        Without loss of generality,  we can assume that $0<\text{dist}(x,
        \partial\Omega)\le 2$ for all $x\in \O_2$.
        For $\a>0$ small enough, we have
        $$
        \begin{array}{lll}
            \dyle \int_0^T
            ||(-\D)^{\frac{s}{2}}T_{k_1}(w(.,t)||^q_{L^q(\O_2)}dt &\le & \dyle \int_0^T
            \int_{\O_2}\bigg(\io \frac{|T_{k_1}(w(y,t))}{|x-y|^{N+s}}dy\bigg)^q dxdt\\
            &\le &\dyle  \int_0^T \int_{\O_2} \frac{dx}{\d^{q\a}(x)}\bigg(\io
            \frac{|T_{k_1}(w(y,t))|}{\d^s(y)}\frac{dy}{|x-y|^{N-\a}}dy\bigg)^q dx dt\\
            &\le & \dyle \int_0^T \int_{\O_2}\frac{R^q(x,t)}{\d^{q\a}(x)}dxdt,\\
        \end{array}
        $$
        where $R(x,t)=\dyle \io
        \frac{|T_{k_1}(w(y,t))|}{\d^s(y)}\frac{dy}{|x-y|^{N-\a}}dy$.
        Notice that $\dfrac{T_{k_1}(w(y,t))}{\d^s(y)}\in L^2(\O)$, then {by}
        using Theorem \ref{stein1}, it holds that $R(.,t)\in
        L^{\frac{2N}{N-2\a}}(\O_2)$ and
        $||R(.,t)||_{L^{\frac{2N}{N-2\a}}(\O_2)}\le
        C\bigg\|\frac{T_{k_1}(w(.,t))}{\d^s(.)}\bigg\|_{L^2(\O)}$. Hence, for
        $q<\frac{2N}{N-2\a}$, using {H\"older's } inequality, choosing $\a>0$
        such that $q\a\frac{2N}{2N-q(N-2\a)}<<1$, we get
        $$
        \begin{array}{lll}
            \dyle \int_0^T
            ||(-\D)^{\frac{s}{2}}T_{k_1}(w(.,t))||^q_{L^q(\O_2)}dt &\le & \dyle \int_0^T
            \int_{\O_2}\bigg(R^{\frac{2N}{N-2\a}}(x,t)
            dx\bigg)^{\frac{q(N-2\a)}{2N}}\bigg(
            \int_{\O_2}\frac{dx}{(\d(x))^{q\a\frac{2N}{2N-q(N-2\a)}}}\bigg)^{\frac{2N-q(N-2\a)}{2N}}
            dt\\ \\
            &\le & \dyle C(\O_2)\int_0^T
            \int_{\O_2}\bigg(R^{\frac{2N}{N-2\a}}(x,t)
            dx\bigg)^{\frac{q(N-2\a)}{2N}}dx\\ \\&\le & \dyle C(\O)\int_0^T
            \bigg\|\frac{T_{k_1}(w(.,t))}{\d^s(.)}\bigg\|^q_{L^2(\O)}dt
            \le\dyle C(\O_T) \bigg\|\frac{T_{k_1}w}{\d^s}\bigg\|^q_{L^2(\O_T)}.
        \end{array}
        $$
        Thus, combining the above estimates, we deduce that
        $$
        \dyle \int_0^T
        ||(-\D)^{\frac{s}{2}}T_{k_1}(w(.,t)||^q_{L^q(\ren\backslash
            \O)}dt\le C(\O_T)\bigg\|\frac{T_{k_1}w}{\d^s}\bigg\|^q_{L^2(\O_T)}.
        $$
        Since $T_{k_1}(w)\in L^2(0,T,{\W^{s,2}_0(\O)})$, then using the
        Hardy's
        inequality in Theorem \ref{hardyd}, we reach that
        $$
        \bigg\|\frac{T_{k_1}(w)}{\d^s}\bigg\|^2_{L^2(\O_T)}\le C(\O) \int_0^T
        \iint_{\ren\times \ren}
        \frac{|T_{k_1}(w(x,t))-T_{k_1}(w(y,t))|^2}{|x-y|^{N+2s}}dxdydt\le
        k_1\bigg (||h||_{L^1(\O_T)}+||w_0||_{L^1(\O)}\bigg).
        $$
        In conclusion, we find that

        \begin{equation}\label{estim2}
            \bigg|\bigg\{(x,y,t)\in \O\times \O\times (0,T)\mbox{ with }
            \bigg|\frac{T_{k_1}(w(x,t))-T_{k_1}(w(y,t))}{|x-y|^{\frac{N}{q}+s}}\bigg|\ge
            \varrho \bigg\}\bigg|\le \frac{C(\O_T)k^{\frac{q}{2}}_1}{\varrho^q}\bigg
            (||h||_{L^1(\O_T)}+||w_0||_{L^1(\O)}\bigg)^{\frac{q}{2}}.
        \end{equation}
        Recall that $k_1=C(\O)k$, then,  combining estimates
        \eqref{estim1c} and \eqref{estim2}, we conclude that
        $$
        \bigg|\bigg\{(x,y,t)\in \O\times \O\times (0,T)\mbox{  with }
        |W(x,y,t)|\ge \varrho\bigg\}\bigg|\le
        \frac{C(\O_T)k^{\frac{q}{2}}}{\varrho^q}\bigg
        (||h||_{L^1(\O_T)}+||w_0||_{L^1(\O)}\bigg)^{\frac{q}{2}}+C(\O)\frac{||w||^r_{L^r(\O_T)}}{k^r}.
        $$
        Now minimizing in $k>0$, it holds that
        $$
        \bigg|\bigg\{(x,y,t)\in \O\times \O\times (0,T)\mbox{  with }
        |W(x,y,t)|\ge \varrho\bigg\}\bigg|\le
        \frac{C(\O_T)}{\s^{\frac{2qr}{2r+q}}}\bigg
        (||h||_{L^1(\O_T)}+||w_0||_{L^1(\O)}\bigg)^{\frac{qr}{2r+q}}||w||^{\frac{qr}{2r+q}}_{L^r(\O_T)}.
        $$
    Hence, the result follows.
    \end{proof}

    \begin{remarks}

        \

        $\bullet$ If $s>\frac 12$, according to \eqref{control1}, we
        obtain {that}
        $$
        \sup_{\varrho>0}\s \bigg|\bigg\{(x,t)\in \O_T\mbox{ with } |\n w(x,t)|\ge
        \varrho\bigg\}\bigg|^{\frac{r+\a-1}{\a r}}:=M\le
        C\bigg(||h||_{L^1(\O_T)}+||w_0||_{L^1(\O)}\bigg)^{\frac 12}||w||^{\frac{
                (\a-1)}{2}}_{L^r(\O_T)}.$$ Thus, $|\n w|\in \mathcal{M}^{\frac{\a
                r}{r+\a-1}}(\O_T)$, the Marcinkiewicz space. Since $\O_T$ is a
        bounded domain, we deduce that if $\frac{\a r}{r+\a-1}>1$, then
        for all $1\le b<\frac{\a r}{r+\a-1}$, we have
        \begin{equation*}
            ||\n w||_{L^b(\O_T)}\le C(\O_T,r,b) M\le
            C(\O_T,r,b)\bigg(||h||_{L^1(\O_T)}+||w_0||_{L^1(\O)}\bigg)^{\frac{
                    1}{\a}}||w||^{\frac{ (\a-1)}{\a}}_{L^r(\O_T)}.
        \end{equation*}
Taking into consideration that the above estimate holds for all
$\a<2s$ and for all $r<\frac{N+2s}{N}$, then we deduce that for
all $b<\frac{N+2s}{N+1}$, we have
\begin{equation}\label{local11}
            ||\n w||_{L^b(\O_T)}\le C(\O_T,r,b) M\le
            C(\O_T,r,b)\bigg(||h||_{L^1(\O_T)}+||w_0||_{L^1(\O)}\bigg)^{\frac{
                    1}{\a}}||w||^{\frac{ (\a-1)}{\a}}_{L^r(\O_T)}.
        \end{equation}

        $\bullet$ If $\frac 14<s\le \frac 12$, according to
        \eqref{control2}, we get
        $$
        \sup_{\varrho>0}\varrho \bigg|\bigg\{(x,y,t)\in \O\times \O\times (0,T)\mbox{
            with } |W(x,y,t)|\ge \varrho\bigg\}\bigg|^{\frac{2r+q}{2qr}}:=M\le
        C\bigg(||h||_{L^1(\O_T)}+||w_0||_{L^1(\O)}\bigg)^{\frac 12}||w||^{\frac{
                1}{2}}_{L^r(\O_T)}.$$ Hence, $w\in \mathcal{M}^{\frac{2qr}{2r+q}}(\O\times
        \O\times (0,T))$. Since $\O\times \O\times (0,T)$ is bounded, we
        obtain that, if $\frac{2qr}{2r+q}>1$, then $W\in L^a(\O\times
        \O\times (0,T))$ for all $a<\frac{2qr}{2r+q}$ and
        \begin{equation}\label{globalss}
            ||W||_{L^a(\O\times \O\times (0,T))}\le C(\O_T,r,q) M\le
            C(\O_T,r,p)\bigg(||h||_{L^1(\O_T)}+||w_0||_{L^1(\O)}\bigg)^{\frac{
                    1}{\a}}||w||^{\frac{1}{\a}}_{L^r(\O_T)}.
        \end{equation}
    \end{remarks}

    We are now in position to show that $\G$ is a compact operator.
    \begin{Theorem}\label{compact1}
        Consider $\Phi$, the operator defined previously, then
        $\Phi$ is compact.
    \end{Theorem}
    \begin{proof}
        Let $\{h_n\}_n\subset L^1(\O_T), \{w_{0n}\}_n\subset L^1(\O)$ be
        such that
        $||h_n||_{L^1(\O_T)}+||w_{0n}||_{L^1(\O)}\leq \hat{C}$. Define $w_n=\Phi(h_n,w_{0n})$, then
        $w_n$ solves the Problem:
        \begin{equation}\label{ttt}
            \left\{
            \begin{array}{llll}
                (w_n)_t+(-\Delta)^s w_n&=& h_n
                & \text{in}\quad\Omega_T,\vspace{0.2cm}\\
                w_n(x,t)&=&0&\text{in} \quad (\mathbb{R}^N\setminus\Omega)\times(0,T),\vspace{0.2cm}\\
                w_n(x,0)&=&w_{0n}&\text{in}\quad    \Omega.
            \end{array}
            \right.
        \end{equation}
        According to Corollary \ref{mainr1000} and the previous regularity
        results, for $\rho$ fixed such that $s\le
        \rho<s+\min\{\frac{s}{(N+2s)},\frac{4s-1}{(N+2s-1)}\}$, then for
        all $a<\widehat{\kappa}_{s,\rho}$, we have
        $$
        ||w_n||_{L^a( 0,T;\mathbb{L}_0^{\rho,a}(\Omega))}\le
        C(\O_T)\bigg(||h_n||_{L^1(\O_T)}+||w_{0n}||_{L^{1}(\Omega)}\bigg).
        $$

        Let $   {q}<\min\{\frac{N+2s}{N+s},\frac{N+2s}{N+1-2s}\}$, then we get
        the existence of $\rho>s$ close to $s$ such that
        $q<\widehat{\kappa}_{s,\rho}$ and $ ||w_n||_{{L^q(
            0,T;\mathbb{L}_0^{{\rho,q}}(\Omega))}}\le C(\O_T)$.  Hence, we get the
        existence of $w\in L^{q}(0,T; \mathbb{L}_0^{\rho,q}(\Omega))$ such
        that, up to a subsequence, we have $w_n\rightharpoonup w$ weakly in
        $L^{q}(0,T; \mathbb{L}_0^{s,q}(\Omega))$ and $w_n\to w$ strongly
        in $L^a(\O_T)$ for all $a<\frac{N+2s}{N}$ and $w_n\to w$ a.e. in
        $\O_T$. Moreover, we have also that $\dfrac{w_n}{\d^s}\to
        \dfrac{w}{\d^s}$ strongly in $L^{a_1}(\O_T)$ for all
        $a_1<\frac{N+2s}{N+s}$.

        For $n\ge m$, we set $\mathcal{W}_{n,m}=w_n-w_m$, then
        $\mathcal{W}_{n,m}$ solves the problem
        \begin{equation}
            \left\{
            \begin{array}{llll}
                (\mathcal{W}_{n,m})_t+(-\Delta)^s \mathcal{W}_{n,m}&=& h_n-h_m
                & \text{in}\quad\Omega_T,\vspace{0.2cm}\\
                \mathcal{W}_{n,m}&=&0&\text{in} \quad (\mathbb{R}^N\setminus\Omega)\times(0,T),\vspace{0.2cm}\\
                \mathcal{W}_{n,m}(x,0)&=&0&\text{in}\quad\Omega.
            \end{array}
            \right.
        \end{equation}
        Notice that $\mathcal{W}_{n,m}\to 0$ a.e. in $\O_T$ and
        $\mathcal{W}_{n,m}\to 0$ strongly in $L^a(\O_T)$ for all
        $a<\frac{N+2s}{N}$ as $n,m\to \infty$.

        We spilt the proof into two cases according to the value {of} $s$.

{    {\bf{Case: $s>\frac 12$.}}} In this case we have
$\frac{N+2s}{N+s}<\frac{N+2s}{N+1-2s}$. Fixed
$q<\frac{N+2s}{N+s}$. Let $a<\frac{N+2s}{N+1}$ to be chosen later,
then we get the
        existence of $1<\a<2s$ such that $a<\frac{N+2s}{N+\frac{2s}{\a}}$.
        Hence, there exists $r<\frac{N+2s}{N}$ such that $q<\frac{\a
            r}{r+\a-1}$. Going back to estimate \eqref{local11}, it
holds that
        $$
        ||\n \mathcal{W}_{n,m}||_{L^a(\O_T)}\le C(\O_T,r,a, \hat{C})
        ||\mathcal{W}_{n,m}||^{\frac{ (\a-1)}{\a}}_{L^r(\O_T)}\to 0\mbox{
            as  }n,m\to \infty.
        $$
        Therefore, we conclude that $\{w_n\}_n$ is a Cauchy sequence in the space
        ${L^q(0,T;\W^{1,a}_0(\O))}$ for all $a<\frac{N+2s}{N+1}$.
        Thus,
        $w_n\to w$ strongly in ${L^a(0,T;\W^{1,a}_0(\O))}$. Hence  $w_n\to
        w$ strongly in $L^{a}(0,T; \W^{1,a}(\ren))$. Using an
        interpolation result, see \cite{Leonibook}, it results that, for
        $q<\frac{N+2s}{N+s}$ fixed, we can get the existence of
        $a<\frac{N+2s}{N+1}$ and $r<\frac{N+2s}{N}$ such that
        $\frac{1}{q}=\frac{1-s}{r}+\frac{s}{a}$ with
        $$
        ||w_n(.,t)-w(.,t)||_{\mathbb{L}_0^{s,q}(\Omega)}\le
        ||w_n(.,t)-w(.,t)||^{1-s}_{L^r(\O)}||w_n(.,t)-w(.,t)||^{s}_{\W^{1,a}_0(\O)}\,
   a.e. \mbox{  for } t\in (0,T).$$ Integrating in time, we get
        $$
        ||w_n-w||_{L^q(0,T;\mathbb{L}_0^{s,q}(\Omega))}\le
        ||w_n-w||^{1-s}_{L^r(\O_T)}||w_n-w||^{s}_{{L^a(0,T;\W^{1,a}_0(\O))}}.
        $$
        Therefore, we conclude that, for $q<\frac{N+2s}{N+s}$,
        $$
        ||w_n-w||_{L^q(0,T;\mathbb{L}_0^{s,q}(\Omega))}\to
        0\mbox{  as  }n\to \infty,
        $$
        that is the desired result for $s>\frac 12$.

        {{\bf Case: $s\le \frac 12$. }} Recall that
        $q<\min\{\frac{N+2s}{N+s},\frac{N+2s}{{N+1-2s}}\}$.

        Let $\varrho>0$, then applying estimate \eqref{control2} to the
        function $\mathcal{W}_{n,m}$, it holds that
        $$\bigg|\bigg\{(x,y,t)\in \O\times \O\times (0,T)\mbox{  with }
        \bigg|\frac{\mathcal{W}_{n,m}(x,t)-\mathcal{W}_{n,m}(y,t)}{|x-y|^{\frac{N}{q}+s}}\bigg|\ge
        \varrho\bigg\}\bigg|\le
        \frac{C(\O_T)}{\varrho^{\frac{2qr}{2r+q}}}\bigg\|\mathcal{W}_{n,m}\bigg\|^{\frac{qr}{2r+q}}_{L^r(\O_T)}
        \to 0\mbox{  as }n,m\to \infty.
        $$
        Thus,
        $\bigg\{\frac{\mathcal{W}_{n,m}(x,t)-\mathcal{W}_{n,m}(y,t)}{|x-y|^{\frac{N}{q}+s}}\bigg\}_{n,m}$
        is a Cauchy sequence in the sense of measure. Hence, up to a
        subsequence, for a.e. $(x,y,t)\in \O\times \O\times (0,T)$, we
        have,
        $$
        \frac{|w_{n}(x,t)-w_n(y,t)|^q}{|x-y|^{N+qs}}\to
        \frac{|w(x,t)-w(y,t)|^q}{|x-y|^{N+qs}}\mbox{  as  }n\to \infty.
        $$
        Since $q<\frac{N+2s}{N+s}$, then
        $$\int_0^T\iint_{\O\times \O}
        \frac{|w_n(x,t)-w_n(y,t)|^q}{|x-y|^{N+qs}}dxdydt\le C.$$
        Therefore, by Vitali's {Lemma}, we deduce that
        \begin{equation}\label{GTR}
            \int_0^T\iint_{\O\times \O}
            \frac{|w_n(x,t)-w_n(y,t)|^q}{|x-y|^{N+qs}}dxdydt\to
            \int_0^T\iint_{\O\times \O}
            \frac{|w(x,t)-w(y,t)|^q}{|x-y|^{N+qs}}dxdydt\mbox{ as  }n\to
            \infty.
        \end{equation}
        Recall that $\dfrac{w_n}{\d^s}\to \dfrac{w}{\d^s}$ strongly in
        $L^\beta(\O_T)$ for all $\beta<\frac{N+2s}{N+s}$. Hence, using
        \eqref{GTR}, we deduce that $w_n\to w$ strongly in
        $L^q(0,T;\W^{s,q}_0(\O))$. Now, by Proposition \ref{proper},
        since $\frac{N+2s}{N+s}<2$, we obtain that $w_n\to w$ strongly in
        $L^a(0,T;\mathbb{L}_0^{s,a}(\Omega))$ for all
        $a<\min\{\frac{N+2s}{N+s},\frac{N+2s}{N+1-2s}\}$. {Hence, } the result
        follows.
    \end{proof}

    \begin{remarks}\label{comparr}
        In the case where $s\in [\frac 13, \frac 12]$, then
        $\frac{N+2s}{N+s}\le \frac{N+2s}{N+1-2s}$. {Hence, } the compactness
        result holds for all $a<\frac{N+2s}{N+s}$.
    \end{remarks}

    \begin{remarks}\label{comparr00}
        The arguments used in the case $s\le \frac 12$ are valid also for the case $s>\frac 12$. However,
        we have include the alternative proof in the case $s>\frac 12$ which is based on the regularity of the local gradient.
    \end{remarks}

    \section{{Fractional Kardar-Parisi-Zhang Problem with nonlocal gradient term}.} \label{Application_Section}

    The purpose of this section is to apply the previous
    regularity and compactness results to analyze the existence of solution $u$ to the
{    Problem}
    \begin{equation*}\label{Main_Problem}
        {(KPZ_f)} \quad \quad \left\{
        \begin{array}{llll}
            u_t+(-\Delta)^{s}u&=& |(-\Delta)^{\frac{s}{2}}u|^q+f& \text{in}\quad \Omega_T,\vspace {0.2cm}\\
            u(x,t)&=&0& \text{in}\quad (\mathbb{R}^N\setminus \Omega)\times (0,T),\vspace{0.2cm}\\
            u(x,0)&=&u_0(x) & \text{in }\quad\Omega.
        \end{array}
        \right.
    \end{equation*}
    Here $(f,u_0)\in L^m(\O_T)\times L^\s(\O)$ with $(m,\s)\in
    [1,+\infty)^2$, { $\frac 13<s<1$} and $q\leq 1$. To simplify the
    presentation of our results, we will consider separately two
    cases:

    $i)$ Case: $ f\neq 0$ and $u_0=0$;

    $ii)$ Case: $ f =0$ and $u_0\neq 0$.

    First, in the next definition, we specify the sense of the weak solution to {Problem} ${(KPZ_f)} $.

    \begin{Definition} \label{weak_kpz_solution}
        Suppose $(f,u_0)\in L^1(\O_T)\times L^1(\O)$. We say that $u$ is a weak solution to $(KPZ_f)$
        if $u, |(-\Delta)^{\frac{s}{2}}u|^q \in L^1 (\O_T)$ with $u\equiv 0$ a.e. in $(\mathbb{R}^N\setminus \O) \times (0,T)$, and for all
        $\phi \in \widetilde{\mathcal{P} _s}$, we have
        \begin{equation}
            \iint_{\O_T}\,u\big(-\phi_t\, +(-\Delta)^{s}\phi\big)\,dx\,dt=\\
            \iint_{\O_T}\, (|(-\Delta)^{\frac{s}{2}}u|^q+f)\phi\,dxdt +\int_\Omega{u_0(x)\phi (x,0)\,dx},
        \end{equation}
        with $\phi\in \widetilde{\mathcal{P}_s}$, given by
        $$\widetilde{\mathcal{P}_s}:=\left\{ \phi:\mathbb{R}^N\times[0,T]\rightarrow\mathbb{R} \; ; \; \phi(x,t)=0 \;\;\, \text{in}\;\; (\mathbb{R}^N\setminus \O) \times (0,T)\;\; \text{and} \;\; (-\Delta)^s\phi \in L^{\infty} (\O_T)
        \right\}.$$
    \end{Definition}

    \subsection{{Case:} $ f\neq 0$
        and $u_0=0$.} In this case, we treat mainly the { Problem
}   \begin{equation}\label{exis_prb}
        \left\{
        \begin{array}{llll}

            u_t+(-\Delta)^{s}u&=& |(-\Delta)^{\frac{s}{2}}u|^q+f& \text{in}\quad \Omega_T,\vspace{0.2 cm}\\
            u(x,t)&=&0& \text{in}\quad (\mathbb{R}^N\setminus \Omega)\times (0,T),\vspace{0.2cm}\\
            u(x,0)&=&0 & \text{in} \quad \Omega.
        \end{array}
        \right.
    \end{equation}
    Our existence result is given in the next theorem.

    \begin{Theorem}\label{existence-nancy}
        Suppose $f\in L^m(\O_T)$ with $1\leq m$ and $s\in (\frac 14, 1)$. Assume that one of the following conditions hold

        \begin{enumerate}
            \item $s\in (\frac 13, 1)$ and
            \begin{equation}\label{ConI}
                \left\{\begin{array}{lll}
                    1\leq m\leq \frac{N+2s}{s} \vspace{0.2cm}\mbox{  and } 1<q \leq
                    \frac{N+2s}{N+2s-ms},\\
                    \mbox{  or  }\\
                    m>\frac{N+2s}{s}\vspace{0.4cm}\mbox{  and }q \in (1, \infty).
                \end{array}
                \right.
            \end{equation}

            \item $s\in (\frac 14,\frac 13]$ and
            \begin{equation}\label{ConIII}
                \left\{\begin{array}{lll}
                    1\leq m\leq \frac{N+2s}{4s-1}\vspace{0.4cm}  \mbox{  and } 1<q \leq
                    \frac{N+2s}{N+2s-m(4s-1)},\\
                    \mbox{  or  }\\
                    m>\frac{N+2s}{4s-1}\vspace{0.6cm} \mbox{  and  } q\in(1, \infty).
                \end{array}
                \right.
            \end{equation}
        \end{enumerate}
        Then, there exists $0<T^*<T$ such that Problem \eqref{exis_prb}
        has a solution $u\in L^{\gamma}(0,T^*; \mathbb{L
        }^{s,\gamma}_0(\Omega))$ with
        \begin{enumerate}
            \item $1<\gamma< \overline{m}:=\frac{m(N+2s)}{N+2s-ms}$ if $1\leq
            m<\frac{N+2s}{s}$, and $1<\g<\infty$ if $
            m\geq\frac{N+2s}{s}$, if $s\in (\frac 13,1)$, \\
            or
            \item $1<\gamma< \overline{m}:=\frac{m(N+2s)}{N+2s-m(4s-1)}$ if $1\leq
            m<\frac{N+2s}{4s-1}$, and $1<\g<\infty$ if $
            m\geq\frac{N+2s}{4s-1}$, if $s\in (\frac 14,\frac 13]$.
        \end{enumerate}
    \end{Theorem}
    \begin{remarks}
        Taking into consideration the relation between $\W^{s,p}_0(\O)$
        and $\mathbb{L}^{s,p}_0(\O)$, we deduce that the solution to
        Problem \eqref{exis_prb} obtained in the previous theorem
        satisfies:
        \begin{enumerate}
            \item If $s\in (\frac 13,1)$, then \\
            $i)$ $u\in L^{\gamma}(0,T^*; \mathbb{W}^{s,\gamma}_0(\Omega))$ for
            all $1<\gamma<\frac{m(N+2s)}{N+2s-ms}$ if $1\leq
            m<\frac{N+2s}{s}$,\\
            $ii)$ $u\in L^{\gamma}(0,T^*; \mathbb{W}^{s,\gamma}_0(\Omega))$
            for all $1<\g<\infty$ if $ m\geq\frac{N+2s}{s}$.
            \item If $s\in (\frac 14,\frac 13]$, then \\
            $i)$ $u\in L^{\gamma}(0,T^*; \mathbb{W}^{s,\gamma}_0(\Omega))$ for
            all $1<\gamma<\frac{m(N+2s)}{N+2s-m(4s-1)}$ if $1\leq
            m<\frac{N+2s}{4s-1}$,\\
            $ii)$ $u\in L^{\gamma}(0,T^*; \mathbb{W}^{s,\gamma}_0(\Omega))$
            for all $1<\g<\infty$ if $ m\geq\frac{N+2s}{4s-1}$.
        \end{enumerate}

    \end{remarks}

    {\bf Proof of Theorem \ref{existence-nancy}.}

    We give a complete proof under the first set of conditions, the other cases follow in the same
    way using the corresponding regularity tools..\\

    Suppose that $f\in L^m(\O_T)$ with $1\leq m\leq \frac{N+2s}{s}${, let} $q>1$ be such that $1<q \leq
    \frac{N+2s}{N+2s-ms}$, then $qm<\overline{m}_s=\frac{m(N+2s)}{N+2s-ms}$. { Hence, } we get the existence of $r>1$ such that
    $qm<r<\overline{m}_s$. \\
    Fixed $r$ as above, then we define the set
    $$E:=\left\{\varphi\in L^1(0,T;\,\, \mathbb{L}^{s,1+\xi}_0(\O))\,\,; \,\, ||\varphi||_{L^r(0,T;\,\, \mathbb{L}^{s,r}_0(\O))}\leq \ell ^{\frac 1q}\right\},$$
    with {$0<\xi<q-1$}. It is clear that $E$ is a convex and closed set. Now, we consider the operator $\hat{\Phi}$ defined by
    $$
    \begin{array}{lll}
        \hat{\Phi}&:&E\to L^1(0,T;\mathbb{L}^{s,1+\xi}_0(\O))\\
        &&\varphi\to{\hat{ \Gamma}}(\varphi)=u,
    \end{array}$$
    with $u$ being the unique weak solution of the {Problem}
    \begin{equation}\label{exis_prb1}
        \left\{
        \begin{array}{llll}
            u_t+(-\Delta)^{s}u&=& |(-\Delta)^{\frac{s}{2}}\varphi|^q+f& \text{in}\quad \Omega_T,\vspace{ 0.2cm}\\
            u(x,t)&=&0& \text{in}\quad (\mathbb{R}^N\setminus \Omega)\times (0,T),\vspace{0.2cm}\\
            u(x,0)&=&0 & \text{in} \quad \Omega.
        \end{array}
        \right.
    \end{equation}
    To obtain our existence result, we just need to prove that
    $\hat{\Phi}$ has a fixed point. The main idea is to the Schauder's fixed point
    Theorem in a suitable set. The proof will be {given} in
    several steps.

    {\bf{Step 1.}} We show that $\hat{\Phi}$ is well-defined.\\ Indeed,
    since $\varphi\in E$ and $qm<r$, then
    $|(-\Delta)^{\frac{s}{2}}\varphi|^q+f \in L^m( \O_T)\subset
    L^1(\O_T)$. Thus, according to Theorem \ref{main-exis} and the regularity result in Corollary \ref{mainr1000}, we get
    the existence and the uniqueness of a weak solution $u$ to Problem \eqref{exis_prb1} such
    that $u\in L^\s(0,T;\,\, \mathbb{L}^{s,\s}_0(\O))$ for all $\s<\frac{N+2s}{N+s}$. .
  {Hence, } $\hat{\Phi}$ is well defined.

    {\bf{Step 2}}. We show that $\hat{\Phi}(E)\subset E$.

    Let $h=|(-\Delta)^{\frac{s}{2}}\varphi|^q+f,$ then $h\in
    L^m(\O_T)$, using Theorem \ref{global}, we get
    $$\begin{array}{llll}
        ||u||_{L^{\nu}(0,T; \mathbb{L}_0^{s,\nu}(\Omega))}&\le &C (\O,T)||h ||_{L^m(\O_T)}\vspace{0.2cm}\\
        &\leq & C (\O,T)\left(||(-\Delta)^{\frac{s}{2}}\varphi||_{L^{qm}(\O_T)}^q+ ||f||_{L^m(\O_T)}\right)
        \vspace{0.2cm}\\
        &\leq & C(\O,T)\left(||\varphi||_{L^r(0,T;\,\,\mathbb{L}^{s,r}_0(\O) )}^q+||f||_{L^m(\O_T)}\right) \vspace{0.2cm}\\
        &\leq & C(\O,T)\left(\ell+||f||_{L^m(\O_T)}\right),
    \end{array}
    $$
{   for all $\nu<\overline{m}_s$}. Therefore, using the fact that
    $C(\O,T)\to 0$ as $T\to 0$, we get the existence of $T^*>0$ such
    that for $T\le T^*$ fixed, there exists $\ell>0$ with
    $C(\O,T)\left(\ell+||f||_{L^m(\O_T)}\right)\leq \ell ^{\frac 1q}$. Hence
    $$||u||_{L^{\nu}(0,T; \mathbb{L}_0^{s,\nu}(\Omega))} \leq \ell ^{\frac 1q}. $$
    Choosing $\nu=r$, we reach that $u\in E$. {Thus, } $\hat{\Phi}(E)\subset E$.

    {\bf Step 3}. Compactness and continuity of {$\hat{\Phi}$}.

    Let us begin by proving that $\hat{\Phi}$ is continuous. Consider
    $\left\{\varphi_n\right\}_n \subset E$ such that
    $\varphi_n\to \varphi$ strongly in $L^1(0,T;\,\, \mathbb{L}^{s,1+\xi}_0(\O))$.
    \\
    Setting $u_n:=\hat{\Phi}(\varphi_n)$ and $u:=\hat{\Phi}(\varphi)$.
    Then, $(u_n,u)$ satisfies
    \begin{equation}\label{Sssys1}
        \left\{
        \begin{array}{lclll}
            \partial_t u_n+(-\Delta)^{s} u_{n} &= &|(-\Delta)^{\frac s2}\varphi_n|^{q}+ f & \text{ in }&\Omega_T, \\
            \partial_t u+(-\Delta)^{s} u &= &|(-\Delta)^{\frac s2}\varphi|^{q}+ f & \text{ in }&\Omega_T, \\
            u(x,t) =u_n(x,t)&=& 0 &\hbox{ in }& (\mathbb{R}^N\setminus\Omega)\times (0,T),\\
            u(x,0) = u_n(x,0)&=& 0 &\hbox{ in }& \Omega.
        \end{array}%
        \right.
    \end{equation}
    Setting $w_n=u_n-u$, then $w_n$ solves the {Problem}
    \begin{equation}\label{auxi1}
        \left\{
        \begin{array}{lclll}
            \partial_t w_n+(-\Delta)^{s} w_{n} &= &|(-\Delta)^{\frac s2}\varphi_n|^{q}-|(-\Delta)^{\frac s2}\varphi|^{q}& \text{ in }&\Omega_T, \\
            w_n(x,t)&=& 0 &\hbox{ in }& (\mathbb{R}^N\setminus\Omega)\times (0,T),\\
            w_n(x,0)&=& 0 &\hbox{ in }& \Omega.
        \end{array}%
        \right.
    \end{equation}
    Hence, using Corollary \ref{mainr1000}, it holds that
    {\begin{equation}\label{imp1}
            \bigg\||(-\Delta)^{\frac s2} u_n|-|(-\Delta)^{\frac s2} u|\bigg\|_{L^{\alpha}(\O_T)}
            \le||(-\Delta)^{\frac s2} w_n||_{L^{\alpha}(\O_T)} \leq
            C\bigg\||(-\Delta)^{\frac s2}\varphi_n|^{q}-|(-\Delta)^{\frac
                s2}\varphi|^{q}\bigg\| _{L^{1}(\O_T)}\textcolor{magenta}{,}
    \end{equation}}
    for all $\alpha<\frac{N+2s}{N+s}$. {On the other hand, since } $\left\{\varphi_n\right\}_n \subset E$,
    then, for $r$ fixed as above, we have
    $$||\varphi_n||_{L^r(0,T;\,\,\mathbb{L}^{s,r}_0(\O))}, ||\varphi||_{L^r(0,T;\,\,\mathbb{L}^{s,r}_0(\O))}\le C\mbox{  for all }n.$$
    Now, by interpolation, for $\theta\in (0,1)$ with $\frac 1q= \frac {\theta}{1+\zeta}+ \frac {1-\theta}{r}$, it follows that
    \begin{equation*}
        \begin{array}{lll}
            ||(-\Delta)^{\frac s2} (\varphi_n-
            \varphi)||_{L^q(\mathbb{R}^N\times (0,T))} &\le &
            C||(-\Delta)^{\frac s2} (\varphi_n-
            \varphi)||^{\theta}_{L^{1+\zeta}(\mathbb{R}^N\times (0,T))}
            ||(-\Delta)^{\frac s2} (\varphi_n-
            \varphi)||^{1-\theta}_{L^r(\mathbb{R}^N\times (0,T))}\\
            &\leq & C ||(-\Delta)^{\frac s2} (\varphi_n-
            \varphi)||^{\theta}_{L^{1+\zeta}(\mathbb{R}^N\times (0,T))}
            \to 0\mbox{ as }n\to
            \infty.
        \end{array}
    \end{equation*}
{    Thus, } $|(-\Delta)^{\frac s2} (\varphi_n-\varphi)|\to
    0$ strongly in $L^{q}(
    \mathbb{R}^N\times (0,T)))$. {In particular, } $|(-\Delta)^{\frac s2} \varphi_n|\to |(-\Delta)^{\frac s2} \varphi|$ strongly in $L^{q}(\mathbb{R}^N \times (0,T))$.
    Therefore, $(-\Delta)^{\frac s2} u_n\to (-\Delta)^{\frac s2} u$ strongly in $L^{\alpha}(\mathbb{R}^N \times (0,T))$.
    Then, $(-\Delta)^{\frac s2} u_n\to (-\Delta)^{\frac s2} u$ strongly in $L^{1+\zeta}(\mathbb{R}^N\times (0,T))$. {Hence, } $\hat{\Gamma}$ is
    continuous.

    We show now the Compactness of $\hat{\Phi}$. Let
    $\{\varphi_n\}_n\subset E$ be such that
    $\|\varphi_n\|_{L^1(0,T;\,\, \mathbb{L}^{s,1+\xi}_0(\O))}\le C$, with $C $ is a constant independent of $n$ and $u_n=\hat{\Gamma}(\varphi_n)$.
    Since $\{\varphi_n\}_n\subset
    E$, then for $r$ fixed such that $mq<r<\overline{m}_s$, we have $||\varphi_n||_{L^r(0,T;\,\,\mathbb{L}^{s,r}_0(\O))}<C$ for all
    $n$.

    Thus, $h_n=|(-\Delta)^{\frac s2}\varphi_n|^{q}+ f $ is bounded in $ L^{1}(\O_T)$, thanks to the compactness results in
    Theorem \ref{compact1} and Remark \ref{comparr}, we get the existence of a subsequence denoted $\{u_n\}_n$ such that $u_n\to u$
    strongly in $L^{\a}(0, \,T;\,\mathbb {L}^{s,\a}_0(\O))$ for all $\a<\frac{N+2s}{N+s}$.

    In particular, $u_n\to u$ strongly in $L^1(0,\,T;\,\mathbb{L}^{1,1+\zeta}_0(\O))$. Thus, the compactness of $\hat{\Phi}$ follows.\\
    Therefore, since all the conditions of Schauder's fixed point
    Theorem are satisfied, we get the {existence of  }a function $u\in E$ such that $\hat{\Phi}(u)=u$. It is clear that $u$ is a weak solution to
    Problem \eqref{exis_prb}. \cqd

    \

    \subsection{{Case:  $f =0$ and $u_0\neq 0$.}}

    In this subsection, we deal with the case $f =0$ and $u_0\neq
    0$. Notice that the existence result that we are going to present in this situation is not optimal and takes advantage of the existence
    result proved in Theorem \ref{existence-nancy}{, we }refer to the remark below for further comments.

    Consider the {Problem}
    \begin{equation}\label{exis_prb2}
        \left\{
        \begin{array}{llll}
            u_t+(-\Delta)^{s}u&=& |(-\Delta)^{\frac{s}{2}}u|^q& \text{in}\quad \Omega_T,\vspace{0.2 cm}\\
            u(x,t)&=&0& \text{in}\quad (\mathbb{R}^N\setminus \Omega)\times (0,T),\vspace{0.2cm}\\
            u(x,0)&=&u_0 (x)& \text{in} \quad \Omega.
        \end{array}
        \right.
    \end{equation}
    We have the next existence result.
    \begin{Theorem}\label{initialc}
        Suppose $u_0\in L^{\sigma}(\O)$ with $1\le \s<\frac{N}{(1-3s)_+}$. Assume that $1\leq q <\frac{N+ 2s\s}{N+\s s}$. Then, there exists $T^*>0$
        such that Problem \eqref{exis_prb2} has a weak solution $u$ such that
        \begin{enumerate}
            \item if $s>\frac 13$, then  $u \in L^{\theta_1}(0,T^*;
            \mathbb{L}^{s,\theta_1}_0(\Omega))$ for all
            $\theta_1<\min\left\{\frac{q(N+\s s)}{q(N+\s s)-\s
                s},\frac{\s(N+2s)}{N+\s
                s}\right\}$.
            \item if $s\in (\frac 14,\frac 13]$, then $u\in L^{\theta_1}(0,T^*; \mathbb{L}^{s,\theta_1}_0(\Omega))$
            with $\theta_1<\min\left\{\frac{q(N+\s s)}{q(N+\s s)-\s(4
                s-1)},\frac{\s(N+2s)}{N+\s
                s}\right\}$.
        \end{enumerate}
    \end{Theorem}

    \begin{proof}
        To simplify our presentation, we give the proof in the case $s\in
        (\frac 14,\frac 13)$. Let $\phi$ be the solution to the {Problem}
        \begin{equation} 
            \left\{
            \begin{array}{rcll}
                \phi_t+(-\Delta)^{s} \phi &= & 0 & \text{ in }\quad \Omega_T, \\
                \phi(x,t) &=& 0 &\hbox{ in }\quad (\mathbb{R}^N\setminus\Omega)\times (0,T),\\
                \phi(x,0) &=& u_0(x) &\hbox { in}\quad \, \,\Omega.\\
            \end{array}%
            \right.
        \end{equation}
        According to Theorem \ref{globalu_0_t} and Corollary \ref{coryy}, we have
        $\phi\in L^p( 0,T; \,\, \mathbb{L}_0^{s,p}(\Omega))$ for all
        $p\leq \widehat{p_s}:=\frac {\s(N+2s)}{N+\s s}$.

        Let $\widetilde{u}=u-\phi$, then $\widetilde{u}$ solves the
    {Problem}
        \begin{equation}\label{S011A1}
            \left\{
            \begin{array}{rcllll}
                {\widetilde{u}}_t+(-\Delta)^{s} \widetilde{u}&= & |(-\Delta)^{\frac s2}{\widetilde{u}}+
                (-\Delta)^{\frac s2}\phi|^{q} & \text{ in }&\Omega_T, \\
                \widetilde{u}(x,t)&=& 0 &\hbox{ in }& (\mathbb{R}^N\setminus\Omega)\times (0,T),\\
                \widetilde{u}(x,0) &=&0 &\hbox{ in}& \Omega.
            \end{array}%
            \right.
        \end{equation}
        Thus, to prove the existence of solutions to Problem
        \eqref{exis_prb2}, it is enough to show that the problem
        \eqref{S011A1} has a positive solution
        $\widetilde{u}$.\\

        Let $1<r<\frac{{\s}(N+2s)}{N+{\s}s}$ be fixed, close to
        $\frac{{\s}(N+2s)}{N+{\s}s}${, we }define the set
        \begin{equation*}
            F:=\Big\{\varphi\in L^1(0,T;\mathbb{L}^{s,1+\eta}_0(\O)) \,; \, \varphi\in L^r(0,\,T;\, \mathbb{L}^{s,r}_0(\O))\mbox{ and } ||\varphi||_{L^ r(0,\,T;\, \mathbb{L}^{s,r}_0(\O))}\leq \Lambda \Big\},
        \end{equation*}
        with $0<\eta <q-1$. Consider the operator
        $$\mathcal{L}: F\longrightarrow L^1(0,T;\mathbb{L}^{s,1+\eta}_0(\O)) \ \ $$
        $$\varphi \longmapsto \mathcal{L}(\varphi)=\widetilde{u},$$
        with $\widetilde{u}$ is the unique solution of the {Problem}

        { \begin{equation}\label{aux022}
                \left\{
                \begin{array}{rcllll}
                    {\widetilde{u}}_t+(-\Delta)^{s_1} \widetilde{u}&= & |(-\Delta)^{\frac s2} {\varphi}+(-\Delta)^ {\frac s2} \phi|^{q} &
                    \text{ in }&{\Omega_T}, \\
                    \widetilde{u}(x,t) &=& 0 &\hbox{ in }& (\mathbb{R}^N\setminus\Omega)\times (0,T),\\
                    \widetilde{u}(x,0) &=&0 &\hbox{ in } &\Omega.
                \end{array}%
                \right.
        \end{equation}}
        Notice that
        $$
        |(-\Delta)^{\frac s2} {\varphi}+(-\Delta)^ {\frac s2} \phi|^{q}
        \le C(q) (|(-\Delta)^{\frac s2} {\varphi}|^q+|(-\Delta)^{\frac s2}
        \phi|^{q}).
        $$
        Define $f= C(q)|(-\Delta)^{\frac s2} {\phi}|^{q}$, then $f\in L^{m}(\O_T)$ for all
        $m<\frac{{\s}(N+2s)}{q(N+{\s}s)}$.\\
        Going back to Theorem \ref{existence-nancy}, using the fact that $\frac{{\s}(N+2s)}{q(N+{\s}s)}<\frac{(N+2s)}{s}$, then,
        according to Theorem \ref{existence-nancy}, we get the existence of a positive solution $\tilde{u}$ to Problem \eqref{aux022} such that
        $\tilde{u}\in L^{\gamma}(0,T^*; \mathbb{L}^{s,\gamma}_0(\Omega))$ for all $1<\gamma<\frac{q(N+\s s)}{q(N+\s s)-\s
            s}$. {Since} $u=\widetilde{u}+\phi$, then $u\in L^{\theta_1}(0,T^*; \mathbb{L}^{s,\theta_1}_0(\Omega))$
        with $\theta_1<\min\left\{\frac{q(N+\s s)}{q(N+\s s)-\s
            s},\frac{\s(N+2s)}{N+\s
            s}\right\}$. It is clear that $u$ solves the Problem \eqref{exis_prb2}.{ Hence,  }we conclude.
    \end{proof}

    \begin{remarks}
        The above argument used to prove Theorem \ref{initialc} forces us
        to limit ourselves to the case where $q<\frac{N+2\s s}{N+\s s}$,
        that is used to show that $m\ge 1$. In a forthcoming work
        \cite{ABH1}, we
        will consider the general case working in weighted spaces of the
        form
{   $$
        \mathbb{X}^{s,p,\g}_0(\O\times (0,\infty)):=\{\phi(t,.)\in
        \mathbb{L}^{s,p}_0(\O)\:\:  \mbox{ for {\textit{a.e.}} }t\in (0,\infty) \mbox{ and
        }\sup_{t\ge 0}t^\g||\phi(.,t)||_{\mathbb{L}^{s,p}_0(\O)}<\infty\}.
        $$}
    \end{remarks}

    \section{Extensions and further results}\label{open}


    In this section, we give some extensions of our
    regularity-existence results to other nonlinear problem with
    different structures.

    $\bullet$ In the case where $s>\frac 12$, the above
    regularity result can be extended to the case where the fractional
    Laplacian is perturbed by a drift term. More precisely, we define
    the operator
    $$
    L(w)=(-\Delta)^{s} w+B(x,t).\n w, $$ where $B$ is a vector fields
    satisfies the condition
    \begin{equation}\label{QQQ}
\int_0^T\bigg(\io |B(x,t)|^\theta
dx\bigg)^{\rho}{\theta}dt<\infty,
\end{equation}
with $1\le \theta, \rho\le \infty$, verified
$\frac{N}{2s}\frac{1}{\theta}+\frac{1}{\rho}<1-\frac{1}{2s}$, see
example 3, page 335 in \cite{JSK}. We refer also to Definition 1
in \cite{JSK} for more general condition related to Kato class of
functions.

Notice that if $|B|\in L^{\s}(\O_T)$ with $\s>\frac{N+2s}{2s-1}$,
then the condition \eqref{QQQ} holds.

    Denote by {$\widetilde{P}_\O$}, the Dirichlet heart kernel
    associate to $L$, then according to \cite{BJak}, \cite{CZ} and
    \cite{JSK}, we obtain that {$\widetilde{P}_\O\simeq P_\O$} and {$|\n_x
    \widetilde{P}_\O(x,y,t)|\le C\frac{P_\O(x,y,t)}{\d(x)\wedge
        t^{\frac{1}{2s}}}$}. Thus, if $\widetilde{w}$ is the solution to
{ the Problem}
    \begin{equation}\label{Drift1}
        \left\{
        \begin{array}{rcllll}
            \widetilde{w}_t+L(\widetilde{w})&= h & \text{ in }&\Omega_T, \\
            \widetilde{w}(x,t) &=& 0 &\hbox{ in }& (\mathbb{R}^N\setminus\Omega)\times (0,T),\\
            \widetilde{w}(x,0) &=&w_0(x) &\hbox{ in } &\Omega.
        \end{array}%
        \right.
    \end{equation}
{Then,} the regularity of $\widetilde{w}$ is the same as the
    regularity of $w$, the solution to Problem $(FHE)$ obtained in
Section \ref{Regularity_heat_Equation_Section}.

    \

    $\bullet$ Related to the Problem ${(KPZ_f)} $, we can replace the term
    $|(-\D)^{\frac{s}{2}} u|^q$ by $|\n^s u|^q$ defined in
    \eqref{frac-g}. Namely, the same existence result holds for the
{    Problem}
    \begin{equation*}
        \left\{
        \begin{array}{llll}
            u_t+(-\Delta)^{s}u&=& |\n^s u|^q+ f& \text{in}\quad \Omega_T,\vspace{0.2 cm}\\
            u(x,t)&=&0& \text{in}\quad (\mathbb{R}^N\setminus \Omega)\times (0,T),\vspace{0.2cm}\\
            u(x,0)&=&u_0 (x)& \text{in} \quad \Omega.
        \end{array}
        \right.
    \end{equation*}
    under the same {conditions on } the parameters.

 $\bullet$ We can also use the same approach to treat the case
    where the fractional gradient term appears as an absorption
    term. Namely, we can consider the Problem

{\begin{equation*}
            \left\{
            \begin{array}{llll}
                u_t+(-\Delta)^{s}u+|(-\Delta)^s u|^q &=&  f& \text{in}\quad \Omega_T,\vspace{0.2 cm}\\
                u(x,t)&=&0& \text{in}\quad (\mathbb{R}^N\setminus \Omega)\times (0,T),\vspace{0.2cm}\\
                u(x,0)&=&u_0 (x)& \text{in} \quad \Omega.
            \end{array}
            \right.
    \end{equation*}}

    The
    existence is proved for small time {values}.

    \noindent Notice that, as observed in \cite{AOP}, in the case of the
    gradient appears as an absorption term, it is possible to show the
    existence for $L^1$ data if $q<2s$. The question of existence of a solution for $L^1$ datum is still open
    for $q>2s$ including for the local case.

    $\bullet$ The question of uniqueness seems to be more difficult.
    In the case where the fractional gradient is substituted by the
    local gradient and $s>\frac 12$, the uniqueness of the good solution is proved in
    \cite{AB01}. The situation is more delicate in our case, since the
    comparison principle for the equation
    $$
    \left\{
    \begin{array}{rcllll}
        v_t+(-\D)^sv &= & C|(-\D)^{\frac{s}{2}}v| &\text{ in }&{\Omega_T}, \\
        v(x,t) &=& 0 &\hbox{ in }& (\mathbb{R}^N\setminus\Omega)\times (0,T),\\
        v(x,0) &=&0 &\hbox{ in } &\Omega,
    \end{array}%
    \right.
    $$
    seems to be not true.

    \section{Appendix}\label{Appendix_Section}
    In this {section},  we prove some technical tools that we have used
    previously throughout the paper.

    \vspace{0.2cm}

    {\bf Proof of Theorem \ref{first_integrals_regularity}.}
    Recall that $$
    G_{\l}(x,t)=\io \frac{g(y)}{(t^{\frac{
                1}{2s}}+|x-y|)^{N+\l}}dy.
    $$
    To get the regularity of $G_{\l}${,  we }will use a duality
    argument. Let $\phi\in \mathcal{C}^\infty_0(\O)$,
    then for $p\ge 1$, we have
    \begin{equation*}
        \begin{array}{lll}
            ||G_{\l}(.,t)||_{L^p(\Omega)}&=&\dyle\sup\limits_{\{|| \phi||_{L^{p'}(\Omega)}\leq 1\}}{\int_{\O}}  G_{\l}(x,t)\phi(x) dx\leq
            \dyle\sup\limits_ {\{|| \phi||_{L^{p'}(\Omega)}\leq 1\}}\int_{\O} |\phi(x)|\int_{\Omega}|g(y)|{H}(x-y,t)dy  dx,\\
        \end{array}
    \end{equation*}
    where $H(x,t):=\dfrac{1}{(t^{\frac{1}{2s}}+\rvert
        x\rvert)^{N+\l}}$. Using Young's inequality, it holds that

    \begin{equation}\label{firsti1}
        \begin{array}{lll}
            ||G_{\l}(.,t)||_{L^p(\Omega)}&\leq& \sup\limits_{\{|| \phi|| _{L^{p'}(\Omega)\leq 1}\}} || \phi ||_{L^{p'}(\Omega)}||g||_{L^{m}(\Omega)}
            || H(.,t)||_{L^{a}(\Omega)},
        \end{array}
    \end{equation}
    {where $\dfrac{1}{a}+\dfrac{1}{p'}+\dfrac{1}{m}=2$}.
    Then by a direct computation, we have {\begin{eqnarray*}
            \io H^a(x,t) dx &=& \io\dfrac{1}{(t^{\frac{1}{2s}}+\rvert
                x\rvert)^{(N+\l)a}}dx=t^{-a\frac{N+\l}{2s}}\io \dfrac{1}{(1+\frac{|x|}{t^{\frac{1}{2s}}})^{(N+\l)a}}dx.\\
    \end{eqnarray*}}
    Setting $z=\frac{x}{t^{\frac{1}{2s}}}$, it holds that
    $$
    \io H^a(x,t) dx \le t^{\frac{N}{2s}-a\frac{N+\l}{2s}} \int_{\mathbb{R}^N}
    \dfrac{1}{(1+|z|)^{(N+\l)a}}dz.
    $$
    The last integral is finite if and only if $(N+\l)a>N$. Thus $\frac{1}{p}-\frac{1}{m}<\frac{\l}{N}$.

    Now, if $\l\ge 0$, then $\frac{1}{p}-\frac{1}{m}<\frac{\l}{N}$
    for all $p>m$. Thus from \eqref{firsti1}, we get
    $$
    ||G_{\l}(.,t)||_{L^p(\O)}\le  C
    t^{\frac{-\l}{2s}-\frac{N}{2s}(\frac
        1m-\frac{1}{p})}||g||_{L^m(\O)}.
    $$
    Suppose that $\l<0$, then
    $$
    G_{\l}(x,t)\le C(t^{-\frac{\l}{2s}}+1)\io
    \frac{g(y)}{(t^{\frac{1}{2s}}+|x-y|)^{N}}dy.
    $$
    Now, repeating the same duality argument as above for the last
    term, it holds that for all $p>m$, we have
    \begin{equation}\label{YY1}
        ||G_{\l}(.,t)||_{L^p(\O)}\le  C(t^{-\frac{\l}{2s}}+1)
        t^{-\frac{N}{2s}(\frac
            1m-\frac{1}{p)}}||g||_{L^m(\O)}.
    \end{equation}

    Notice that, if $-N<\l<0$, we can improve estimate \eqref{YY1}.
    Namely for $m_0$ fixed such that $1\le
    m_0<\min\{m,-\frac{N}{\l}\}$, then $g\in L^{m_0}(\O)$ and
    $-\frac{1}{m_0}<\frac{\l}{N}$. Thus, fixed $p$ such that
    $p>\frac{m_0N}{N+m_0\l}$, then $p>m_0$ and
    $\frac{1}{p}-\frac{1}{m_0}<\frac{\l}{N}$. Repeating the same
    computations as in the case $\l>0$, with $m_0$ instead of
    $m$, we obtain that
    $$
    ||G_{\l}(.,t)||_{L^p(\O)}\le  C
    t^{\frac{-\l}{2s}-\frac{N}{2s}(\frac
        1{m_0}-\frac{1}{p})}||g||_{L^{m_0}(\O)}
    \le  C(\O)t^{\frac{-\l}{2s}-\frac{N}{2s}(\frac
        1{m_0}-\frac{1}{p})}||g||_{L^{m}(\O)}.
    $$
    \cqd

    \

    {\bf Proof of Theorem \ref{integrals regularityII}.}
    Recall that
    $$
    G_{\a,\l}(x,t)=\int_0^t\io
    \frac{g(y,\tau)(t-\tau)^{\a}}{((t-\tau)^{\frac{1}{2s}}+|x-y|)^{N+\l}}
    dy d\tau.
    $$
    We give a detailed proof in the case where  $p\le
    \frac{N+2s}{2s(\a+1)-\l}$. The other cases follow in the same way.

    Recall that $\l<2s(\a+1)$. We will use a duality argument.
    Let $\phi\in \mathcal{C}^\infty_0(\O_T)$, then
    \begin{equation*}
        \begin{array}{lll}
            ||G_{\a,\l}||_{L^{\ell}(\Omega_T)}&=&\dyle\sup\limits_{{|| \phi|| _{L^{\ell'}(\Omega_T)}\leq 1}}{\iint_{\O_T}}  G_{\a,\l}(x,t)\phi(x,t) dx dt,\vspace{0.2cm}\\
            &\leq &\dyle\sup\limits_ {{||\phi||_{L^{\ell'}(\Omega_T)}\leq 1}}\iint_{\O_T} |\phi(x,t)|\int_{0}^{t}\int_{\Omega}|g(y,\tau)|{H}(x-y,t-\tau)dy d\tau dx dt,\\
        \end{array}
    \end{equation*}
    where
    \begin{equation*}
        {H}(x,t):=\dfrac{t^\a}{(t^{\frac{1}{2s}}+\rvert x\rvert)^{N+\l}}.
    \end{equation*}
    Using Young's inequality, it holds that
    \begin{equation}\label{hjj}
        \begin{array}{lll}
            ||G_{\a,\l}||_{L^{\ell}(\Omega_T)}&\leq& \sup\limits_{{|| \phi|| _{L^{\ell'}(\Omega_T)\leq 1}}}\Int_{0}^{T}\Int_{0}^{t}
            || \phi(.,t) || _{L^{\ell'}(\Omega)}|| g(.,\tau)||_{L^{p}(\Omega)} || H(.,t-\tau)|| _{L^{a}(\Omega)}d\tau
            dt,
        \end{array}
    \end{equation}
    where $\dfrac{1}{a}+\dfrac{1}{p}+\dfrac{1}{\ell'}=2$.

    Notice that
    $$
    || H(.,t)||_{L^{a}(\Omega)}\leq
    t^{\a+\frac{N}{2sa}-\frac{N+\l}{2s}}\bigg(\Int_{\ren}\frac{dz}{(1+|z|)^{(N+\l)a}}\bigg)^{\frac{1}{a}}.
    $$
    The last integral is finite if $a(N+\l)>N$ which implies that
    \begin{equation}\label{ell}
        \frac{1}{\ell}-\frac{1}{p}<\frac{\l}{N}. \end{equation} Hence
    assuming the above hypothesis, we obtain that
    \begin{equation}\label{HH0}
        || H(.,t)||_{L^{a}(\Omega)}\leq
        Ct^{\a+\frac{N}{2sa}-\frac{N+\l}{2s}}.
    \end{equation}
    If $\l\ge 0$, then \eqref{ell} holds for all $\ell>p$. Let
    $\gamma=\a+\frac{N}{2sa}-\frac{N+\l}{2s}$. Since $2s\a<N+\l$, then
    $\g<0$. Going back to \eqref{hjj}, we obtain
    \begin{eqnarray*}
        ||G_{\a,\l}||_{L^{\ell}(\Omega_T)} &\leq & \sup_{{|| \phi|| _{L^{\ell'}(\Omega_T)\leq 1}}}
        \Int_{0}^{T}\Int_{0}^{T}|| \phi(.,t) || _{L^{\ell'}(\Omega)}|| g(.,\tau)||_{L^{p}(\Omega)}|t-\tau|^{\gamma}d\tau dt.
    \end{eqnarray*}
    Using again Young's inequality in the time, we reach that
    \begin{equation}\label{ppp}
        \begin{array}{lll}
            || G_{\a,\l}||_{L^{\ell}(\Omega_T)}
            &\leq&C\sup\limits_{{|| \phi|| _{L^{\ell'}(\Omega_T)\leq 1}}} || g||_{L^{p}(\Omega_T)}      || \phi || _{L^{\ell'}(\Omega_T)} \left(\Int_{0}^{T}t^{\bar{a}\gamma}d\tau\right )^{\frac{1}{\bar{a}}},\vspace{0.2cm}\\
        \end{array}
    \end{equation}
    with $\dfrac{1}{\bar{a}}+\dfrac{1}{p}+\dfrac{1}{\ell'}=2$. It is clear that
    $\bar{a}=a=\frac{p'\ell}{p'+\ell}$ and $\g=\frac{2s\a-\l}{2s}-\frac{N}{2s}(\frac 1p-\frac{1}{\ell})$. \\
    Now by \eqref{ppp}, it follows that the last integral is finite if
    $a|\g|<1$. Taking into consideration the values of $a$ and $\g$,
    the above condition holds if
    $\frac{1}{p}-\frac{1}{\ell}<\frac{2s(\a+1)-\l}{N+2s}$. Hence for
    all $\ell$ such that
    $p<\ell<\frac{p(N+2s)}{(N+2s-p(2s(\a+1)-\l))_+}$, we have
    \begin{equation}\label{UUT0}
        || G_{\a,\l}||_{L^{\ell}(\Omega_T)}
        \leq CT^{\g+\frac{1}{a}}||g||_{L^{p}(\Omega_T)}=C T^{\frac{2s(\a+1)-\l}{2s}-\frac{N+2s}{2s}(\frac{1}{p}-\frac{1}{\ell})}||g||_{L^{p}(\O_T)}.
    \end{equation}
    Now, if $\l<0$, then as in the proof of Theorem
    \ref{first_integrals_regularity}, since $\O$ is a bounded domain,
    we get
    $$
    G_{\a,\l}(x,t)\le C(T^{-\frac{\l}{2s}}+1)\int_0^t\io
    \frac{g(y,\tau)(t-\tau)^{\a}}{((t-\tau)^{\frac{1}{2s}}+|x-y|)^{N}}
    dy d\tau. $$ Following the same argument as in the case $\l\ge0$,
    we deduce that for all $\ell$ such that
    $p<\ell<\frac{p(N+2s)}{(N+2s-p(2s(\a+1)))_+}$, we have
    \begin{equation}\label{UUT1}
        || G_{\a,\l}||_{L^{\ell}(\Omega_T)}
        \leq C(T^{-\frac{\l}{2s}}+1)T^{\a+1-\frac{N+2s}{2s}(\frac{1}{p}-\frac{1}{\ell})}||g||_{L^{p}(\O_T)}.
    \end{equation}

    In the case where $\l<0$, we can improve estimate \eqref{UUT1}
    under additional hypothesis. More precisely, assume that $\l<0$
    and $N(\a+1)>|\l|$. We fix $p_0$ such that $1\le
    p_0<\min\{p,-\frac{N}{\l}\}$, then $g\in L^{p_0}(\O_T)$. Since
    $-\frac{1}{p_0}<\frac{\l}{N}$, then for
    $\ell>\frac{p_0N}{N+p_0\l}$, the condition \eqref{ell} is verified
    for the pair $(p_0,\ell)$.

    Notice that
    $\frac{p_0N}{N+p_0\l}<\frac{p_0(N+2s)}{(N+2s-p_0(2s(\a+1)-\l))_+}$.
    Thus, for all $\ell$ such that
    $\frac{p_0N}{N+p_0\l}<\ell<\frac{p_0(N+2s)}{(N+2s-p_0(2s(\a+1)-\l))_+}$,
    we have
    \begin{equation*}
        || G_{\a,\l}||_{L^{\ell}(\Omega_T)}
        \leq
        C(\O)T^{\frac{2s(\a+1)-\l}{2s}-\frac{N+2s}{2s}(\frac{1}{p_0}-\frac{1}{\ell})}
        ||g||_{L^{p_0}(\Omega_T)}\le
        C(\O)T^{\frac{1}{p_0}-\frac{1}{p}+\frac{2s(\a+1)-\l}{2s}-\frac{N+2s}{2s}(\frac{1}{p_0}-\frac{1}{\ell})}
        ||g||_{L^{p}(\Omega_T)}.
    \end{equation*}
    \

    If $N(\a+1)\le |\l|$, then $-1<\a<0$. Choosing $a>0$ such that
    $a(N+\l)<N$, then $\frac{1}{p}-\frac{1}{\ell}<-\frac{\l}{N}$.
    Hence
    $$
    \io H^a(x,t) dx \le t^{\a a}\io\dfrac{1}{|x|^{(N+\l)a}}dx\le
    C(\O)t^{\a a}.$$ Thus the integral in time converges if
    $\a>\-\frac{1}{a}$. Hence $\frac{1}{p}-\frac{1}{\ell}<1+\a$. Then
    for $\ell$ fixed such that $\ell<\min\{\frac{pN}{N+\l
        p},\frac{p}{(1-p(\a+1))_+}\}$, we have
    \begin{equation*}
        || G_{\a,\l}||_{L^{\ell}(\Omega_T)}
        \leq C(\O)T^{(\a+1)-(\frac{1}{p}-\frac{1}{\ell})}
        ||g||_{L^{p}(\Omega_T)}.
    \end{equation*}
    \cqd

    \begin{remarks}\label{lastr}

        \

        If $\l>2s(\a+1)$, then $G_{\a,\l}$ is not well defined in the
        sense that the integral $\dyle\int_0^t\io
        \frac{g(y,\tau)(t-\tau)^{\a}}{((t-\tau)^{\frac{1}{2s}}+|x-y|)^{N+\l}}
        dy d\tau$ can diverge including for regular data. To see that, we
        consider the case where $g=C$, then choosing $x\in \O$ such that
        $B_r(x)\subset \O$, we have
        \begin{eqnarray*}
            G_{\a,\l}(x,t)&\ge &\int_0^t\int_{B_r(x)}
            \frac{(t-\tau)^{\a}}{((t-\tau)^{\frac{1}{2s}}+|x-y|)^{N+\l}} dy
            d\tau=\int_0^t\int_{B_r(x)}
            \frac{\s^{\a}}{(\s^{\frac{1}{2s}}+|x-y|)^{N+\l}} dy d\s\\
            &=&\int_0^t \s^{\a-\frac{N+\l}{2s}}\int_{B_r(x)}
            \frac{dy}{(1+\frac{|x-y|}{\s^{\frac{1}{2s}}})^{N+\l}}dy d\s.
        \end{eqnarray*}
        Setting $z=\frac{(x-y)}{\s^{\frac{1}{2s}}}$, then
        \begin{eqnarray*}
            G_{\a,\l}(x,t)&\ge &\int_0^t
            \s^{\a-\frac{N+\l}{2s}+\frac{N}{2s}}\int_{B_{\frac{r}{\s^{\frac{
                                1}{2s}}}}(0)} \frac{dz}{(1+|z|)^{N+\l}}dz d\s.
        \end{eqnarray*}
        Now, for $t<<1$, we have $\s\le t<<1$, then
        $$
        \int_{B_{\frac{r}{\s^{\frac{1}{2s}}}}(0)}
        \frac{dz}{(1+|z|)^{N+\l}}dz\ge \int_{B_{r}(0)}
        \frac{dz}{(1+|z|)^{N+\l}}dz=C.
        $$
        Thus
        $$
        G_{\a,\l}(x,t)\ge C\int_0^t
        \s^{\a-\frac{N+\l}{2s}+\frac{N}{2s}}d\s=\infty.
        $$
    \end{remarks}

\end{document}